\documentclass[12pt]{article}
\pdfoutput=1
\usepackage[english]{babel}
\usepackage{a4wide, float, lscape}
\usepackage{amssymb, amsmath, verbatim, MnSymbol, marvosym, units, array}%
\usepackage{color,textcomp,graphicx, rotating}
\usepackage{sseq,ifthen,pdfpages}
\usepackage[all,poly]{xy}
\usepackage{url}
\newtheorem{theorem}{Theorem}[section]
\newtheorem{lemma}[theorem]{Lemma}

\newenvironment{proof}{\begin{trivlist}
  \item[\hskip \labelsep \emph{Proof:}]} {\hfill $\Box$ \end{trivlist}}

\newcommand{\mbb}{\mathbb}

\newcommand{\pt}{\mathit{pt}}
\newcommand{\tmf}{\mathit{tmf}}
\newcommand{\Tmf}{\mathit{Tmf}}
\newcommand{\TMF}{\mathit{TMF}}
\newcommand{\Sq}{\mathit{Sq}}
\newcommand{\mell}{\overline{\mathcal{M}}_{ell}}
\newcommand{\weier}{\overline{\mathcal{M}}_{ell}^{\,_+}}
\newcommand{\Spf}{\mathit{Spf}}
\newcommand{\longtwoheadrightarrow}{\ensuremath{\relbar\joinrel\twoheadrightarrow}}
\restylefloat{figure}

\newcommand{\gray}{\renewcommand{\ssconncolor}{sseqgr} \renewcommand{\ssplacecolor}{sseqgr}}
\newcommand{\black}{\renewcommand{\ssconncolor}{ssblack} \renewcommand{\ssplacecolor}{ssblack}}

\newcommand{\res}[1]{\resizebox{\linewidth}{!}
{#1}
\vspace{-.7cm}}

\newcommand{\ret}[4]{
\begin{figure}[H]
\resizebox{\linewidth}{!}{
\begin{sseq}{#1}{#2}
#3
\end{sseq}
}
\vspace{-.7cm}
\caption{#4}
\end{figure}\vspace{-.7cm}
}

\newcommand{\alphamult}{
\ssline 3 1
\ssdropbull
\ssmove 7 1
}

\newcommand{\firstrightatthree}{
\gray
\ssarrow {-1} 5
\ssmove 1 {-5}
\ssline 3 1
\black \ssdropbull \ssline 3 5
\ssmove {4} {-4}
}

\newcommand{\threestep}{
\gray \ssarrow {-1} 9
\ssdropbull \ssline 3 1
\ssdropbull \ssline[arrowfrom] 1 {-5}
\ssdropbull \ssline 3 1
\ssdropbull \ssline[arrowfrom] 1 {-5}
\ssdropbull \ssline 3 1
\ssdropbull \black
}

\newcommand{\movefour}{
\black
\ssmove 4 {0}
\ssdrop{\square}
}

\newcommand{\minfour}{
\black
\ssmove {-4} 0
\ssdrop{\square}
}

\newcommand{\cmult}{
\ssline 3 1 \ssdropbull
\ssline 3 1 \ssdropbull
\ssline 3 1 \ssdropbull
\ssmove {-4} {-2} \ssdropbull \ssline 1 1 \ssdropbull \ssarrow 1 1\ssdrop{}\ssmove {-1} {-1}
\ssmove 2 {0} \ssdropbull
\ssline 1 1 \ssmove {-6}{-2}
\ssdropextension \ssmove 5 1}

\newcommand{\ccmult}{
\ssmove 7 1 \ssdropbull
\ssline {-1}{-1} \ssdropbull
\ssline 3 1 \ssdropbull
\ssmove 3 1 \ssdropbull\cmult \ssmove{-8}{-2} \ssline 1 1 \ssdropbull \ssarrow 1 1\ssdrop{}\ssmove {-2} {-2}
\ssdropextension \ssdropextension
\ssmove {-3} {-1} \ssstroke
\ssmove {11} 3}

\newcommand{\xdy}{
\ssline 3 1 \ssdropbull
\ssline 3 1 \ssdropbull
\ssline 3 1 \ssdropbull
\ssline {-1}{-1} \ssdropbull
\ssline {-1}{-1} \ssdropbull
\ssline 3 1 \ssdropbull
\ssline {-1}{-1} \ssdropbull
\ssmove 6 2 \ssdropbull
\ssline {-1}{-1} \ssdropbull
\ssline 3 1 \ssdropbull
\ssline {-1}{-1} \ssdropbull
\ssline {-1}{-1} \ssdropbull
\ssline 3 1 \ssdropbull
\ssline 3 1 \ssdropbull}

\newcommand{\minxdy}{
\ssmove 3 1 \ssdropbull
\ssline 3 1 \ssdropbull
\ssline 3 1 \ssdropbull
\ssline {-1}{-1} \ssdropbull
\ssline {-1}{-1} \ssdropbull
\ssline 3 1 \ssdropbull
\ssline {-1}{-1} \ssdropbull
\ssmove 6 2 \ssdropbull
\ssline {-1}{-1} \ssdropbull
\ssline 3 1 \ssdropbull
\ssline {-1}{-1} \ssdropbull
\ssline {-1}{-1} \ssdropbull
\ssline 3 1 \ssdropbull
\ssline 3 1 \ssdropbull}

\def\htower#1{\expandafter\nlineii\romannumeral\number\number #1 000\relax}
\def\nlineii#1{\if#1m\expandafter\theline\expandafter\nlineii\fi}
\def\theline{\ssline 1 1\ssdropbull}

\title{The homotopy groups of the spectrum $\Tmf$}
\author{Johan Konter}
\date{October 31, 2012}

\begin{document}

\newcommand{\coker}{\mathop{\mathgroup\symoperators coker}\nolimits}
\newcommand{\im}{\mathop{\mathgroup\symoperators im}\nolimits}
\newcommand{\Aut}{\mathop{\mathgroup\symoperators Aut}\nolimits}
\newcommand{\ord}{\mathop{\mathgroup\symoperators ord}\nolimits}
\newcommand{\PGL}{\mathop{\mathgroup\symoperators PGL}\nolimits}
\newcommand{\Ext}{\mathop{\mathgroup\symoperators Ext}\nolimits}
\newcommand{\Tor}{\mathop{\mathgroup\symoperators Tor}\nolimits}
\newcommand{\Hom}{\mathop{\mathgroup\symoperators Hom}\nolimits}
\renewcommand{\top}{\mathop{\mathgroup\symoperators top}\nolimits}
\newcommand{\Spec}{\mathop{\mathgroup\symoperators Spec}\nolimits}

\def\sseqpacking{\sspackdiagonal}

\maketitle

\begin{abstract}
We use the structure of the homotopy groups of the connective spectrum $\tmf$ of topological modular forms and the elliptic and Adams-Novikov spectral sequences to compute the homotopy groups of the non-connective version $\Tmf$ of that spectrum. This is done separately for the localizations at 2, 3 and higher primes.
\end{abstract}


\section{Introduction}



The complete computation of the homotopy groups of spheres appears to be an untractable problem. 
An approximation to the stable homotopy groups of spheres can be given by the cohomology of the Steenrod algebra $A$. More precisely the Adams spectral sequence with $\Ext_A(\mbb F_p,\mbb F_p)$ as $E_2$-term abuts to the $p$-component of the stable homotopy groups of spheres. There is a sequence of ring spectra characterized by the the fact that their mod 2 cohomology is a quotient of the Steenrod algebra $A$. The first spectrum is the Eilenberg-MacLane spectrum $H\mbb Z/2$ and its cohomology is the whole Steenrod algebra. Then comes $H\mbb Z$ whose cohomology is $A//\langle\Sq^1\rangle$. After that come the connective complex and real K-theory spectra $ku$ and $ko$ whose cohomologies are given by respectively $A//\langle\Sq^1,\Sq^2\Sq^1+\Sq^1\Sq^2\rangle$ and $A//\langle\Sq^1,\Sq^2\rangle$ (see \cite{ada} for details). The Adams spectral sequences whose $E_2$-term is the cohomology of the respective finite dimensional subalgebras as their $E_2$-terms, converge to the homotopy groups of the corresponding ring spectra. A bigger subalgebra gives a better approximation of $\Ext_A(\mbb F_2,\mbb F_2)$, and thus of the homotopy groups of spheres. This paper is concerned with the spectrum whose mod 2 cohomology is given by $A//\langle\Sq^1,\Sq^2, \Sq^4\rangle$. This spectrum is the spectrum of topological modular forms localized at 2 (see \cite{rez}) and denoted $\tmf$. The next quotient of $A$ in this sequence is not realizable as the mod 2 cohomology of a spectrum, so in this sense $tmf$ gives the best possible approximation of the stable homotopy groups of the spheres. 

In the remainder of this introduction, we sketch the construction of some variants of the spectrum of topological modular forms and we outline our strategy for computing their homotopy groups. 

Following Hopkins, Miller and Goerss, \cite{hop}, a presheaf of elliptic homology theories on the moduli stack of elliptic curves $\mathcal M_{ell}$ is constructed. One then uses some heavy machinery to lift this presheaf to a sheaf of $E_\infty$ ring spectra \cite{beh}.  Once we have such a sheaf, we can take its global sections over $\mathcal M_{ell}$, which gives us the \emph{periodic} spectrum of topological modular forms $\TMF$. Similarly we get the \emph{non-connective} spectrum of topological modular forms $\Tmf$ associated to the Deligne-Mumford compactification $\mell$, and a \emph{connective} spectrum of topological modular forms $\tmf$.

The paper is organized as follows. We will start with some background on elliptic curves, see \S\ref{ssec:mod}, elliptic cohomology theory, see \S\ref{ssec:spec}, and Weierstrass equations, see \S\ref{ssec:wei}. People familiar with topological modular forms might want to skip to \S\ref{ssec:hom}, where we start the actual computation of the homotopy groups. The homotopy groups $\pi_*\tmf$ of the connective spectrum of topological modular forms have been calculated by Bauer \cite{bau}. It is fairly easy to deduce $\pi_*\TMF$ from this calculation, namely by $\pi_*\TMF=\pi_*\tmf[\Delta^{-24}]$. The goal of this paper is to fill the niche that is left and compute the homotopy groups of the non-connective $\Tmf$, which is done in Theorem \ref{th:main}. The main computational tool is the use of spectral sequences based on the cohomology of $\mell$ with coefficients in some graded sheaves. 
We split the computation of $\pi_*\Tmf$ into three cases: away from the primes 2 and 3 (in Section \ref{sec:6in}), localized at 3 (in Section \ref{sec:3}) and localized at 2 (in Section \ref{sec:2}).

We assume that the reader is familiar with general algebraic geometry concepts, like sheaves, schemes, cohomology \cite{har}. We view stacks mainly as pseudo-functors of groupoids, as in \S3.1.2 of \cite{vis}.

\subsection{The moduli stack of elliptic curves}\label{ssec:mod}

An elliptic curve over a field $K$ is a genus 1 non-singular projective algebraic curve over $K$ with a marked point over $K$. An elliptic curve over a scheme $S$ is subsequently defined as a proper, flat finitely presented morphism of schemes $p:C\to S$ equipped with a section $e:S\to C$, such that every geometric fiber of $p$ is an elliptic curve in the previous sense. The marked point of such a fiber is then given by the section $e$.

To classify algebro-geometric objects, such as elliptic curves, one generally sets up a moduli problem. In our case of interest, this is the functor 
$F: {\tt Aff}^{\textrm{op}}\to {\tt Sets}$ which sends an affine scheme $\Spec(R)$ to the set of isomorphisms classes of elliptic curves over $\Spec(R)$. 
Now suppose that $F$ is representable, say by an scheme $M$. By definition of representability this means that there exists a natural isomorphism
\begin{equation}\label{eq:rep}\Hom(\Spec(R),M)\simeq F(\Spec(R))=\{\textrm{iso. classes of elliptic curves }C\to\Spec(R)\}\end{equation} for any affine scheme $\Spec(R)$. The identity element in $\Hom(M,M)$ would correspond via this natural isomorphism to a particular elliptic curve over $M$, which we call universal. However, because the the automorphism group of an elliptic curve is non-trivial, this natural isomorphism cannot exist, and $F$ is not representable by a scheme. 
The moduli problem $F: {\tt Aff}^{\textrm{op}}\to {\tt Sets}$ defined by equation (\ref{eq:rep}) is too naive. The correct way to set up things is  to look at $F$ as a stack of elliptic curves, $F: {\tt Aff}^{\textrm{op}}\to {\tt Groupoids}$ where $F(\Spec(R))$ is the \emph{groupoid} of elliptic curves over $\Spec(R)$. So we do not have a moduli scheme, but a moduli stack.

Let $\mathcal M_{ell}$ denote the moduli stack of elliptic curves over $\Spec(\mbb Z)$. Let $\mell$ denote the moduli stack of generalized elliptic curves, i.e. we allow fibers of $p:C\to S$ to be curves with a nodal singularity, but the section $e$ must be contained in the smooth locus of the fibers. This stack is called the Deligne-Mumford compactification of $\mathcal M_{ell}$. Furthermore, we let $\weier$ be the moduli stack of Weierstrass elliptic curves, i.e. we even allow fibers to have a cusp away from the section $e$. So we have the embeddings $\mathcal M_{ell}\hookrightarrow\mell\hookrightarrow\weier$ as open substacks.


\subsection{Universal elliptic cohomology theory}\label{ssec:spec}

We recall that a generalized cohomology theory is a sequence of presheaves of abelian groups $\{E^n(-)\}$ on the category of topological spaces along with coboundary homomorphisms, which satisfy the Eilenberg-Steenrod axioms. We write $E^*(-)=\Pi_nE^n(-)$ for the corresponding functor with values in graded abelian groups. By Brown representability, there exist spaces $E_n$ such that for every space $X$ we have $E^n(X)=[X,E_n]$, the set of homotopy classes of maps from $X$ to $E_n$. The Eilenberg-Steenrod axioms imply that there are natural isomorphisms $E^{n}(X)\to E^{n+1}(\Sigma X)$, which induce maps $\sigma_n:\Sigma E_n\to E_{n+1}$. We denote the resulting spectrum by $E=\{E_n,\sigma_n\}$. This spectrum also allows us to define a corresponding homology theory $\{E_n(-)\}$ and we write $E_*(-)$ for $\bigoplus_nE_n(-)$. From the definitions, it follows immediately that $\pi_*E:=E^{-*}(pt)\simeq E_*(\pt)$. 

In the nineties Landweber, Ravenel, and Stong introduced some interesting cohomology theories called elliptic cohomology theories \cite{lrs}. We call a generalized cohomology theory $E^*(-)$ even and weakly periodic if $E^*(pt)$ is concentrated in the even degrees, $E^2(pt)$ is an invertible $E^0(pt)$-module and the natural map $$E^2(pt)\otimes_{E^0(pt)} E^n(pt)\stackrel{\simeq}{\longrightarrow}E^{n+2}(pt)$$ is an isomorphism for all $n\in\mathbb Z$. An elliptic cohomology theory is defined to be a generalized cohomology theory $E^*(-)$ which is even and weakly periodic, such that there exists an elliptic curve $C$ over a ring $R$, an isomorphism $\pi_*E\simeq R$ and an isomorphism of formal groups $\Spf E(\mbb C\mbb P^\infty)\simeq\hat C$, where $\hat C$ is the formal completion of $C$ along the identity section $e: \Spec(R)\longrightarrow C$.

By definition of $\mathcal M_{ell}$ as a moduli stack, an elliptic curve $C$ over a ring $R$ is equivalent to a map $f:\Spec(R)\longrightarrow\mathcal M_{ell}$ from $\Spec(R)$ to the moduli stack. The formal completion $\hat C$ of this elliptic curve $C$ along the identity section $e: \Spec(R)\longrightarrow C$ has the structure of a formal group. 

Let us digress a little on formal groups. The zeroth space of the periodic complex cobordism spectrum $MP=\bigvee_{m\in\mbb Z}\Sigma^{2m}MU$ is denoted by $MP_0$. When a formal group $\hat C$ over $R$ admits a choice of local coordinates, 
then by Quillen's theorem \cite{ada} it corresponds to a ring homomorphism $MP_0\longrightarrow R$ \cite{lur}. 
This gives a map from the stack of formal groups $\mathcal M_{\textrm FG}$ to the stack associated to the Hopf algebroid $(MP_0, MP_0MP)$, which we denote by $\mathcal M_{(MP_0, MP_0MP)}$. 
By a theorem of Quillen, Landweber and Novikov \cite{koc}, we even know that this map is a natural isomorphism.  A formal group law $MP_0\longrightarrow R$ is called Landweber exact if the functor $R\otimes_{MP_0}-$ is exact on the category of $(MP_0,MP_0MP)$-comodules \cite{hs}. By the Landweber exact functor theorem \cite{lan}, a formal group law $MP_0\longrightarrow R$ is Landweber exact if and only if the corresponding morphism $$\Spec(R)\longrightarrow\Spec(MP_0)\longrightarrow\mathcal M_{(MP_0, MP_0MP)}\simeq\mathcal M_{\textrm FG}$$ is flat \cite{hoh}. 

Back to our case at hand, we have constructed a morphism of stacks $\mathfrak F:\mathcal M_{ell}\longrightarrow\mathcal M_{\textrm FG}$ by sending $C$ to $\hat C:MP_0 \longrightarrow R$. Hopkins and Miller showed that $\mathfrak F$ is flat \cite{hoh}. So provided that the map $f$ is flat, the composite $$\Spec(R)\stackrel{f}{\longrightarrow}\mathcal M_{ell}\stackrel{\mathfrak F}{\longrightarrow}\mathcal M_{\textrm FG}$$ will also be flat. Hence the formal group $\hat C: MP_0\longrightarrow R$ is Landweber exact and therefore the functor $$Ell_{C/R}(-)=MP_*(-)\otimes_{MP_0}R$$ is a homology theory \cite{hoh}. We define the presheaf $\mathcal O^{\hom}$ on the site of affine flat schemes over the moduli stack of elliptic curves by sending $f$ to the homology theory corresponding to the formal group belonging to that elliptic curve, $\mathcal O^{\hom}(f)=Ell_{C/R}$. So we have the contravariant functor
\begin{eqnarray*}
\mathcal O^{\hom}:&(\textrm{Aff}/\mathcal M_{ell})^\textrm{op}&\longrightarrow \textrm{Homology Theories},\\
& C/R & \longmapsto  Ell_{C/R}.
\end{eqnarray*}

So, roughly speaking, there is one elliptic cohomology for every elliptic curve. A natural question to ask is whether there is a universal elliptic cohomology theory, i.e. an initial object in the category of elliptic cohomology theories. Such a universal theory should be related to a universal elliptic curve. However, we have seen in the previous section that there is no such universal elliptic curve. Hence, in this naive sense, there is no universal elliptic cohomology theory.

Another perspective on the reason for the absence of a naive universal elliptic homology theory is the following. To build such a global object from the presheaf $\mathcal O^{\hom}$ we  would like to take global sections. However, in the site of affine flat schemes over the moduli stack of elliptic curves there is no initial object and therefore no notion of global sections. Now we might try to resolve this by taking the limit of the theories $\mathcal O^{\hom}(f)$, where $f$ runs over all \'etale maps from an affine scheme to $\mathcal M_{ell}$. Unfortunately, this naive approach fails because the category of homology theories is not complete and this limit doesn't exist.

We recall that a homology theory can be represented by a spectrum. The category of spectra is better suited for our purposes; it has 
homotopy limits. So we would like a presheaf of spectra which corresponds to the presheaf of homology theories $\mathcal O^{\hom}$. The main theorem of the approach of Hopkins, Miller and Goerss is that such a presheaf does indeed exist.
\begin{theorem} (Hopkins-Miller-Goerss) There exists a sheaf $\mathcal O^{\top}$ of $E_\infty$ ring spectra on $\left(\mathcal M_{ell}\right)_{\acute{e}t}$, the \'etale site of the moduli stack of elliptic curves, whose associated presheaf of homology theories is the presheaf $\mathcal O^{\hom}$ built using the Landweber Exact Functor Theorem.
\end{theorem}
The spectrum $\TMF$ is defined to be $\mathcal O^{\top}\left(\mathcal M_{ell}\right)$. This is the universal spectrum we have been looking for. 
Because the morphism $\mell\longrightarrow\mathcal M_{\textrm FG}$ is also flat, just like the morphism $\mathfrak F:\mathcal M_{ell}\longrightarrow\mathcal M_{\textrm FG}$, we can also construct and evaluate a sheaf $\mathcal O^{\top}$ on the stack of generalized elliptic curves. The spectrum $\Tmf$ is consequently defined as $\Tmf=\mathcal O^{\top}\left(\mell\right)$. The sensible step to make next is trying to evaluate $\mathcal O^{\top}$ on the stack of Weierstrass elliptic curves. The morphism $\weier\longrightarrow\mathcal M_{\textrm FG}$, however, is not flat, and there is no such sheaf over $\weier$. The spectrum $\tmf$ is defined as the connective cover $\tmf=\tau_{\geq0}\left(\mathcal O^{\top}\left(\mell\right)\right)$, meaning that $\pi_n\tmf=0$ for all $n<0$, and there exists a morphism of spectra $\tmf\to \mathcal O^{\top}\left(\mell\right)$ such that the induced map $\pi_n\tmf\to\pi_n\mathcal O^{\top}\left(\mell\right)$ is an isomorphism for all $n\geq0$.

Because $\mathcal O^{\top}$ is a sheaf of $E_\infty$ ring spectra, it makes sense to talk about $\pi_*\left(  \mathcal O^{\top}\left(\mell\right)\right)$. There is a spectral sequence relating this to $\left(  \pi_*\mathcal O^{\top}\right)\left(\mell\right)$ or rather to $H^*\left(\mell, \pi_*\mathcal O^{\top}\right)$, as we will see in \S\ref{ssec:hom}. We observe that $\pi_n\mathcal O^{\top}$ is a presheaf of abelian groups partially defined on $\left(\mell\right)_{\acute{e}t}$ which allows a sheafification $\pi_{n}^\dag\mathcal O^{\top}$. Because the sheaf $\mathcal O^{\top}$ is associated to the presheaf $\mathcal O^{\hom}$ we know that for an element $f:\Spec(R)\to \mell$ of the \'etale site $\left(\mell\right)_{\acute{e}t}$, the ring spectrum $\mathcal O^{\top}(f)$ is even and weakly periodic. This means that $\pi_2\mathcal O^{\top}(f)$ is a invertible $\pi_0\mathcal O^{\top}(f)$-module and $\pi_{n+2}\mathcal O^{\top}(f)\simeq\pi_n\mathcal O^{\top}(f)\otimes\pi_2\mathcal O^{\top}(f)$. Because $f$ is affine, the sheafification yields that $\pi_{n}^\dag\mathcal O^{\top}(f)\simeq\pi_n\mathcal O^{\top}(f)$. Hence $\pi_{n+2}^\dag\mathcal O^{\top}(f)\simeq\pi_n^\dag\mathcal O^{\top}(f)\otimes\pi_2^\dag\mathcal O^{\top}(f)$ and writing $\omega$ for the invertible sheaf $\pi_2^\dag\mathcal O^{\top}$ we find an isomorphism $\pi_{2t}^\dag\mathcal O^{\top}\simeq\omega^{\otimes t}$. (See \cite{beh} for a survey on the proof.)

\subsection{Weierstrass equations}\label{ssec:wei}

Any elliptic curve over a field $K$ is isomorphic to a curve in affine coordinates defined by a Weierstrass equation (see \cite{sil})
\begin{equation}\label{eq:wei}
y^2+a_1xy+a_3y=x^3+a_2x^2+a_4x+a_6.
\end{equation}
If both 2 and 3 are invertible in $K$, then we can write down the change of coordinates $x=x_1-\frac{1}{12}a_1^2-\frac13a_2$ and $y=y_1-\frac12a_1x_1+\frac{1}{24}a_1^3+\frac16a_1a_2-\frac12a_3$, which results in the Weierstrass equation
\begin{equation}\label{eq:norm}
y_1^2=x_1^3-27c_4x_1-54c_6,
\end{equation}
for some $c_4$ and $c_6$. The discriminant $\Delta:=\frac{1}{1728}(c_4^3-c_6^2)$ is nonzero iff the curve is smooth. 

Now we take a Weierstrass elliptic curve over a scheme $S$ and we look at the section $e:S\to C$, which is a relative Cartier divisor. Like in the case of elliptic curves over an algebraically closed field, we can use Riemann-Roch to show that $p_*\mathcal O(ne)$ is locally free of rank $n$. 
In particular the rank of $p_*\mathcal O(e)$ is 1 and, because $p_*\mathcal O(e)$ contains the constant sections, it contains only the constants. For $n=2$ we find that $p_*\mathcal O(2e)$ is of rank 2. Assuming it is free, we pick a second generator, call it $x$. Because $p_*\mathcal O(3e)$ is of rank three, it must contain another function besides those generated by 1 and $x$. Let us pick a third generator of $p_*\mathcal O(3e)$ and call it $y$.  Because $p_*\mathcal O(6e)$ is of rank 6, the elements of $\{1,x,y,x^2,xy,y^2,x^3\}$ cannot be linearly independent, which gives us a global version of the Weierstrass equation (\ref{eq:wei}) where $x$ and $y$ are now functions. If 2 and 3 are invertible in the base scheme $S$, then we can write this equation in the form of (\ref{eq:norm}). See \S2.2 of \cite{kat} for more details.

We are interested the groupoid of elliptic curves over a scheme $S$.  So let us look at isomorphisms of an elliptic curve. Suppose we have elliptic curves $\xymatrix@1{C \ar[r]_p & S \ar@/_/[l]_e }$ and $\xymatrix@1{C' \ar[r]_{p'} & S \ar@/_/[l]_{e'} }$ and chosen elements $x\in p_*\mathcal O(2e)$, $y\in p_*\mathcal O(3e)$, $x'\in p'_*\mathcal O(2e')$, and $y'\in p'_*\mathcal O(3e')$. Furthermore suppose we are given an isomorphism $\varphi:C\to C'$. Then we know that $p=p'\circ\varphi$ and $e'=\varphi\circ e$, so we find that $\varphi$ also induces an isomorphism of sheaves $$p'_*\mathcal O(ne')\stackrel{\simeq}{\longrightarrow}(p'\circ\varphi)_*\mathcal O(\varphi^{-1}\circ ne')=p_*\mathcal O(ne).$$
Hence $\varphi$ induces a change of bases in $p_*\mathcal O(ne)$. Because $x\in p_*\mathcal O(2e)$ and $y\in p_*\mathcal O(3e)$ such a base change is of the form $x=u^2x'+r$ and $y=u^3y'+u^2sx'+t$, for some invertible function $u$. When we substitute this in the Weierstrass equation (\ref{eq:wei}) we find the following relations between the coefficients of the two Weierstrass equations:
\begin{equation}\label{eq:aa}
\begin{aligned}
ua'_1&=a_1+2s,\\
u^2a'_2&=a_2-sa_1+3r-s^2,\\
u^3a'_3&=a_3+ra_1+2t,\\
u^4a'_4&=a_4-sa_3+2ra_2-ta_1-rsa_1-2st-3r^2,\\   
u^6a'_6&=a_6+ra_4-ta_3+r^2a_2-rta_1+r^3-t^2.
\end{aligned}
\end{equation}
When we substitute this in the Weierstrass equation (\ref{eq:norm}) we get relations between the coefficients given by
\begin{eqnarray*}
u^4c'_4&=c_4,\\
u^6c'_6&=c_6,\\
u^{12}\Delta'&=\Delta.
\end{eqnarray*}
To simplify the coming calculations, we set $u$ equal to 1. This means we only look at strict isomorphisms of elliptic curves. We capture the above groupoid structure in the Hopf algebroid $A\rightrightarrows\Gamma$, where $A=\mbb Z[a_1, a_2, a_3, a_4, a_6]$, $\Gamma=A[r,s,t]$, the left unit is given by the canonical inclusion $A\hookrightarrow\Gamma$ and the right unit by the relations in (\ref{eq:aa}). This Hopf algebroid assigns to an affine scheme $\Spec(R)$ a groupoid with $\Hom(A,R)$ as the set of objects and $\Hom(\Gamma,R)$ as the set of morphisms. This assignment is a sheaf
, but it is almost never a stack. 
Let us write $\mathcal M(A,\Gamma)$ for the stackification of this sheaf. 

Every element of $A$ determines a Weierstrass equation (\ref{eq:wei}), which in turn uniquely determines a Weierstrass elliptic curve as well as a choice of local coordinates modulo degree 5. The algebra $\Gamma$ on the other hand paramatrizes the strict isomorphisms of Weierstrass elliptic curves. Hence there is a morphism of stacks $\pi:\mathcal M(A,\Gamma)\longrightarrow\weier$.

This morphism can be used to relate the cohomology of the Hopf algebroid $A\rightrightarrows\Gamma$ with the cohomology $H^q(\mathcal M,\omega^{\otimes\frac12p})$, which we introduced at the end of the previous section. The category of $(A,\Gamma)$-comodules has enough injectives and we define $$H^{q,p}(A,\Gamma)=\Ext^q_{(A,\Gamma)-\text{comod}}(A,A[-p]),$$ where $A[-p]$ is a copy of $A$ with the grading shifted: $(A[-p])_n=A_{n+p}$. Following Example 1.14 of \cite{goe3} there is a sheaf $\mathcal F_p$ defined by $\mathcal F_p\big(\Spec(R)\to\mathcal M(A,\Gamma)\big)=R\otimes_AA[-p]$ such that we have a natural ismorphism to sheaf cohomology $$\Ext^q_{(A,\Gamma)-\text{comod}}(A,A[-p])\simeq H^q(\mathcal M(A,\Gamma),\mathcal F_p).$$
There is an obvious map of stacks $\rho:\weier\to\Spec(\mathbb Z)$. The composition $\mathcal M(A,\Gamma)\stackrel{\pi}{\longrightarrow}\weier\stackrel{\rho}{\longrightarrow}\Spec(\mathbb Z)$ allows us to write $H^q(\mathcal M(A,\Gamma),\mathcal F_p)=R^q(\rho\circ\pi)_*\mathcal F_p$. The Grothendieck spectral sequence corresponding to $(\rho\circ\pi)_*$  starts with $(R^r\rho_*\circ R^s\pi_*)\mathcal F_p$ and converges to $R^{r+s}(\rho\circ\pi)_*\mathcal F_p$. Because all fibers of $\pi$ look like $\mathbb G_m$, we know that $\pi_*$ has no derived functors. Hence the Grothendieck spectral sequence collapses and gives an isomorphism of sheaf cohomologies $$H^q(\mathcal M(A,\Gamma),\mathcal F_p)\simeq H^q\left(\weier,\omega^{\otimes\frac12 p}\right),$$ which is zero when $\frac12p\not\in\mathbb Z$. We conclude that $$H^{q,p}(A,\Gamma) \simeq H^q\left(\weier,\omega^{\otimes\frac12 p}\right).$$

One definition of integral modular forms of weight $k$ (and level 1) is as the global sections of $\omega^{\otimes k}$ over $\weier$; that is, $MF_* = H^0(\weier, \omega^{\otimes *})$ is the ring of integral modular forms. By the above isomorphism we may think of elements in this ring in terms of elements of $A$. When 2 and 3 are invertible equation (\ref{eq:norm}) is universal and the ring of $\mbb Z[\frac16]$-modular forms is the polynomial ring $\mbb Z[\frac16][c_4,c_6]$. The ring of integral modular forms is generated by the modular forms $c_4$, $c_6$, and $\Delta$ of weights 4, 6, and 12 respectively and there is an isomorphism of graded rings
$$MF_*\simeq\mbb Z[c_4, c_6, \Delta]/(c^3_4 - c^2_6 - 1728\Delta).$$
See \cite{del} for a proof. Note that $1728 = 12^3$, which is a manifestation of the fact that the primes 2 and 3 are special in this subject.

\subsection{Computation of the homotopy groups}\label{ssec:hom}

We use the following grading convention for the cohomology of the stacks: $H^{q,p}(\mathcal M)=H^q(\mathcal M,\pi_p^\dag\mathcal O^{\textrm top})=H^q(\mathcal M,\omega^{\otimes\frac12p})$ and $H^q(\mathcal M)=H^{q,*}(\mathcal M)$. We say that the elements of $H^{q,p}$ are of bidegree $(q,p)$. By elements of degree $p$ we mean the elements of $H^{*,p}$ of $\weier$ or $\mell$. The elements of $H^{q,*}(\mathcal M)=H^q(\mathcal M)$ are said to be in filtration $q$.

We may compute the homotopy groups of the spectra of topological modular forms using the convergent spectral sequences
\begin{equation} H^q\left(\mathcal M_{ell},\omega^{\otimes p}\right)\Rightarrow\pi_{2p-q}\TMF,\end{equation}
\begin{equation}\label{eq:ess} H^q\left(\mell, \omega^{\otimes p}\right)\Rightarrow\pi_{2p-q}\Tmf,\end{equation}
\begin{equation} H^q\left(\weier,\omega^{\otimes p}\right)\Rightarrow\pi_{2p-q}\tmf.\end{equation}

The convergence of the first elliptic spectral sequence, converging to $\pi_*\TMF$, is proven in detail in \cite{dou}. The second elliptic spectral sequence, converging to $\pi_*\Tmf$, follows from analogous arguments. The third follows from a $MU_*$-based Adams-Novikov spectral sequence. Its convergence may be checked by arguments given in \cite{goe2}. The edge homomorphism, which maps the $E_\infty$-sheet of a spectral sequence to the zero filtration of the $E_2$-sheet, gives a map $\pi_*\tmf\longrightarrow MF_*$. This is why the elements of these spectra are called \emph{topological modular forms}.

The homotopy groups of the spectrum $\tmf$ have been computed in \cite{bau}, \cite{nau} and \cite{rez}. From the homotopy groups of $\tmf$ we can the retrieve homotopy groups of the periodic spectrum $\TMF$ in the following way. A Weierstrass elliptic curve is non-singular if and only if the discriminant $\Delta$ of the corresponding Weierstrass equation is invertible, so $\mathcal M_{ell}=\weier[\Delta^{-1}]$. When we unravel the definitions we find that $$H^{q,p}\left(\mathcal M_{ell}\right)=H^{q,p}\left(\weier[\Delta^{-1}]\right)=H^q\left(\weier[\Delta^{-1}],(\omega|_{\weier[\Delta^{-1}]})^{\otimes\frac12p}\right).$$ We know that the function $\Delta$ represents a class in $MF_*=H^0(\weier)$ and every section of $\omega|_{\weier[\Delta^{-1}]}$ can uniquely be restricted to a section of $\omega[\Delta^{-1}]$ on $\weier$. So we find that $$H^q\left(\weier[\Delta^{-1}],(\omega|_{\weier[\Delta^{-1}]})^{\otimes\frac12p}\right)=H^q\left(\weier,(\omega[\Delta^{-1}])^{\otimes\frac12p}\right)=H^{q,p}(\weier)[\Delta^{-1}].$$ We conclude that $H^*(\mathcal M_{ell})=H^*(\weier[\Delta^{-1}])=H^*(\weier)[\Delta^{-1}]$, which justifies our notation $\weier[\Delta^{-1}]$ for the substack of $\weier$ where $\Delta$ is invertible. Following the computations in \cite{bau} we find that the elements $\Delta^8$, $\Delta^3$ and $\Delta$ are non-zero-divisor generators in $\pi_*\tmf$ localized at respectively 2, 3 and higher primes. Hence adding the inverse of $\Delta$ to the third spectral sequence gives $H^*(\weier)[\Delta^{-1}]\Rightarrow\pi_*\tmf[\Delta^{-24}]$. When we compare this with the first spectral sequences we see that we can retrieve the periodic spectrum $\TMF$ from the connective spectrum $\tmf$ by the relation $$\pi_*\TMF=\pi_*\tmf[\Delta^{-24}].$$ We remark that the period of $\TMF$ is $24^2=576$.

In \cite{bau} Bauer shows how to compute the cohomology of the moduli stack $\weier$ of Weierstrass elliptic curves. Using his results we are going to compute the cohomology of the moduli stack $\mell$ of generalized elliptic curves. We will then run the spectral sequence (\ref{eq:ess}) to compute the homotopy groups of $\Tmf$.

We have seen that a Weierstrass elliptic curve is non-singular if and only if the discriminant $\Delta$ of the corresponding Weierstrass equation is invertible. A singular Weierstrass elliptic curve has no cuspidal singularities if and only if $c_4$ is invertible. Together the substacks $\weier[\Delta^{-1}]$ and $\weier[c_4^{-1}]$ form an open cover of the moduli stack of generalized elliptic curves: $$\mell=\weier[c_4^{-1}]\bigsqcup_{\weier[c_4^{-1},\Delta^{-1}]}\weier[\Delta^{-1}]\subset \weier.$$
Just like $\Delta$ the function $c_4$ represents a class in $H^0(\weier)$, hence we also know that $H^i(\weier[c_4^{-1}])=H^i(\weier)[c_4^{-1}]$ . In general, for a stack $\mathcal M$ and a function $a\in H^0(\mathcal M)$, we have the equality $H^i(\mathcal M[a^{-1}])=H^i(\mathcal M)[a^{-1}]$. For any ring we write $i_a:R\to R[a^{-1}]$ and we observe that $x\in\ker i_a$ if and only if there exists an $n\in\mathbb N$ such that $xa^n=0$ in $R$. For the above reason we call this the localization of $R$ away from $a$.

In particular we have the maps $$i_\Delta:H^i(\weier)[c_4^{-1}]\longrightarrow H^i(\weier)[c_4^{-1},\Delta^{-1}]$$ and $$i_{c_4}: H^i(\weier)[\Delta^{-1}]\longrightarrow H^i(\weier)[c_4^{-1},\Delta^{-1}].$$
Just as for spaces we can calculate $H^i(\mell)$ with the Mayer-Vietoris long exact sequence
\begin{equation}\label{eq:vie}
\cdots\stackrel{D}{\longrightarrow}H^{i-1}(\weier)[c_4^{-1},\Delta^{-1}]\stackrel{d}{\longrightarrow}H^i(\mell)\stackrel{\rho}{\longrightarrow}H^i(\weier)[c_4^{-1}]\oplus H^i(\weier)[\Delta^{-1}]\stackrel{D}{\longrightarrow}\cdots,\end{equation} where $\rho$ sends a class to the direct sum of its restriction to the substacks and where $D$ sends $v\oplus w$ to $i_\Delta v-i_{c_4} w$. The long exact sequence also induces a short exact sequence
\begin{equation}\label{eq:sh}
 0\longrightarrow\coker^{i-1} D\longrightarrow H^i(\mell)\longrightarrow\ker^i D\longrightarrow0,\end{equation}
where $\coker^{i-1}D$ is our notation for the cokernel of $D$ in the $i-1^{\textrm{th}}$ cohomology filtration and $\ker^iD$ is the kernel of $D$ in the $i^{\textrm{th}}$ cohomology filtration.



\begin{theorem}\label{th:main}
The $\mathbb Z$-module $\pi_*\Tmf$ only has torsion at the primes 2 and 3. In positive degrees, one can write down a basis of the torsion-free part
of $\pi_*\Tmf$ in terms of the standard basis $\{c_4^jc_6^k\Delta^l\}$ of $MF_*$.
The basis consists of elements $2^a3^bc_4^jc_6^k\Delta^l$ in degree $8j + 12k + 24l$ with $j,l\geq0$ and $k\in\{0,1\}$, where \begin{align*}
a&=\begin{cases} 1\quad \textrm{if}\ k=1,\\ 3\quad \textrm{if}\ j=k=0\ \textrm{and}\ l\equiv 1,3,5,7\mod 8,\\ 2\quad \textrm{if}\ j=k=0\ \textrm{and}\ l\equiv 2,6\mod 8,\\ 1 \quad \textrm{if}\ j=k=0\ \textrm{and}\ l\equiv 4\mod 8,\\0\quad \textrm{else},\end{cases}\\
\textrm{and}\\
b&=\begin{cases} 1\quad \textrm{if}\ j=k=0\ \textrm{and}\ l\equiv 1,2 \mod 3,\\ 0\quad \textrm{else}.\end{cases}\end{align*}
In negative degrees, there is a basis of the torsion free part of
$\pi_*\Tmf$ consisting of elements $[2^a3^bc_4^jc_6^k\Delta^l]$ in degree $8j+12k+24l-1$, with $j,l\leq-1$ and $k\in\{0,1\}$, where $a$ and $b$ are now given by \begin{align*}
a&=\begin{cases} 1\quad \textrm{if}\ j<-1\ \textrm{and}\ k=1,\\ -2\quad \textrm{if}\ j=-1,\ k=1\ \textrm{and}\ l\equiv 0,2,4,6 \mod 8,\\ 
-1\quad \textrm{if}\ j=-1,\ k=1\ \textrm{and}\ l\equiv 3,5 \mod 8,\\ 0\quad \textrm{else},\end{cases}\\
\textrm{and}\\
b&=\begin{cases} -1\quad \textrm{if}\ j=-1,\ k=1\ \textrm{and}\ l\equiv 0, 1 \mod 3,\\ 0\quad \textrm{else}.\end{cases}\end{align*}
The values of $m$ and $n$ of the goups $\mathbb Z/(2^m3^n\mathbb Z)$ which build up the torsion can be found in Theorem \ref{th:2} and Theorem \ref{th:3} respectively.
\end{theorem}
\begin{proof}
In Theorem \ref{th:6inv} we will find that $\pi_*\Tmf[\frac16]$ is built up additively by copies of $\mathbb Z[\frac16]$. This implies that there is only torsion at the primes 2 and 3. 
At the same time this determines the multiplicative structure of the torsion free part of $\pi_*\Tmf$ up to powers of two and three.  We find these values of $a$ and $b$ by localizing at two in Theorem \ref{th:2} and at three in Theorem \ref{th:3}. These two theorems also compute the torsion and the rest of the multiplicative structure of $\pi_*\Tmf$.
\end{proof}

\subsection{Acknowledgements}

This paper was written under the supervision of Andr\'e Henriques. I am indebted to him for choice of the subject and many helpful discussions. The idea of how to do this computation is due to Mike Hill. The results are heavily based on the computations of Tilman Bauer on the homotopy groups of $\tmf$. We also used his package {\emph sseq}, which is available on his website, to draw all the pictures in this paper.

\nocite{goe}

\section{Drawing the spectral sequences}\label{sec:leg}
The spectral sequences we will be looking at start with the doubly graded cohomology $H^{q,p}(\mathcal M)$ of some stack $\mathcal M$ as the $E_2$-sheet. We draw this sheet using Adams indexing. In the square with coordinates $(x,y)$ we draw the group $H^{y, x+y}$. All the following sheets will be drawn in the same picture and we will sometimes allow ourselves to talk about elements of $H^{q,p}$, while we actually mean the corresponding element in $E_r^{q,p}$ for some $r\geq 2$. Notice that, because we use the Adams indexing, a given homotopy group $\pi_k \Tmf$ corresponds to a column of the $E_\infty$-sheet.
\begin{itemize}
\item We introduce the following symbols for our drawings localized at a prime $n$.\smallskip

\begin{tabular}{>{$}c<{$}|>{\centering $}p{0.15\textwidth}<{$}>{\centering $}p{0.15\textwidth}<{$}>{\centering $}p{0.15\textwidth}<{$}>{\centering $}p{0.15\textwidth}<{$}} n=3 & \cdot & [c_4^{-3}\Delta] & [c_4^3\Delta^{-1}] & [c_4^{-3}\Delta, c_4^3\Delta^{-1}] \tabularnewline \hline \mbb Z_{(3)} & \square & \boxbackslash & \boxslash & \boxtimes \tabularnewline \mbb Z/3 & \bullet & \obackslash & \oslash & \otimes \end{tabular}\bigskip

\begin{tabular}{>{$}c<{$}|>{\centering $}p{0.15\textwidth}<{$}>{\centering $}p{0.15\textwidth}<{$}>{\centering $}p{0.15\textwidth}<{$}>{\centering $}p{0.15\textwidth}<{$}} n=2 & \cdot & [c_4^{-3} \Delta] & [c_4^3\Delta^{-1}] & [c_4^{-3}\Delta, c_4^3\Delta^{-1}] \tabularnewline \hline \mbb Z_{(2)} & \square & \boxminus & \boxvert & \boxplus \tabularnewline \mbb Z/2 & \bullet & \ominus & \overt & \oplus \end{tabular}\smallskip

To be more precise we say this in words below.
\begin{itemize}
\item A dot $\bullet$ represents a copy of $\mbb Z/n$.
\item A square $\square$ represents a copy of $\mbb Z_{(n)}$.
\item A dashed square $\boxbackslash$ resp. a dashed circle $\obackslash$ represents a copy of $\mbb Z_{(3)}[c_4^{-3}\Delta]$ resp. $\mbb Z/3[c_4^{-3}\Delta]$.
\item A dashed square $\boxslash$ resp. a dashed circle $\oslash$ represents a copy of $\mbb Z_{(3)}[c_4^3\Delta^{-1}]$ resp. $\mbb Z/3[c_4^3\Delta^{-1}]$.
\item A crossed square $\boxtimes$ resp. a crossed circle $\otimes$ represents a copy of $\mbb Z_{(3)}[c_4^{-3}\Delta,c_4^3\Delta^{-1}]$ resp. $\mbb Z_{(3)}[c_4^{-3}\Delta,c_4^3\Delta^{-1}]$.
\item A square with a horizontal line $\boxminus$ resp. a circle with a horizontal line $\ominus$ represents a copy of $\mbb Z_{(2)}[c_4^{-3}\Delta]$ resp. $\mbb Z/2[c_4^{-3}\Delta]$.
\item A square with a vertical line $\boxvert$ resp. a circle with a vertical line $\overt$ represents a copy of $\mbb Z_{(2)}[c_4^3\Delta^{-1}]$ resp. $\mbb Z/2[c_4^3\Delta^{-1}]$.
\item A crossed square $\boxplus$ resp. a crossed circle $\oplus$ represents a copy of $\mbb Z_{(2)}[c_4^{-3}\Delta,c_4^3\Delta^{-1}]$ resp. $\mbb Z_{(2)}[c_4^{-3}\Delta,c_4^3\Delta^{-1}]$.
\end{itemize}
\item Furthermore, a circled dot represents a copy of $\mbb Z/n^2$ and a doubly circled dot a copy of $\mbb Z/n^3$.
\item A line of positive slope denotes a multiplication by $\alpha$, $h_1$ or $h_2$ (see fig. \ref{fig:p31} and \ref{fig:p21}). 
\item An arrow of negative slope denotes a differential. Classes that do not survive in higher sheets of the spectral sequence are drawn in gray.
\end{itemize}
When computing $H^*(\mell)$ localized at the primes 2 and 3 we will find that in most bidegrees either the kernel of $D=(i_{\Delta}, -i_{c_4})$ or the cokernel of $D$ is zero. This implies that the short exact sequence (\ref{eq:sh}) induces an isomorphism for these bidegrees. There are a few cases however where both the kernel and the cokernel are non-zero. In these cases the Mayer-Vietoris sequence (\ref{eq:vie}) or the induced short exact sequence are not enough to determine the additive structure of $H^{q,p}(\mell)$. To resolve this problem we will have to resort to cohomology with $\mbb Z/n$-coefficients, by which we mean  $$H^q\left(\mathcal M, \omega^{\otimes\frac12p}\otimes_{\mbb Z}\mbb Z/n\right).$$ The relation between integral cohomology and cohomology with $\mbb Z/n$-coefficients is given by a version of the Universal Coefficient Theorem.
The short exact sequence $$0\longrightarrow\omega\stackrel{\cdot n}{\longrightarrow}\omega\longrightarrow\omega\otimes_{\mbb Z}\mbb Z/n\longrightarrow 0$$ of sheaves induces a exact sequence for sheaf cohomology $$0\longrightarrow H^q(\mathcal M,\omega^{\otimes\frac12 p})\otimes_{\mbb Z} \mbb Z/n\longrightarrow H^q(\mathcal M,\omega^{\otimes\frac12 p}\otimes_{\mbb Z}\mbb Z/n)\longrightarrow\Tor(H^{q+1}(\mathcal M,\omega^{\otimes\frac12 p}),\mbb Z/n)\longrightarrow0,$$ which splits, but not naturally. In line with the abreviation $H^q(\mathcal M)=H^q(\mathcal M,\omega^{\otimes\frac12p})$ we will adopt the notation $H^q(\mathcal M,\mbb Z/n)=H^q(\mathcal M,\omega^{\otimes\frac12 p}\otimes_{\mbb Z}\mbb Z/n)$.

Because the exact sequence splits we may pretend that $H^q(\mathcal M)$ and $H^{q+2}(\mathcal M)$ are zero for the moment, calculate $H^{q}(\mathcal M,\mbb Z/n)$ and $H^{q+1}(\mathcal M,\mbb Z/n)$ from $H^{q+1}(\mathcal M)$ and take the direct sum over all $q$ afterwards. When $H^{q+1}(\mathcal M)=\mbb Z$, we get $H^q(\mathcal M, \mbb Z/n)\simeq\Tor(\mbb Z,\mbb Z/n)=0$ and $H^{q+1}(\mathcal M,\mbb Z/n)\simeq\mbb Z\otimes\mbb Z/n\simeq\mbb Z/n$. If we have $H^{q+1}(\mathcal M)=\mbb Z/n^k$, then we find $H^q(\mathcal M, \mbb Z/n)\simeq\Tor(\mbb Z/n^k,\mbb Z/n)\simeq\mbb Z/\gcd(n^k,n)=\mbb Z/n$ and $H^{q+1}(\mathcal M,\mbb Z/n)\simeq\mbb Z/n^k\otimes Z/n\simeq\mbb Z/n$. 

\begin{tabular}{l|cc}
 & $\mbb Z$-coefficients & $\mbb Z/n$-coefficients\\
\hline
&&\\
$\mbb Z$ &
\raisebox{-.6\height}{\begin{sseq}[labels=none]{2}{2}\ssmoveto {0} 1 \ssdrop{\square}\end{sseq}} &
\raisebox{-.6\height}{\begin{sseq}[labels=none]{2}{2}\ssmoveto {0} 1 \ssdropbull\end{sseq}}\\
$\mbb Z/n$ &
\raisebox{-.6\height}{\begin{sseq}[labels=none]{2}{2}\ssmoveto {0} 1 \ssdropbull\end{sseq}}&
\raisebox{-.6\height}{\begin{sseq}[labels=none]{2}{2}\ssmoveto {0} 1 \ssdropbull\ssmoveto 1 {0} \ssdropbull\end{sseq}}\\
$\mbb Z/n^2$ &
\raisebox{-.6\height}{\begin{sseq}[labels=none]{2}{2}\ssmoveto {0} 1 \ssdropbull\ssdropextension\end{sseq}}&
\raisebox{-.6\height}{\begin{sseq}[labels=none]{2}{2}\ssmoveto {0} 1 \ssdropbull\ssmoveto 1 {0} \ssdropbull\end{sseq}}\\
$\mbb Z/n^3$ &
\raisebox{-.6\height}{\begin{sseq}[labels=none]{2}{2}\ssmoveto {0} 1 \ssdropbull\ssdropextension\ssdropextension\end{sseq}}&
\raisebox{-.6\height}{\begin{sseq}[labels=none]{2}{2}\ssmoveto {0} 1 \ssdropbull\ssmoveto 1 {0} \ssdropbull\end{sseq}}\\
\end{tabular}


\section{$\pi_*\Tmf$ when 6 is invertible}\label{sec:6in}

We recall that we compute the homotopy groups of the non-connective spectrum of topological modular forms $\Tmf$ using the convergent elliptic spectral sequence
$$ H^q\left(\mell, \omega^{\otimes p}\right)\Rightarrow\pi_{2p-q}\Tmf$$
and the Mayer-Vietoris long exact sequence (\ref{eq:vie}) or the induced short exact sequence (\ref{eq:sh}).

In this section we localize our moduli problem away from 2 and 3. In this case $\tmf$ becomes a complex oriented spectrum and we find that $H^*(\weier)=H^0(\weier)=\mathbb Z[\frac16, c_4, c_6,\Delta]/(c_4^3-c_6^2-1728\Delta=0)$, the polynomial algebra over $\mathbb Z[\frac16]$ with generators $c_4$ and $c_6$. We can also write this additively as $\mathbb Z[\frac16, c_4, \Delta]\otimes E(c_6)$ with $E(c_6)$ the exterior algebra in $c_6$, because $c_6^2$ can be expressed in terms of $\frac16$, $c_4$ and $\Delta$. Finally, we can also rewrite $\mathbb Z[\frac16, c_4, \Delta]\otimes E(c_6)$ as $\mathbb Z[\frac16]\left\{c_4^ac_6^b\Delta^c\right\}\hspace{-2mm}\tiny{\begin{array}{l}a\geq 0 \\ b=0,1 \\ c\geq0\end{array}}$, by which we mean the $\mbb Z[\frac16]$-algebra generated by the elements $c_4^ac_6^b\Delta^c$ with $a\geq0$, $b\in\{0,1\}$ and $c\geq0$. We find furthermore that $$H^*(\weier)[c_4^{-1}]=\mathbb Z[\frac16, c_4^{\pm 1}, \Delta]\otimes E(c_6)=\mathbb Z[\frac16]\left\{c_4^ac_6^b\Delta^c\right\}\hspace{-2mm}\tiny{\begin{array}{l}a\in\mathbb Z \\ b=0,1 \\ c\geq0\end{array}}$$ and $$H^*(\weier)[\Delta^{-1}]=\mathbb Z[\frac16, c_4, \Delta^{\pm 1}]\otimes E(c_6)=\mathbb Z[\frac16]\left\{c_4^ac_6^b\Delta^c\right\}\hspace{-2mm}\tiny{\begin{array}{l}a\geq0 \\ b=0,1 \\ c\in\mathbb Z\end{array}}.$$ Both maps $i_{c_4}$ and $i_\Delta$ have trivial kernel, so $D(v\oplus w)=i_\Delta(v)-i_{c_4}(w)=v-w$. Now we compute the kernel, the image, and the cokernel of $D$ as follows:
\begin{eqnarray*}
\ker D&=&\{v\oplus v: v\in H^*(\weier)\}\simeq H^*(\weier),\\
\im D&=&H^*(\weier)[c_4^{-1}]+H^*(\weier)[\Delta^{-1}]=\mathbb Z[\frac16]\left\{c_4^ac_6^b\Delta^c, c_4^cc_6^b\Delta^a \right\}\hspace{-2mm}\tiny{\begin{array}{l}a\geq0 \\ b=0,1 \\ c\in\mathbb Z\end{array}},\\
\coker D&=&\mathbb Z[\frac16]\left\{c_4^ac_6^b\Delta^c \right\}\hspace{-2mm}\tiny{\begin{array}{l}a<0 \\ b=0,1 \\ c<0 \end{array}}=c_4^{-1}c_6\Delta^{-1}\cdot\mathbb Z[\frac16]\left\{c_4^ac_6^b\Delta^c \right\}\hspace{-2mm}\tiny{\begin{array}{l}a\leq0 \\ b=0,-1 \\ c\leq0\end{array}}.
\end{eqnarray*}
So the basis elements of the cokernel are just $c_4^{-1}c_6\Delta^{-1}$ times the inverses of basis elements of $H^*(\weier)$. We will adopt the notation $$\left(H^i(\weier)\right)^{-1}:=\mbb Z[\frac16]\left\{(c_4^ac_6^b\Delta^c)^{-1}: c_4^ac_6^b\Delta^c\in H^i(\weier)\right\}.$$ We know that $H^i(\weier)=0$ for $i>0$, so we have derived that the $E_2$-sheet of the elliptic spectral sequence is formed by $$H^0(\mell)\simeq H^0(\weier),\qquad H^1(\mell)\simeq c_4^{-1}c_6\Delta^{-1}\cdot\left(H^0(\weier)\right)^{-1}.$$
We denote the basis elements of $H^1(\mell)$ by $[c_4^{-1}c_6\Delta^{-1}t^{-1}]$ where $t$ is a basis element of $H^0(\weier)$. By degree reasons there can be no differentials. Hence the $E_2$-sheet is also the $E_\infty$-sheet and we have proven the following theorem.
\begin{theorem}\label{th:6inv}
Additively $\pi_*\Tmf[\frac16]$ is given by $$\mathbb Z[\frac16]\left\{c_4^jc_6^k\Delta^l \right\}\hspace{-2mm}\tiny{\begin{array}{l}j\geq0 \\ k=0,1 \\ l\geq0 \end{array}}\oplus\mathbb Z[\frac16]\left\{[c_4^jc_6^k\Delta^l] \right\}\hspace{-2mm}\tiny{\begin{array}{l}j\leq-1 \\ k=0,1 \\ l\leq-1 \end{array}}$$ with $c_4^jc_6^k\Delta^l$ in degree $8j+12k+24l$ and $[c_4^jc_6^k\Delta^l]$ in degree $8j+12k+24l-1$. The multiplicative structure is generated by the relation $c_4^3-c_6^2-1728\Delta=0$ and the product on basis elements given by
\begin{eqnarray*}
c_4^jc_6^k\Delta^l\cdot c_4^{j'}c_6^{k'}\Delta^{l'} &=&c_4^{j+j'}c_6^{k+k'}\Delta^{l+l'},\\
c_4^jc_6^k\Delta^l\cdot [c_4^{j'}c_6^{k'}\Delta^{l'}] &=&\begin{cases}[c_4^{j+j'}c_6^{k+k'}\Delta^{l+l'}]\quad&\textrm{if}\quad j+j'\leq-1, l+l'\leq-1, \\ 0\quad&\textrm{if}\quad j+j'>-1, l+l'>-1, \end{cases} \\
\ [c_4^jc_6^k\Delta^l] \cdot[c_4^{j'}c_6^{k'}\Delta^{l'}] &=& 0.
\end{eqnarray*}
\end{theorem}
\begin{figure}[H]
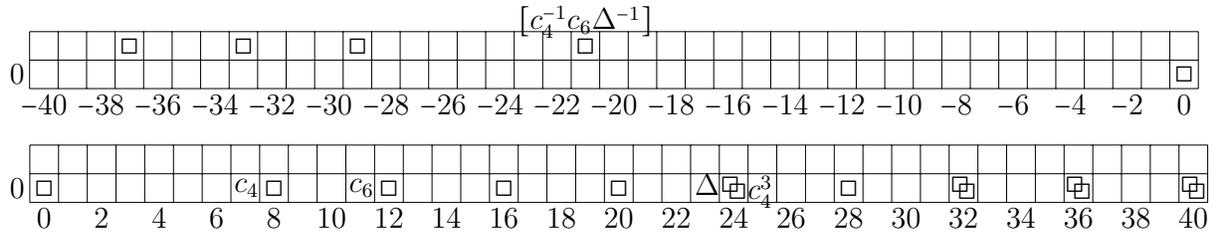

\res{
\begin{sseq}{-40...0}{2}
\ssdrop{\square}
\ssmove 4 {0}
\movefour
\ssdroplabel{c_4\,}
\movefour
\ssdroplabel{c_6\,}
\movefour
\movefour
\movefour
\ssdroplabel{\Delta}
\ssdrop{\square}
\movefour
\movefour
\ssdrop{\square}
\movefour
\ssdrop{\square}
\movefour
\ssdrop{\square}
\ssmoveto {-21} 1
\ssdrop{\square}
\ssdroplabel[U]{[c_4^{-1}c_6\Delta^{-1}]}
\ssmove {-4} 0
\minfour
\minfour
\minfour
\minfour
\ssdrop{\square}
\end{sseq}
}\bigskip

\res{
\begin{sseq}{0...40}{2}
\ssdrop{\square}
\ssmove 4 {0}
\movefour
\ssdroplabel{c_4\,}
\movefour
\ssdroplabel{c_6\,}
\movefour
\movefour
\movefour
\ssdroplabel{\Delta}
\ssdrop{\square}\ssdroplabel[R]{c_4^3}
\movefour
\movefour
\ssdrop{\square}
\movefour
\ssdrop{\square}
\movefour
\ssdrop{\square}
\ssmoveto {-21} 1
\ssdrop{\square}
\ssmove {-4} 0
\minfour
\minfour
\minfour
\minfour
\ssdrop{\square}
\end{sseq}}
\caption{$\pi_*\Tmf$ when 6 is invertible.}
\end{figure}\vspace{-.7cm}
Here every $\square$ denotes a copy of $\mbb Z[\frac16]$ and the elements $c_4$, $c_6$, $\Delta$, $c_4^3$ and $[c_4^{-1}c_6\Delta^{-1}]$ are noted. The $k^\textrm{th}$ column of this $E_\infty$-sheet corresponds to $\pi_k\Tmf$.

\section{$\pi_*\Tmf$ localized at the prime 3}\label{sec:3}

As indicated before, the spectrum $\Tmf$, as well as the other spectra of topological modular forms, behaves rather differently at the primes 2 and 3 than at the other primes. Of these two cases $\pi_*\Tmf$ turns out to be significantly simpler at the prime 3. The following picture from \cite{bau} will be our starting point.

\ret{41}{9}{
\ssdrop{\square}
\ssline 3 1
\ssdropbull
\ssdroplabel[U]{\alpha}
\ssmove 7 1
\ssdropbull \ssdroplabel{\beta} \alphamult
\ssdropbull \alphamult
\ssdropbull \alphamult
\ssdropbull \alphamult
\ssmoveto {0} {0}
\ssmove 4 {0}
\movefour
\ssdroplabel{c_4\,}
\movefour
\ssdroplabel{c_6\,}
\movefour
\movefour
\movefour
\ssdroplabel{\Delta}
\ssname{d}
\ssdrop{\square}
\movefour
\movefour
\ssdrop{\square}
\movefour
\ssdrop{\square}
\movefour
\ssdrop{\square}
\ssgoto{d}
\alphamult
\ssdropbull \alphamult
}{$H^*(\weier)$ localized at 3.\label{fig:p31}}

Here, every $\bullet$ denotes a copy of $\mbb Z/3$ and every $\square$ denotes a copy of $\mbb Z_{(3)}$. Recall that a line connecting two classes indicates a multiplication, in this case by $\alpha$.

\subsection{The cohomology of $\mell$}


Because $c_4\alpha=c_6\alpha=0$, the map $i_{c_4}:H^*(\weier)\longrightarrow H^*(\weier)[c_4^{-1}]$ sends $\alpha$ and $\beta$ to zero. For the elements in filtration zero, $i_{c_4}$ acts as an inclusion, so $$H^*(\weier)[c_4^{-1}]=H^0(\weier)[c_4^{-1}]=\mathbb Z_{(3)}[c_4, \Delta, c_4^{-1}]\otimes E(c_6).$$
We find, for example, that $H^{0,0}(\weier)[c_4^{-1}]\simeq\mbb Z_{(3)}[c_4^{-3}\Delta]$ and $H^{0,4}(\weier)[c_4^{-1}]=c_4^{-1}c_6\cdot H^{0,0}(\weier)[c_4^{-1}]$.

\ret{-8...32}{2}{
\ssmoveto {-8} {0}
\ssdrop{\boxbackslash}
\ssmove 4 {0}
\ssdrop{\boxbackslash}
\ssmove 4 {0}
\ssdrop{\boxbackslash}
\ssmove 4 {0}
\ssdrop{\boxbackslash}
\ssdroplabel{c_4^{-1}c_6\,}
\ssmove 4 {0}
\ssdrop{\boxbackslash}
\ssmove 4 {0}
\ssdrop{\boxbackslash}
\ssmove 4 {0}
\ssdrop{\boxbackslash}
\ssmove 4 {0}
\ssdrop{\boxbackslash}
\ssmove 4 {0}
\ssdrop{\boxbackslash}
\ssmove 4 {0}
\ssdrop{\boxbackslash}
\ssmove 4 {0}
\ssdrop{\boxbackslash}
\ssmove 4 {0}
\ssdrop{\boxbackslash}
\ssmove 4 {0}
\ssdrop{\boxbackslash}
\ssmove 4 {0}
\ssdrop{\boxbackslash}
}{$H^*(\weier)[c_4^{-1}]$  localized at 3.}

Here, every $\boxbackslash$ denotes a copy of $\mbb Z_{(3)}[c_4^{-3}\Delta]$.

On the other hand, $\Delta$ is a not a zero divisor in $H^*(\weier)$, so $i_\Delta : H^*(\weier)\longrightarrow H^*(\weier)[\Delta^{-1}]$ is an inclusion. Writing $H^{0,0}(\weier)[\Delta^{-1}]=\mbb Z_{(3)}[c_4^3\Delta^{-1}]$, the filtration zero looks the same as above, but with $c_4^3\Delta^{-1}$ instead of $c_4^{-3}\Delta$.

\ret{-8...32}{12}{
\ssmoveto {-8} {0}
\ssdrop{\boxslash}
\ssmove 4 {0}
\ssdrop{\boxslash}
\ssmove 4 {0}
\ssdrop{\boxslash}
\ssmove 4 {0}
\ssdrop{\boxslash}
\ssmove 4 {0}
\ssdrop{\boxslash}
\ssdroplabel{c_4\,}
\ssmove 4 {0}
\ssdrop{\boxslash}
\ssdroplabel{c_6\,}
\ssmove 4 {0}
\ssdrop{\boxslash}
\ssmove 4 {0}
\ssdrop{\boxslash}
\ssmove 4 {0}
\ssdrop{\boxslash}
\ssdroplabel{\Delta}
\ssmove 4 {0}
\ssdrop{\boxslash}
\ssmove 4 {0}
\ssdrop{\boxslash}
\ssmove 4 {0}
\ssdrop{\boxslash}
\ssmove 4 {0}
\ssdrop{\boxslash}
\ssmove 4 {0}
\ssdrop{\boxslash}
\ssmoveto {-8} 8
\ssdropbull \alphamult
\ssdropbull \alphamult
\ssmoveto {-4} 4
\ssdropbull \alphamult
\ssdropbull \alphamult
\ssdropbull \alphamult
\ssdropbull \alphamult
\ssmoveto 0 0
\ssline 3 1
\ssdropbull
\ssdroplabel[U]{\alpha}
\ssmove 7 1
\ssdropbull \ssdroplabel{\beta} \alphamult
\ssdropbull \alphamult
\ssdropbull \alphamult
\ssmoveto {24} 0
\alphamult
}{$H^*(\weier)[\Delta^{-1}]$  localized at 3.}

Notice that every $\boxslash$ denotes a copy of $\mbb Z_{(3)}[c_4^3\Delta^{-1}]$, but that a $\bullet$ just denotes a copy of $\mbb Z/3$. 

At this point it is easy to construct $H^*(\weier)[c_4^{-1},\Delta^{-1}]$, for example by starting from $H^*(\weier)[c_4^{-1}]$ and localizing away from $\Delta$.

\ret{-8...32}{2}{
\ssmoveto {-8} {0}
\ssdrop{\boxtimes}
\ssmove 4 {0}
\ssdrop{\boxtimes}
\ssmove 4 {0}
\ssdrop{\boxtimes}
\ssmove 4 {0}
\ssdrop{\boxtimes}
\ssmove 4 {0}
\ssdrop{\boxtimes}
\ssdroplabel{c_4\,}
\ssmove 4 {0}
\ssdrop{\boxtimes}
\ssdroplabel{c_6\,}
\ssmove 4 {0}
\ssdrop{\boxtimes}
\ssmove 4 {0}
\ssdrop{\boxtimes}
\ssmove 4 {0}
\ssdrop{\boxtimes}
\ssdroplabel{\Delta}
\ssmove 4 {0}
\ssdrop{\boxtimes}
\ssmove 4 {0}
\ssdrop{\boxtimes}
\ssmove 4 {0}
\ssdrop{\boxtimes}
\ssmove 4 {0}
\ssdrop{\boxtimes}
\ssmove 4 {0}
\ssdrop{\boxtimes}
}{$H^*(\weier)[c_4^{-1}, \Delta^{-1}]$  localized at 3.}

Here, every $\boxtimes$ represents a copy of $\mbb Z_{(3)}[c_4^{-3}\Delta,c_4^3\Delta^{-1}]$.

To compute $H^*(\mell)$ we observe that $i_\Delta: H^*(\weier)[c_4^{-1}]\longrightarrow H^*(\weier)[c_4^{-1}, \Delta^{-1}]$ is just the inclusion and that $i_{c_4}: H^*(\weier)[\Delta^{-1}]\longrightarrow H^*(\weier)[c_4^{-1}, \Delta^{-1}]$ maps $\alpha$ and $\beta$ to zero and is an inclusion in filtration zero. The kernel of $D: H^*(\weier)[c_4^{-1}]\oplus H^*(\weier)[\Delta^{-1}]\longrightarrow H^*(\weier)[c_4^{-1}, \Delta^{-1}]$ is the sum of two pieces. The first piece consists of the elements of $H^*(\weier)[\Delta^{-1}]$ in filtration one or higher and the second piece consists of those elements that appear in filtration zero of both $H^*(\weier)[c_4^{-1}]$ and $H^*(\weier)[\Delta^{-1}]$. So we end up with
$$\ker D=\left(0\oplus\alpha\cdot\mbb Z/3[\beta,\Delta^{\pm1}]\right)\bigoplus\left(0\oplus\beta\cdot\mbb Z/3[\beta,\Delta^{\pm1}]\right)\bigoplus\left\{v\oplus v: v\in \mathbb Z_{(3)}[c_4, \Delta]\otimes E(c_6) \right\}.$$ 
For the cokernel of $D$, we remark first that the image of $D$ is given by
$$\im D=\mbb Z_{(3)}\left\{c_4^ac_6^b\Delta^c, c_4^cc_6^b\Delta^a\right\}\hspace{-2mm}\tiny{\begin{array}{l}a\in\mathbb Z \\ b=0,1 \\ c\geq0\end{array}}\subset H^*(\weier)[c_4^{-1},\Delta^{-1}]=H^0(\weier)[c_4^{-1},\Delta^{-1}].$$
The cokernel is now found by taking the quotient $H^0(\weier)[c_4^{-1},\Delta^{-1}]/\im D$. We continue to compute $H^*(\mell)$ using the short exact sequence (\ref{eq:sh}) by looking at different filtrations separately. 

For $i\neq1$ we have $H^{i-1}(\weier)[c_4^{-1},\Delta^{-1}]=0$. Hence $\coker^{i-1}D=0$ and $H^i(\mell)=\ker^i D$.

On the other hand for $i=1$ we find $\ker^1 D=\alpha\cdot\mbb Z_{(3)}[\Delta^{\pm1}]$, which only has elements in bidegrees $(1,24j+4)$ for all $j\in\mbb Z$, and $\coker^0 D\simeq c_4^{-1}c_6\Delta^{-1}\cdot\left(H^0(\weier)\right)^{-1}$.
So $H^1(\mell)$ is given by $H^1(\weier)$ in the positive degrees and looks in general like $c_4^{-1}c_6\Delta^{-1}\cdot\left(H^0(\weier)\right)^{-1}$ in the negative degrees, with the exception of $H^{1,-24j+4}(\weier)$ with $j>0$. Indeed, these are the cases where both the cokernel and the kernel are non-zero and we find for the short exact sequence
\begin{align*}0\quad&\longrightarrow& H^{0, 24(j-1)}(\weier)\quad&\longrightarrow& H^{1,-24j+4}(\mell)\quad&\longrightarrow&\alpha\Delta^{-j}\cdot\mbb Z/3\quad&\longrightarrow&0,\\ &&t\quad&\longmapsto&[c_4^{-1}c_6\Delta^{-1}t].\qquad \end{align*}
This leaves us with the following two possibilities: \begin{equation}\label{eq:is3} H^{1, -24j+4}(\mell)\simeq\begin{cases} \mbb Z_{(3)}^j&\:\text{or}\\ \mbb Z_{(3)}^j\times \mbb Z/3&.\end{cases}\end{equation}

\subsection{The mod 3 cohomology of $\mell$}

The trick to find out which of the two cases in (\ref{eq:is3}) is true, is to compute the same groups with $\mbb Z/3$-coefficients; that is, we are going to compute $H^*(\mell, \mbb Z/3)$, which is our shorthand for $H^q(\mell,\omega^{\otimes\frac12p}\otimes_{\mbb Z}\mbb Z/3)$. We start by using the Universal Coefficient Theorem on the first diagram of this section, Figure \ref{fig:p31}.

\ret{41}{9}{
\ssdropbull
\ssline 3 1
\ssdropbull
\ssdroplabel[U]{\alpha}
\ssmove 7 1
\ssdropbull \ssdroplabel{\beta} \alphamult
\ssdropbull \alphamult
\ssdropbull \alphamult
\ssdropbull \alphamult
\ssmoveto {11} 1
\ssdropbull \alphamult
\ssdropbull \alphamult
\ssdropbull \alphamult
\ssdropbull \alphamult
\ssmoveto {0} {0}\ssmove 4 {0}
\ssdropbull \ssdroplabel[L]{a_2} \ssmove 4 {0}
\ssdropbull \ssmove 4 {0}
\ssdropbull \ssmove 4 {0}
\ssdropbull \ssmove 4 {0}
\ssdropbull \ssmove 4 {0}
\ssdropbull \ssdroplabel{\Delta} \ssname{d}
\ssdropbull \ssmove 4 {0}
\ssdropbull \ssdropbull \ssmove 4 {0}
\ssdropbull \ssdropbull \ssmove 4 {0}
\ssdropbull \ssdropbull \ssmove 4 {0}
\ssdropbull \ssdropbull \ssmove 4 {0}
\ssdropbull \ssdropbull \ssmove 4 {0}
\ssgoto{d}
\alphamult
\ssdropbull \alphamult
\ssmoveto {35} 1
\ssdropbull \alphamult
}{$H^*(\weier, \mbb Z/3)$ localized at 3.}

Every $\bullet$ denotes a copy of $\mbb Z/3$ and we labelled the new generator in bidegree $(0,4)$ by $a_2$.


We approach the computation of $H^*(\mell, \mbb Z/3)$ in the same manner as the computation of $H^*(\mell)$. So we start by localizations away from $a_2$ and $\Delta$ and then we use the Mayer-Vietoris sequence (\ref{eq:vie}). It is clear that $a_2\alpha=0$ and $a_2^2\beta=0$, so $H^*(\weier, \mbb Z/3)[a_2^{-1}]$ is very easy.

\ret{-8...32}{2}{
\ssmoveto {-8} {0}
\ssdrop{\obackslash} \ssmove 4 {0}
\ssdrop{\obackslash} \ssdroplabel[U]{a_2^{-1}}\ssmove 4 {0}
\ssdrop{\obackslash} \ssmove 4 {0}
\ssdrop{\obackslash} \ssdroplabel[L]{a_2} \ssmove 4 {0}
\ssdrop{\obackslash} \ssmove 4 {0}
\ssdrop{\obackslash} \ssmove 4 {0}
\ssdrop{\obackslash} \ssmove 4 {0}
\ssdrop{\obackslash} \ssmove 4 {0}
\ssdrop{\obackslash} \ssmove 4 {0}
\ssdrop{\obackslash} \ssmove 4 {0}
\ssdrop{\obackslash} \ssmove 4 {0}
\ssdrop{\obackslash} \ssmove 4 {0}
}{$H^*(\weier, \mbb Z/3)[a_2^{-1}]$.}

Here, every $\obackslash$ represents a copy of $\mbb Z/3[a_2^{-6}\Delta]$. The computation of $H^*(\weier, \mbb Z/3)[\Delta^{-1}]$ isn't that complicated either and we find the following picture.

\ret{-8...32}{8}{
\ssmoveto {-8} {0}
\ssdrop{\oslash} \ssmove 4 {0}
\ssdrop{\oslash}\ssdroplabel[U]{\Delta^{-1}a_2^5} \ssmove 4 {0}
\ssdrop{\oslash} \ssmove 4 {0}
\ssdrop{\oslash} \ssdroplabel[R]{a_2} \ssmove 4 {0}
\ssdrop{\oslash} \ssmove 4 {0}
\ssdrop{\oslash} \ssmove 4 {0}
\ssdrop{\oslash} \ssmove 4 {0}
\ssdrop{\oslash} \ssmove 4 {0}
\ssdrop{\oslash} \ssdroplabel{\Delta} \ssname{d} \ssmove 4 {0}
\ssdrop{\oslash} \ssmove 4 {0}
\ssdrop{\oslash} \ssmove 4 {0}
\ssdrop{\oslash} \ssmove 4 {0}
\ssmoveto {-24} {0}
\ssdrop{\oslash} \alphamult
\ssdropbull \alphamult
\ssdropbull \alphamult
\ssdropbull \alphamult
\ssdropbull \alphamult
\ssmoveto {-13} 1
\ssdropbull \alphamult
\ssdropbull \alphamult
\ssdropbull \alphamult
\ssdropbull \alphamult
\ssdropbull \alphamult
\ssmoveto {0} {0}
\ssline 3 1
\ssdropbull
\ssdroplabel[U]{\alpha}
\ssmove 7 1
\ssdropbull \ssdroplabel{\beta} \alphamult
\ssdropbull \alphamult
\ssdropbull \alphamult
\ssdropbull \alphamult
\ssmoveto {11} 1
\ssdropbull \alphamult
\ssdropbull \alphamult
\ssdropbull \alphamult
\ssdropbull \alphamult
\ssdropbull \alphamult
\ssgoto{d} \alphamult
\ssdropbull \alphamult
\ssdropbull \alphamult
\ssdropbull \alphamult
\ssdropbull \alphamult
}{$H^*(\weier, \mbb Z/3)[\Delta^{-1}]$.}

In this picture every $\oslash$ denotes a copy of $\mbb Z/3[a_2^6\Delta^{-1}]$ and a $\bullet$ still denotes a copy of $\mbb Z/3$. We construct $H^*(\weier, \mbb Z/3)[a_2^{-1},\Delta^{-1}]$ by inverting $\Delta$ in $H^*(\weier, \mbb Z/3)[a_2^{-1}]$.

\ret{-8...32}{2}{
\ssmoveto {-8} {0}
\ssdrop{\otimes} \ssmove 4 {0}
\ssdrop{\otimes} \ssdroplabel[U]{a_2^{-1}}\ssmove 4 {0}
\ssdrop{\otimes} \ssmove 4 {0}
\ssdrop{\otimes} \ssdroplabel[L]{a_2} \ssmove 4 {0}
\ssdrop{\otimes} \ssmove 4 {0}
\ssdrop{\otimes} \ssmove 4 {0}
\ssdrop{\otimes} \ssmove 4 {0}
\ssdrop{\otimes} \ssmove 4 {0}
\ssdrop{\otimes} \ssmove 4 {0}
\ssdrop{\otimes} \ssmove 4 {0}
\ssdrop{\otimes} \ssmove 4 {0}
\ssdrop{\otimes} \ssmove 4 {0}
}{$H^*(\weier, \mbb Z/3)[a_2^{-1}, \Delta^{-1}]$.}

Every $\otimes$ represents a copy of $\mbb Z/3[a_2^6\Delta^{-1}, a_2^{-6}\Delta]$. Now we can compute the kernel and the cokernel of the difference map $D$ used in the Mayer-Vietoris long exact sequnce (\ref{eq:vie}). That isn't necessary, though, because we are only interested in very specific parts of the picture. More specifically we see that $\coker^{0, -20}D=0$, so
$$H^{1,-20}(\mell,\mbb Z/3)\simeq\ker^{1,-20}D\simeq\mbb Z/3.$$
It follows by the Universal Coefficient Theorem that $H^{1,-20}(\mell)\simeq\mbb Z_{(3)}$ and by a similar computation that $$H^{1,-24 j+4}(\mell)\simeq\mbb Z_{(3)}^j,\qquad\textrm{ for all } j\in\mbb N.$$

\subsection{The elliptic spectral sequence}

We have resolved the issue stated in (\ref{eq:is3}), so let us look at the multiplicative structure of $H^*(\mell)$ concerning the groups $H^{1,-24 j+4}(\mell)$. These groups are part of the short exact sequence
$$0\longrightarrow H^{0, 24(j-1)}(\weier)\longrightarrow H^{1,-24 j+4}(\mell)\longrightarrow\alpha\Delta^{-j}\cdot\mbb Z/3\longrightarrow0,$$ of which we now know that it looks like $$0\longrightarrow\mbb Z_{(3)}^j\longrightarrow\mbb Z_{(3)}^j\longrightarrow\mbb Z/3\longrightarrow0.$$ The first map is multiplication by 3 in one of the copies of $\mbb Z_{(3)}$ and the second map is the projection to that copy modulo 3. When we identify $H^{0, 24(j-1)}(\weier)$ with its image in $H^{1,-24 j+4}(\mell)$, then a basis element $t\in H^{0, 24(j-1)}(\weier)$ is sent to an element, denoted as before by $[c_4^{-1}c_6\Delta^{-1}t^{-1}]$, which is three times the generator of $H^{1,-24 j+4}(\mell)$. This is why we write $[\frac13c_4^{-1}c_6\Delta^{-1}t^{-1}]$ for the generator of $H^{1,-24 j+4}(\mell)$ and we draw a $\square$ with a $\frac13$ in it for this bidegree. On the other hand, the generator of $H^{1,-24 j+4}(\mell)$ is sent to $\alpha\Delta^{-j}$, which is why we also denote it by $\langle\alpha\Delta^{-j}\rangle$. In general, we will use the notation $\langle xy\rangle$ for an element of $H^*(\mell)$ when $x$ or $y$ is not in $H^*(\mell)$ and when this element corresponds to the product $xy$ in the kernel of $D$.


\begin{figure}[H]
\res{
\begin{sseq}[xlabelstep=4]{-52...32}{0...15}
\ssdrop{\square}
\ssmove 4 {0}
\movefour
\ssdroplabel{c_4\,}
\movefour
\ssdroplabel{c_6\,}
\movefour
\movefour
\movefour
\ssdroplabel{\Delta}
\ssname{d}
\ssdrop{\square}
\movefour
\movefour
\ssdrop{\square}
\movefour
\ssdrop{\square}
\movefour
\ssdrop{\square}

\ssmoveto {-21} 1
\ssdropboxed{\resizebox{1.3mm}{!}{$\,\frac13\,$}}
\ssdroplabel[RU]{\hspace{-2mm}\mathtt{\tiny [\frac13c_4^{-1}c_6\Delta^{-1}]}}
\ssmove {-4} 0
\minfour
\minfour
\minfour
\minfour
\ssmove {-4} 0
\ssdrop{\square}
\ssname{-d}
\ssdroplabel[LU]{\mathtt{\tiny [c_{4}^{-4}c_{6}\Delta^{-1}]}\hspace{-2mm}}
\ssdropboxed{\resizebox{1.3mm}{!}{$\,\frac13\,$}}
\ssdroplabel[RD]{\hspace{-2mm}\mathtt{\tiny [\frac13c_4^{-1}c_6\Delta^{-2}]}}
\minfour
\minfour

\ssgoto{d}\alphamult
\ssdropbull \alphamult

\ssmoveto 0 0
\ssline 3 1
\ssdropbull
\ssdroplabel[U]{\alpha}
\ssmove 7 1
\ssdropbull \ssdroplabel{\beta} \alphamult
\ssdropbull \alphamult
\ssdropbull \alphamult

\ssmoveto {-14} 2
\ssdropbull \alphamult
\ssdropbull \alphamult
\ssdropbull \alphamult
\ssdropbull \alphamult
\ssdropbull \alphamult

\ssmoveto {-38} 2
\ssdropbull \alphamult
\ssdropbull \alphamult
\ssdropbull \alphamult
\ssdropbull \alphamult
\ssdropbull \alphamult
\ssdropbull \alphamult
\ssdropbull \alphamult
\ssdropbull \alphamult

\ssmoveto {-62} 2
\ssdropbull \alphamult
\ssdropbull \alphamult
\ssdropbull \alphamult
\ssdropbull \alphamult
\ssdropbull \alphamult
\ssdropbull \alphamult
\ssdropbull \alphamult
\ssdropbull \alphamult

\ssmoveto {-86} 2
\ssdropbull \alphamult
\ssdropbull \alphamult
\ssdropbull \alphamult
\ssdropbull \alphamult
\ssdropbull \alphamult
\ssdropbull \alphamult
\ssdropbull \alphamult
\ssdropbull \alphamult

\ssmoveto {-110} 2
\ssdropbull \alphamult
\ssdropbull \alphamult
\ssdropbull \alphamult
\ssdropbull \alphamult
\ssdropbull \alphamult
\ssdropbull \alphamult
\ssdropbull \alphamult
\ssdropbull \alphamult
\end{sseq}}
\caption{$H^*(\mell)$  localized at 3.}
\end{figure}\vspace{-.7cm}

Now that we have $H^*(\mell)$, we proceed to compute the differentials in the elliptic spectral sequence that abuts to $\pi_*\Tmf$. We make use, again, of the computations in \cite{bau}. The differentials which are calculated by Bauer for $H^*(\weier)$ are also found periodically repeated in $H^*(\weier)[\Delta^{-1}]$. 
We can take the pull-back of these differentials via the map $$\rho: H^i(\mell)\longrightarrow H^i(\weier)[c_4^{-1}]\oplus H^i(\weier)[\Delta^{-1}].$$ 
Define the $\mathbb Z$-module $T^+$, which is depicted in Figure \ref{fig:3pos}, as 
\begin{multline}\label{eq:t+}
T^+=\mathbb Z_{(3)}\left\{c_4^jc_6^k\Delta^l \right\}\hspace{-2mm}{\tiny{\begin{array}{l}j\geq0 \\ k\in\{0,1\} \\ l\in\{0,1,2\}\\ j+k>0 \end{array}}}\oplus\mathbb Z_{(3)}\oplus3\Delta\mathbb Z_{(3)}\oplus3\Delta^2\mathbb Z_{(3)}%
\oplus\mathbb Z/3\mathbb Z\left\{\alpha,\beta,\alpha\beta,\beta^2,\beta^3,\beta^4,\alpha\Delta,\alpha\beta\Delta\right\},
\end{multline}
with $c_4^jc_6^k\Delta^l$ in degree $8j+12k+24l$ and with $\alpha^a\beta^b\Delta^l$ in degree $3a+10b+24l$. 
Define the $\mathbb Z$-module $T^-$, which is depicted in Figure \ref{fig:3neg}, as 
\begin{multline}\label{eq:t-}
T^-=\mathbb Z_{(3)}\left\{[c_4^jc_6^k\Delta^l]\right\}\hspace{-2mm}{\tiny{\begin{array}{l}j\leq-1 \\ k\in\{0,1\} \\ l\in\{-3,-2-1\}\\ j+k<0 \end{array}}}\oplus[c_4^{-1}c_6\Delta^{-1}]\mathbb Z_{(3)}\oplus[\frac13c_4^{-1}c_6\Delta^{-2}]\mathbb Z_{(3)}\oplus[\frac13c_4^{-1}c_6\Delta^{-3}]\mathbb Z_{(3)}\\%
\oplus\mathbb Z/3\mathbb Z\left\{\langle\beta\Delta^l\rangle,\langle\alpha\beta\Delta^l\rangle,\langle\beta^2\Delta^l\rangle,\langle\alpha\beta^2\Delta^l\rangle,\langle\beta^3\Delta^l\rangle, \langle\beta^4\Delta^l\rangle,\langle\alpha\beta\Delta^{l+1}\rangle,\langle\alpha\beta^2\Delta^{l+1}\rangle\right\},
\end{multline}
with $[c_4^jc_6^k\Delta^l]$ in degree $8j+12k+24l-1$ and with $\langle\alpha^a\beta^b\Delta^l\rangle$ in degree $3a+10b+24l$. In this section we have proven the following theorem.

\begin{theorem}\label{th:3}
Additively $\pi_*\Tmf_{(3)}$ is given by 
$$\pi_*\Tmf_{(3)}=(\mathbb Z[\Delta^3]\otimes T^+)\oplus(\mathbb Z[\Delta^{-3}]\otimes T^-),$$
with $T^+$ and $T^-$ as defined above in equations (\ref{eq:t+}) and (\ref{eq:t-}).

Multiplication of elements $c_4^jc_6^k\Delta^l$, $[c_4^jc_6^k\Delta^l]$, $\alpha^a\beta^b\Delta^l$ and $\langle\alpha^a\beta^b\Delta^l\rangle$ is done without taking brackets into account, and then adding brackets such that the result is an element of $\pi_*\Tmf_{(3)}$. If this cannot be done, then the result is 0, except for 
\begin{align}
\alpha\cdot \alpha\Delta&=\beta^3,\label{eq:aad3}\\
\alpha\cdot [\frac13c_4^{-1}c_6\Delta^l]=\alpha\langle\alpha\Delta^l\rangle&=\langle\beta^3\Delta^{l-1}\rangle  \quad\textrm{with}\quad l<-1,l\equiv1\mod 3, \label{eq:a[]3}\\
c_6^2&=c_4^3-1728\Delta,
\end{align}
and multiplicative consequences of these equations. Equations \ref{eq:aad3} and \ref{eq:a[]3} are also depicted in the figures.
\end{theorem}

\begin{figure}[H]
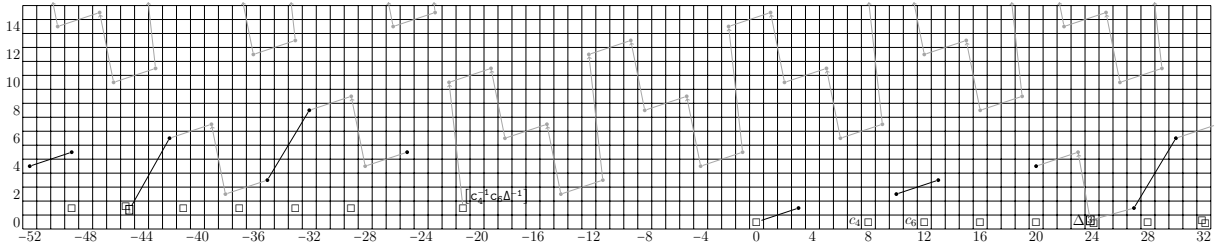

\res{
\begin{sseq}[xlabelstep=4]{-52...32}{0...15}
\ssdrop{\square}
\ssmove 4 {0}
\movefour
\ssdroplabel{c_4\,}
\movefour
\ssdroplabel{c_6\,}
\movefour
\movefour
\ssmove 4 {0}
\ssdropboxed{\scriptstyle 3}
\ssname{del}
\ssdroplabel{\Delta}
\ssdrop{\square}
\movefour
\movefour
\ssdrop{\square}
\movefour
\ssdrop{\square}
\movefour
\ssdrop{\square}

\ssmoveto {-21} 1
\ssdrop{\square}
\ssdroplabel[RU]{\hspace{-1mm}\mathtt{\tiny [c_4^{-1}c_6\Delta^{-1}]}}
\ssmove {-4} 0
\minfour
\minfour
\minfour
\minfour
\ssmove {-4} 0
\ssdrop{\square}
\ssdropboxed{\resizebox{1.3mm}{!}{$\,\frac13\,$}}
\ssname{mindel}
\minfour
\minfour

\ssmoveto 0 0
\alphamult
\ssdropbull \alphamult
\ssdropbull \gray \alphamult \black
\ssdropbull \gray \alphamult \black
\ssdropbull

\ssgoto{del}
\gray \ssarrow {-1} 5
\ssgoto{del}
\ssline 3 1
\black \ssdropbull \ssline 3 5

\ssmoveto {-21} 1
\threestep
\threestep
\threestep
\threestep
\threestep
\gray \ssarrow {-1} 9
\ssdropbull \ssline 3 1
\ssdropbull \ssline 1 {-5}
\ssdropbull \black

\ssmoveto {-52} 4
\ssdropbull
\ssmove 3 1 \ssdropbull \ssstroke
\ssmove 7 1
\ssdropbull \gray
\ssline 3 1
\ssdropbull
\ssmove 7 1 \black
\ssdropbull \gray
\ssline 3 1
\ssdropbull
\ssmove 7 1

\ssmoveto {-38} 2
\gray \ssdropbull \firstrightatthree
\gray \ssdropbull \ssarrow {-1} 5 \ssmove 1 {-5}
\ssline 3 1
\black \ssdropbull

\ssgoto{mindel}
\ssline 3 5

\ssmoveto {-50} {14}
\gray
\ssdropbull
\ssarrow {-1} 5
\ssdropbull
\ssmove 1 {-5}
\ssline 3 1
\ssdropbull
\ssline[arrowfrom] 1 {-5}
\ssdropbull
\ssline 3 1
\ssdropbull
\threestep
\threestep
\threestep
\threestep
\end{sseq}}
\caption{$\pi_*\Tmf$  localized at 3.}
\end{figure}\vspace{-.7cm}

The next two diagrams are close-ups of the positive degrees, and of the negative degrees. In filtration zero the first diagram has generators $c_4$, $c_6$, and $\Delta$. In the higher filtrations, the diagram is repetitive with period 72 generated by $\Delta^3$. In filtrations higher than 1, the second diagram is repetitive with period $-72$.

\begin{figure}[H]
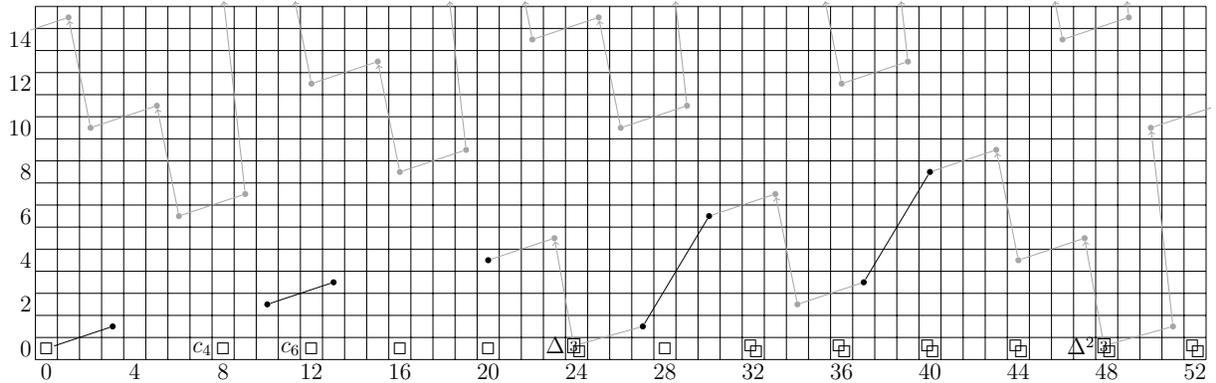

\res{
\begin{sseq}[xlabelstep=4]{0...52}{0...15}
\ssdrop{\square}
\ssmove 4 {0}
\movefour
\ssdroplabel{c_4\,}
\movefour
\ssdroplabel{c_6\,}
\movefour
\movefour
\ssmove 4 {0}
\ssdropboxed{\scriptstyle 3}
\ssname{del}
\ssdroplabel{\Delta}
\ssdrop{\square}
\movefour
\movefour
\ssdrop{\square}
\movefour
\ssdrop{\square}
\movefour
\ssdrop{\square}
\movefour
\ssdrop{\square}
\ssmove 4 {0}
\ssdropboxed{\scriptstyle 3}
\ssname{twodel}
\ssdroplabel{\Delta^2}
\ssdrop{\square}
\movefour
\ssdrop{\square}

\ssmoveto {-21} 1
\ssdrop{\square}
\ssmove {-4} 0
\minfour
\minfour
\minfour
\minfour
\ssmove {-4} 0
\ssdrop{\square}
\ssdropboxed{\resizebox{1.3mm}{!}{$\,\frac13\,$}}
\ssname{mindel}
\minfour
\minfour

\ssmoveto 0 0
\alphamult
\ssdropbull \alphamult
\ssdropbull \gray \alphamult \black
\ssdropbull \gray \alphamult \black
\ssdropbull \gray \alphamult
\ssdropbull \alphamult
\ssdropbull \alphamult

\ssgoto{del}
\gray \ssarrow {-1} 5
\ssgoto{del}
\ssline 3 1
\black \ssdropbull \ssline 3 5
\ssmove 3 1
\gray \ssline[arrowfrom] 1 {-5}
\ssdropbull
\ssline 3 1
\black \ssdropbull
\ssline 3 5
\ssmoveto {44} 4
\gray \ssdropbull
\ssarrow {-1} 5
\ssmove 1 {-5}
\alphamult

\ssgoto{twodel}
\gray \ssarrow {-1} 5
\ssgoto{twodel}
\ssline 3 1
\ssdropbull
\ssarrow {-1} 9

\ssmoveto {9} 7
\gray \ssdropbull
\ssline {-3} {-1}
\ssdropbull
\ssarrow {-1} 5
\ssdropbull
\ssline {-3} {-1}
\ssdropbull
\ssarrow {-1} 5
\ssdropbull
\ssline {-3} {-1}
\ssdropbull
\ssmoveto {9} 7
\threestep
\threestep
\threestep
\threestep
\threestep
\threestep

\ssmoveto {-52} 4
\ssdropbull
\ssmove 3 1 \ssdropbull \ssstroke
\ssmove 7 1
\ssdropbull \gray
\ssline 3 1
\ssdropbull
\ssmove 7 1 \black
\ssdropbull \gray
\ssline 3 1
\ssdropbull
\ssmove 7 1

\ssmoveto {-38} 2
\gray \ssdropbull \firstrightatthree
\gray \ssdropbull \ssarrow {-1} 5 \ssmove 1 {-5}
\ssline 3 1
\black \ssdropbull

\ssgoto{mindel}
\ssline 3 5

\ssmoveto {-50} {14}
\gray
\ssdropbull
\ssarrow {-1} 5
\ssdropbull
\ssmove 1 {-5}
\ssline 3 1
\ssdropbull
\ssline[arrowfrom] 1 {-5}
\ssdropbull
\ssline 3 1
\ssdropbull
\threestep
\threestep
\threestep
\threestep
\end{sseq}}
\caption{$\pi_*\Tmf$  localized at 3, the positive degrees.\label{fig:3pos}}
\end{figure}\vspace{-.7cm}

\begin{figure}[H]
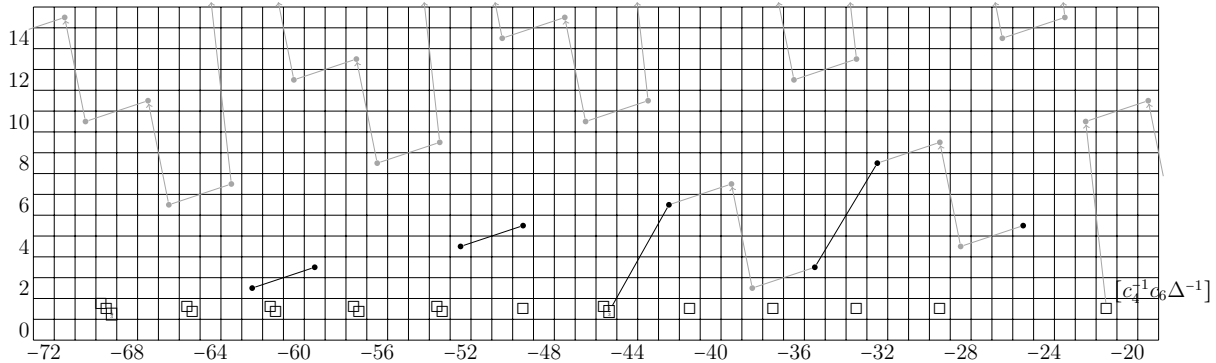

\res{
\begin{sseq}[xlabelstep=4]{-72...-19}{0...15}
\ssdrop{\square}
\ssmove 4 {0}
\movefour
\ssdroplabel{c_4\,}
\movefour
\ssdroplabel{c_6\,}
\movefour
\movefour
\ssmove 4 {0}
\ssdropboxed{3}
\ssname{del}
\ssdroplabel{\Delta}
\ssdrop{\square}
\movefour
\movefour
\ssdrop{\square}
\movefour
\ssdrop{\square}
\movefour
\ssdrop{\square}

\ssmoveto {-21} 1
\ssdrop{\square}\ssdroplabel[RU]{[c_4^{-1}c_6\Delta^{-1}]}
\ssmove {-4} 0
\minfour
\minfour
\minfour
\minfour
\ssmove {-4} 0
\ssdrop{\square}
\ssdropboxed{\resizebox{1.3mm}{!}{$\;\frac13\,$}}
\ssname{mindel}
\minfour
\minfour
\ssdrop{\square}
\minfour
\ssdrop{\square}
\minfour
\ssdrop{\square}
\minfour
\ssdrop{\square}
\ssmove {-4}{0}
\ssdrop{\square}
\ssname{mintwodel}
\ssdrop{\square}
\ssdropboxed{\resizebox{1.3mm}{!}{$\;\frac13\,$}}
\minfour
\ssdrop{\square}

\ssmoveto {-62} 2
\ssdropbull \alphamult
\ssdropbull \alphamult
\ssdropbull \gray \alphamult
\black \ssdropbull \gray \alphamult

\ssmoveto {-38} 2
\gray \ssdropbull \firstrightatthree
\gray \ssdropbull \ssarrow {-1} 5 \ssmove 1 {-5}
\ssline 3 1
\black \ssdropbull

\ssgoto{mindel}
\ssline 3 5

\ssmoveto {-21} 1
\gray \ssarrow {-1} 9
\ssdropbull \ssline 3 1
\ssdropbull \ssline[arrowfrom] 1 {-5}\ssdropbull

\ssmoveto {-63} 7
\gray\ssdropbull
\ssline {-3} {-1}
\ssdropbull
\ssarrow {-1} 5
\ssdropbull
\ssline {-3} {-1}
\ssdropbull
\ssarrow {-1} 5
\ssdropbull
\ssline {-3} {-1}
\ssdropbull
\ssmoveto {-63} 7
\threestep
\threestep
\threestep
\threestep
\threestep
\threestep
\threestep
\threestep
\threestep

\end{sseq}}
\caption{$\pi_*\Tmf$  localized at 3, the negative degrees.\label{fig:3neg}}
\end{figure}\vspace{-.7cm}

\section{$\pi_*\Tmf$ localized at the prime 2}\label{sec:2}

Now that we have done the computation localized at 3, we continue with the more elaborate case of the prime 2. The approach is the same, but the details are more involved, especially towards the end of the computation. Again we start with the picture of $H^*(\weier)$ taken from \cite{bau}.

\ret{41}{9}{
\ssdrop{\square}\ssline 1 1 \ssdropbull \ssdroplabel{h_1}\ssarrow 1 1\ssdrop{}\ssmove {-2} {-2}
\ssline 3 1 \ssdropbull \ssdroplabel[RD]{h_2}
\ssline 3 1 \ssdropbull
\ssline 3 1 \ssdropbull
\ssmove {-4} {-2} \ssdropbull \ssdroplabel{b} \ssline 1 1 \ssdropbull \ssarrow 1 1 \ssdrop{} \ssmove {-1} {-1}
\ssmove 2 {0} \ssdropbull \ssdroplabel{c}
\ssline 1 1 \ssmove {-6}{-2}
\ssdropextension \ssmove 5 1
\ssmove 7 1 \ssdropbull
\ssline {-1}{-1} \ssdropbull \ssdroplabel{d}
\ssline 3 1 \ssdropbull
\ssmove 3 1 \ssdropbull \cmult \ssmove{-8}{-2} \ssline 1 1 \ssdropbull \ssarrow 1 1\ssdrop{}\ssmove {-2} {-2}
\ssdropextension \ssdroplabel[RD]{g} \ssdropextension \ssname{g2}
\ssmove {-3} {-1} \ssstroke
\ssmove {11} 3
\ssmove 7 1 \ssdropbull
\ssline {-1}{-1} \ssdropbull
\ssline 3 1 \ssdropbull
\ssmove 3 1 \ssdropbull \ssline 3 1 \ssdropbull \ssmove {-3} {-1}
\ssvoidarrow 1 1
\ssdropextension \ssdropextension
\ssmove {-3} {-1} \ssstroke
\ssmoveto {0}{0} \ssmove 8 {0}
\ssdrop{\square}\ssdroplabel{c_4} \ssline 1 1 \ssdropbull \ssarrow 1 1\ssdrop{}\ssmove {-2} {-2}
\ssmove 4 {0}
\ssdrop{\square}\ssdroplabel{c_6} \ssline 1 1 \ssdropbull \ssarrow 1 1\ssdrop{}\ssmove {-2} {-2}
\ssmove 4 {0}
\ssdrop{\square}\ssline 1 1 \ssdropbull\ssarrow 1 1\ssdrop{}\ssmove {-2} {-2}
\ssmove 4 {0}
\ssdrop{\square}\ssline 1 1 \ssdropbull\ssarrow 1 1\ssdrop{}\ssmove {-2} {-2}
\ssmove 4 {0}
\ssdrop{\square}\ssdroplabel{\Delta}\ssname{del} \ssline 1 1 \ssdropbull\ssarrow 1 1\ssdrop{}\ssmove {-2} {-2}
\cmult\ssmove{-8}{-2}
\ssdrop{\square}\ssline 1 1 \ssdropbull\ssarrow 1 1\ssdrop{}\ssmove {-2} {-2}
\ssmove 4 {0}
\ssdrop{\square}\ssline 1 1 \ssdropbull\ssarrow 1 1 \ssdrop{}\ssmove {-2} {-2}
\ssmove 4 {0}
\ssdrop{\square}\ssline 1 1 \ssdropbull\ssarrow 1 1\ssdrop{}\ssmove {-2} {-2}
\ssdrop{\square}\ssline 1 1 \ssdropbull\ssarrow 1 1\ssdrop{}\ssmove {-2} {-2}
\ssmove 4 {0}
\ssdrop{\square}\ssline 1 1 \ssdropbull\ssarrow 1 1\ssdrop{}\ssmove {-2} {-2}
\ssdrop{\square}\ssline 1 1 \ssdropbull\ssarrow 1 1\ssdrop{}\ssmove {-2} {-2}
\ssmove 4 {0}
\ssdrop{\square}\ssvoidarrow 1 1
\ssdrop{\square}\ssvoidarrow 1 1
}{$H^*(\weier)$ localized at 2.\label{fig:p21}}

Recall that from now on a $\bullet$ denotes a copy of $\mbb Z/2$ and a $\square$ denotes a copy of $\mbb Z_{(2)}$. Furthermore we find some circled dots; one circle means a copy $\mbb Z/4$, two circles mean a copy $\mbb Z/8$. Lines of slope 1 denote a multiplication with $h_1$, while lines of slope $\frac13$ denote a multiplication with $h_2$.

There are several remarks that should be made about this picture. In $H^*(\weier)$ localized at $2$ we know that $t\Delta=0$ implies $t=0$ and that $tg=0$ implies $t=0$. Secondly observe that the result of multiplying $dh_2$ with $h_2$ is not $g$, represented by the bullet, but $4g$ represented by the outermost circle. The result of multiplication by $c_4$ should be made clear in a few cases. We have \begin{equation}\label{eq:c41} c_4h_2=c_4c=c_4d=0,\qquad c_4b=c_6h_1.\end{equation} Furthermore we find
\begin{equation}\label{eq:c42}c_4g=\Delta h_1^4,\qquad c_42g=2\Delta h_1^4=0.\end{equation} 

\subsection{The cohomology of $\mell$}

With the above remarks it is easy to compute $H^*(\weier)[c_4^{-1}]$.

\ret{-8...32}{3}{
\ssmoveto {-8} {0}
\ssdrop{\boxminus}\ssline 1 1\ssdrop{\resizebox{2mm}{!}{$\ominus$}}\ssvoidarrow 1 1\ssmove {-1}{-1}
\ssmove 4 {0}
\ssdrop{\boxminus}\ssdroplabel{c_4^{-2}c_6}\ssline 1 1\ssdrop{\resizebox{2mm}{!}{$\ominus$}}\ssvoidarrow 1 1\ssmove {-1}{-1}
\ssmove 4 {0}
\ssdrop{\boxminus}\ssline 1 1\ssdrop{\resizebox{2mm}{!}{$\ominus$}}\ssvoidarrow 1 1\ssmove {-1}{-1}
\ssmove 4 {0}
\ssdrop{\boxminus}\ssdroplabel{c_4^{-1}c_6}\ssline 1 1\ssdrop{\resizebox{2mm}{!}{$\ominus$}}\ssvoidarrow 1 1\ssmove {-1}{-1}
\ssmove 4 {0}
\ssdrop{\boxminus}\ssdroplabel{c_4}\ssline 1 1\ssdrop{\resizebox{2mm}{!}{$\ominus$}}\ssvoidarrow 1 1\ssmove {-1}{-1}
\ssmove 4 {0}
\ssdrop{\boxminus}\ssdroplabel{c_6}\ssline 1 1\ssdrop{\resizebox{2mm}{!}{$\ominus$}}\ssvoidarrow 1 1\ssmove {-1}{-1}
\ssmove 4 {0}
\ssdrop{\boxminus}\ssline 1 1\ssdrop{\resizebox{2mm}{!}{$\ominus$}}\ssvoidarrow 1 1\ssmove {-1}{-1}
\ssmove 4 {0}
\ssdrop{\boxminus}\ssline 1 1\ssdrop{\resizebox{2mm}{!}{$\ominus$}}\ssvoidarrow 1 1\ssmove {-1}{-1}
\ssmove 4 {0}
\ssdrop{\boxminus}\ssline 1 1\ssdrop{\resizebox{2mm}{!}{$\ominus$}}\ssvoidarrow 1 1\ssmove {-1}{-1}
\ssmove 4 {0}
\ssdrop{\boxminus}\ssline 1 1\ssdrop{\resizebox{2mm}{!}{$\ominus$}}\ssvoidarrow 1 1\ssmove {-1}{-1}
\ssmove 4 {0}
\ssdrop{\boxminus}\ssline 1 1\ssdrop{\resizebox{2mm}{!}{$\ominus$}}\ssvoidarrow 1 1\ssmove {-1}{-1}
}{$H^*(\weier)[c_4^{-1}]$ localized at 2.}

Here every $\boxminus$ represents a copy of $\mbb Z_{(2)}[c_4^{-3}\Delta]$ and $\ominus$ a copy of $\mbb Z/2[c_4^{-3}\Delta]$. Note that $b=c_4^{-1}c_6h_1$, so $b$ is presented in this picture whilst it is not explicitly drawn. The same is true for $g$

In $H^*(\weier)$ the equation $t\Delta=0$ implies $t=0$, so the map $i_\Delta:H^*(\weier)\longrightarrow H^*(\weier)[\Delta^{-1}]$ is an inclusion.

\ret{-8...32}{12}{
\ssdrop{\boxvert}
\ssline 3 1 \ssdropbull
\ssline 3 1 \ssdropbull
\ssline 3 1 \ssdropbull
\ssmove {-4} {-2} \ssdropbull \ssdroplabel[L]{\scriptstyle b}\ssline 1 1 \ssdropbull \ssarrow 1 1\ssdrop{}\ssmove {-1} {-1}
\ssmove 2 {0} \ssdropbull
\ssline 1 1 \ssmove {-6}{-2}
\ssdropextension \ssmove 5 1
\ccmult
\ssmoveto {-16} 2 \ccmult \ccmult
\ssmove 7 1 \ssdropbull
\ssline {-1}{-1} \ssdropbull
\ssline 3 1 \ssdropbull
\ssmoveto {-20} 6 \ccmult
\ssmove 7 1 \ssdropbull
\ssline {-1}{-1} \ssdropbull
\ssline 3 1 \ssdropbull
\ssline 3 1 \ssdropbull
\ssmoveto {24} {0} \ssdrop{\boxvert}\ssdroplabel{\Delta}\cmult
\ssmoveto {-8} {0}
\ssdrop{\boxvert}\ssline 1 1\ssdrop{\resizebox{2mm}{!}{$\overt$}}\ssarrow 1 1\ssdrop{}\ssmove {-2}{-2}
\ssmove 4 {0}
\ssdrop{\boxvert}\ssline 1 1\ssdrop{\resizebox{2mm}{!}{$\overt$}}\ssarrow 1 1\ssdrop{}\ssmove {-2}{-2}
\ssmove 4 {0}
\ssline 1 1\ssdrop{\resizebox{2mm}{!}{$\overt$}}\ssarrow 1 1\ssdrop{}\ssmove {-2}{-2}
\ssmove 4 {0}
\ssdrop{\boxvert}\ssdroplabel[R]{\scriptstyle c_4^2c_6\Delta^{-1}} \ssline 1 1\ssdrop{\resizebox{2mm}{!}{$\overt$}}\ssarrow 1 1\ssdrop{}\ssmove {-2}{-2}
\ssmove 4 {0}
\ssdrop{\boxvert}\ssline 1 1\ssdrop{\resizebox{2mm}{!}{$\overt$}}\ssarrow 1 1\ssdrop{}\ssmove {-2}{-2}
\ssmove 4 {0}
\ssdrop{\boxvert}\ssline 1 1\ssdrop{\resizebox{2mm}{!}{$\overt$}}\ssarrow 1 1\ssdrop{}\ssmove {-2}{-2}
\ssmove 4 {0}
\ssdrop{\boxvert}\ssline 1 1\ssdrop{\resizebox{2mm}{!}{$\overt$}}\ssarrow 1 1\ssdrop{}\ssmove {-2}{-2}
\ssmove 4 {0}
\ssdrop{\boxvert}\ssline 1 1\ssdrop{\resizebox{2mm}{!}{$\overt$}}\ssarrow 1 1\ssdrop{}\ssmove {-2}{-2}
\ssmove 4 {0}
\ssline 1 1\ssdrop{\resizebox{2mm}{!}{$\overt$}}\ssarrow 1 1\ssdrop{}\ssmove {-2}{-2}
\ssmove 4 {0}
\ssdrop{\boxvert}\ssline 1 1\ssdrop{\resizebox{2mm}{!}{$\overt$}}\ssarrow 1 1\ssdrop{}\ssmove {-2}{-2}
\ssmove 4 {0}
\ssdrop{\boxvert}\ssline 1 1\ssdrop{\resizebox{2mm}{!}{$\overt$}}\ssarrow 1 1\ssdrop{}\ssmove {-2}{-2}
}{$H^*(\weier)[\Delta^{-1}]$ localized at 2.}

Now we have $\boxvert$ for a copy of $\mbb Z_{(2)}[c_4^3\Delta^{-1}]$ and $\overt$ for a copy of $\mbb Z/2[c_4^3\Delta^{-1}]$. This time $$b\not\in h_1c_4^2c_6\Delta^{-1}\cdot\mbb Z/2[c_4^3\Delta^{-1}],$$
so we draw $b$ separately. The same argument holds for $g$.

For $H^*(\weier)[c_4^{-1},\Delta^{-1}]$ we get an easy description again via $H^*(\weier)[c_4^{-1}]$.

\ret{-8...32}{3}{
\ssmoveto {-8} {0}
\ssdrop{\boxplus}\ssline 1 1\ssdrop{\resizebox{2mm}{!}{$\oplus$}}\ssvoidarrow 1 1\ssmove {-1}{-1}
\ssmove 4 {0}
\ssdrop{\boxplus}\ssline 1 1\ssdrop{\resizebox{2mm}{!}{$\oplus$}}\ssvoidarrow 1 1\ssmove {-1}{-1}
\ssmove 4 {0}
\ssdrop{\boxplus}\ssline 1 1\ssdrop{\resizebox{2mm}{!}{$\oplus$}}\ssvoidarrow 1 1\ssmove {-1}{-1}
\ssmove 4 {0}
\ssdrop{\boxplus}\ssline 1 1\ssdrop{\resizebox{2mm}{!}{$\oplus$}}\ssvoidarrow 1 1\ssmove {-1}{-1}
\ssmove 4 {0}
\ssdrop{\boxplus}\ssline 1 1\ssdrop{\resizebox{2mm}{!}{$\oplus$}}\ssvoidarrow 1 1\ssmove {-1}{-1}
\ssmove 4 {0}
\ssdrop{\boxplus}\ssline 1 1\ssdrop{\resizebox{2mm}{!}{$\oplus$}}\ssvoidarrow 1 1\ssmove {-1}{-1}
\ssmove 4 {0}
\ssdrop{\boxplus}\ssline 1 1\ssdrop{\resizebox{2mm}{!}{$\oplus$}}\ssvoidarrow 1 1\ssmove {-1}{-1}
\ssmove 4 {0}
\ssdrop{\boxplus}\ssline 1 1\ssdrop{\resizebox{2mm}{!}{$\oplus$}}\ssvoidarrow 1 1\ssmove {-1}{-1}
\ssmove 4 {0}
\ssdrop{\boxplus}\ssline 1 1\ssdrop{\resizebox{2mm}{!}{$\oplus$}}\ssvoidarrow 1 1\ssmove {-1}{-1}
\ssmove 4 {0}
\ssdrop{\boxplus}\ssline 1 1\ssdrop{\resizebox{2mm}{!}{$\oplus$}}\ssvoidarrow 1 1\ssmove {-1}{-1}
\ssmove 4 {0}
\ssdrop{\boxplus}\ssline 1 1\ssdrop{\resizebox{2mm}{!}{$\oplus$}}\ssvoidarrow 1 1\ssmove {-1}{-1}
}{$H^*(\weier)[c_4^{-1},\Delta^{-1}]$ localized at 2.}

Every $\boxplus$ denotes a copy of $\mbb Z_{(2)}[c_4^{-3}\Delta, c_4^3\Delta^{-1}]$ and every $\oplus$ denotes a copy of $\mbb Z/2[c_4^{-3}\Delta, c_4^3\Delta^{-1}]$. There is also a map $i_{c_4}$ from $H^*(\weier)[\Delta^{-1}]$ to $H^*(\weier)[c_4^{-1},\Delta^{-1}]$, which we will need for the computation of $H^*(\mell)$. The $c_4$-multiplications given in the equations (\ref{eq:c41}) and (\ref{eq:c42}) imply that the map $i_{c_4}$ sends $h_2$, $c$, $d$, $2g$ to 0, sends $b$ to $h_1c_4^{-1}c_6$, and sends $g$ to $h_1^4c_4^{-1}\Delta$.

Now we are ready to compute $H^*(\mell)$ using the Mayer-Vietoris long exact sequence (\ref{eq:vie}) \begin{equation}\label{eq:vie2}\cdots\stackrel{D=}{\stackrel{(i_\Delta,-i_{c_4})}{\longrightarrow}}H^{i-1}(\weier)[c_4^{-1},\Delta^{-1}]\stackrel{d}{\longrightarrow}H^i(\mell)\stackrel{\resizebox{6.5mm}{!}{$\begin{pmatrix} i_{c_4}\\ i_{\Delta}\end{pmatrix}$}}{\longrightarrow}H^i(\weier)[c_4^{-1}]\oplus H^i(\weier)[\Delta^{-1}]\stackrel{D=}{\stackrel{(i_\Delta,-i_{c_4})}{\longrightarrow}}\cdots.\end{equation} First we determine for every element of $H^*(\weier)[c_4^{-1}]\oplus H^*(\weier)[\Delta^{-1}]$ whether or not it is in the kernel of $D$. After that we turn to analyzing the cokernel of $D$, which is the quotient of $H^*(\weier)[c_4^{-1},\Delta^{-1}]$ by the image of $D$. 

\begin{lemma}\label{lem:ker}
There is an isomorphism between the kernel of $D$ and the subset of $H^*(\weier)[\Delta^{-1}]$ generated by the products containing $h_2$, $c$, $d$ or $2g$ and terms $b^hg^ih_1^jc_4^kc_6^l\Delta^m$ with $i+m\geq 0$.
\end{lemma}
\begin{proof}
Because the map $i_\Delta:H^*(\weier)[c_4^{-1}]\longrightarrow H^*(\weier)[c_4^{-1},\Delta^{-1}]$ is an inclusion, an element is in the kernel of $D$ precisely when it is of the form $(u,v)\in H^*(\weier)[c_4^{-1}]\oplus H^*(\weier)[\Delta^{-1}]$ with $u=i_{c_4}(v)$. This means that any element of the kernel of $D$ is of the form $(i_{c_4}(v),v)$ with $v\in H^*(\weier)[\Delta^{-1}]$ and this element of the kernel can then be identified with $v$. Hence we can view $\ker D$ as a subset of $H^*(\weier)[\Delta^{-1}]$; the subset of elements $v\in H^*(\weier)[\Delta^{-1}]$ such that $i_{c_4}(v)\in H^*(\weier)[c_4^{-1}]$. Every product in $H^*(\weier)[\Delta^{-1}]$ containing $h_2$, $c$, $d$ or $2g$ is in the kernel of $D$ because $i_{c_4}$ maps these elements to 0, which is surely an element of $H^*(\weier)[c_4^{-1}]$. What is left is to decide whether products of $b$, $g$, $h_1$, $c_4$, $c_6$ and $\Delta$ are in the kernel. Given an element of the form $b^hg^ih_1^jc_4^kc_6^l\Delta^m\in H^*(\weier)[\Delta^{-1}]$ we compute
$$i_{c_4}(b^hg^ih_1^jc_4^kc_6^l\Delta^m) = h_1^{h+4i+j}c_4^{-h-i+k}c_6^{h+l}\Delta^{i+m}.$$
This belongs to $H^*(\weier)[c_4^{-1}]$ precisely when $i+m\geq 0$. This fully determines the kernel of $D$.
\end{proof}

\begin{lemma}
The cokernel of $D$ is generated by the classes that can be represented by an element $h_1^wc_4^xc_6^y\Delta^z$ of $H^*(\weier)[c_4^{-1},\Delta^{-1}]$ whose exponents satisfy the inequalities $z<0$ and $w+4x+3y<0$.
\end{lemma}
\begin{proof}
The group $H^*(\weier)[c_4^{-1},\Delta^{-1}]$ is spanned by the elements $h_1^wc_4^xc_6^y\Delta^z$ with $w\geq0$, $x\in\mbb Z$, $y\in\{0,1\}$ and $z\in\mbb Z$. Such a basis element is in the image of $D$ precisely when it is in the image of $i_{c_4}$ or in the image of $i_\Delta$. 

We observe that $h_1^wc_4^xc_6^y\Delta^z$ is in the image of $i_\Delta$ if and only if it is an element of $H^*(\weier)[c_4^{-1}]$, i.e. precisely when $z\geq0$. 

On the other hand, $h_1^wc_4^xc_6^y\Delta^z$ is in the image of $i_{c_4}$ if and only if it can be written as $h_1^{h+4i+j}c_4^{-h-i+k}c_6^{h+l}\Delta^{i+m}$, with $i,j,k\geq0$ and $h,l,h+l\in\{0,1\}$. So if $h_1^wc_4^xc_6^y\Delta^z$ is in the image of $i_{c_4}$ then we find that $$w+4x+3y=(h+4i+j)+4(-h-i+k)+3(h+l)=j+4k+3l\geq0.$$ Conversely, if we start with $w$, $x$, $y$, $z$ such that $w+4x+3y\geq0$, then we choose
\begin{equation}\label{eq:ihjklm}
\begin{aligned}
i&=\lfloor\frac{w}{4}\rfloor, \\
h&=\begin{cases} 1 \qquad\textrm{ if } w-4i>0 \textrm{ and } y=1, \\ 0 \qquad\textrm{ else, }\end{cases}\\
j&=w-4i-h, \\
k&=x+h+i,\\
l&=y-h,\\
m&=z-i.\\
\end{aligned}
\end{equation}
These coefficients happen to satisfy the property $h_1^{h+4i+j}c_4^{-h-i+k}c_6^{h+l}\Delta^{i+m}=h_1^wc_4^xc_6^y\Delta^z$. Now we claim that \begin{equation}\label{eq:claim}\begin{aligned} i,j,k&\geq0,\\ h,l,h+l&\in\{0,1\},\end{aligned}\end{equation} which would imply that $b^hg^ih_1^jc_4^kc_6^l\Delta^m$ is an element of $H^*(\weier)[\Delta^{-1}]$ and thence places $h_1^wc_4^xc_6^y\Delta^z$ in the image of $i_{c_4}$. 

Let us check the claims in (\ref{eq:claim}) using the definitions in (\ref{eq:ihjklm}). It is clear that $i\geq0$, because $w\geq0$. That $h\in\{0,1\}$ is obvious. By definition of $i$ we know that $w-4i\geq0$, so if $h=0$ then $j\geq0$. If $h>0$, then we find that $w-4i>0$ and $h=1$ and in this case we get $j=w-4i-h\geq0$ as well. Now we turn to estimating $k$. First of all the given estimation $w+4x+3y\geq0$ implies that $$k=x+h+i=\frac14(4x+4h+4i)\geq\frac14(-w-3y+4h+4i).$$ If $w-4i=0$, then $h=0$ and $-w-3y+4h+4i=-3y\geq-3$. If $3\geq w-4i>0$, then $y=h$ and $-w-3y+4h+4i\geq -w+4i\geq-3$. Hence $$\frac14(-w-3y+4h+4i)\geq \frac{-3}{4}>-1,$$ so $k>-1$ and because $k$ is an integer we have $k\geq0$. Because $y\in\{0,1\}$ we have $h,l,h+l\in\{0,1\}$. We have checked all the statements of (\ref{eq:claim}), thus we have derived that $b^hg^ih_1^jc_4^kc_6^l\Delta^m\in H^*(\weier)[\Delta^{-1}]$ and $i_{c_4}(b^hg^ih_1^jc_4^kc_6^l\Delta^m)=h_1^wc_4^xc_6^y\Delta^z$ is in the image of $i_{c_4}$. We conclude that $h_1^wc_4^xc_6^y\Delta^z$ is in the image of $i_{c_4}$ if and only if $w+4x+3y\geq0$. 

Thus the image of $D$ is generated by elements $h_1^wc_4^xc_6^y\Delta^z$ with either $z\geq0$ or $w+4x+3y\geq0$. Subsequently the generators of the cokernel of $D$ correspond to elements with $z<0$ and $w+4x+3y<0$.
\end{proof}

Now we can draw $H^*(\mell)$ localized at 2 almost completely. The only obstacles remaining are those bidegrees where both the kernel and the cokernel are non-empty. When we draw the picture we observe this is the case in the following bidegrees:
\begin{equation}\label{eq:cases2}
\begin{aligned}
H^{1,-24k+4}(\mell)&=\begin{cases}\mbb (Z_{(2)})^k&\:\text{or}\\ \mbb (Z_{(2)})^k\times \mbb Z/2&\:\text{or}\\ \mbb (Z_{(2)})^k\times \mbb Z/4&,\end{cases}\\
H^{4j+5,-24k+4}(\mell)&=\begin{cases}(\mbb Z/2)^{k-1}\times\mbb Z/8&\:\text{or}\\ (\mbb Z/2)^{k}\times\mbb Z/4&,\end{cases}\\
H^{4j+3,-24k-4}(\mell)&=\begin{cases}(\mbb Z/2)^{k-1}\times\mbb Z/4&\:\text{or}\\ (\mbb Z/2)^{k+1}&,\end{cases}\\
H^{4j+3,-24k-12}(\mell)&=\begin{cases}(\mbb Z/2)^{k-1}\times\mbb Z/4&\:\text{or}\\ (\mbb Z/2)^{k+1}&,\end{cases}\\
H^{4j+2,-24k-14}(\mell)&=\begin{cases}(\mbb Z/2)^{k-1}\times\mbb Z/4&\:\text{or}\\ (\mbb Z/2)^{k+1}&.\end{cases}
\end{aligned}
\end{equation}
In Figure \ref{fig:mell2q} we have labelled these bidegrees with a question mark, but we have also drawn the answers to these questions, which will be explained afterwards.

\clearpage
\begin{figure}
\begin{minipage}{.85\linewidth}
\includepdf[pagecommand={}]{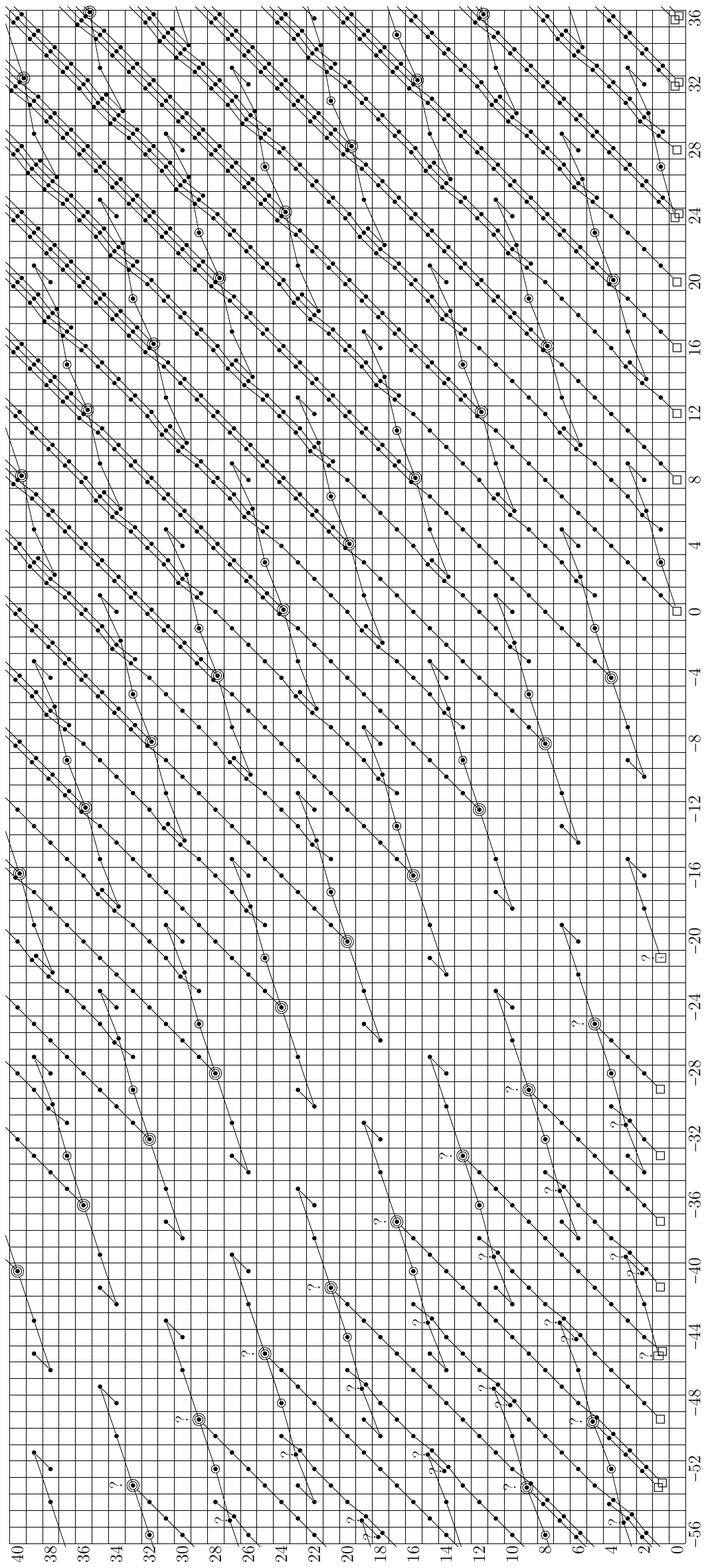}
\end{minipage}\hfill
\begin{minipage}{.15\linewidth}
\rotcaption{$H^*(\mell)$ localized at 2 with the bidegrees still to be determined marked with a question mark. \label{fig:mell2q}}
\end{minipage}
\end{figure}
\clearpage

\subsection{The mod 2 cohomology of $\mell$}

To resolve the problems listed in (\ref{eq:cases2}) we resort to the Universal Coefficient Theorem. Applying this theorem to Figure \ref{fig:p21} we get Figure \ref{fig:wmod2}. The Universal Coefficient Theorem gives all the elements of the bigraded algebra, but it does not say anything specific about the algebraic structure. Some extensions, for example, are not carried over by the functor $\Tor(-,\mbb Z/2)$. Take the surjection $\mbb Z/p^k\twoheadrightarrow\mbb Z/p^l$ with $k>l$. This map induces a commutative diagram of two short exact sequences used to define $\Tor$ 
$$\xymatrix{0 \ar[r] & \mbb Z/2 \ar[r]^{\cdot 2^k} \ar[d]^{\cdot 0} & \mbb Z/2^{k+1} \ar[r]^{\cdot 1}\ar[d]^{\cdot 1} & \mbb Z/2^k \ar[r]\ar[d]^{\cdot 1} & 0\\
0 \ar[r] & \mbb Z/2 \ar[r]^{\cdot 2^l} & \mbb Z/2^{l+1} \ar[r]^{\cdot 1} & \mbb Z/2^l \ar[r] & 0}$$
Notice that the first vertical map is the zero map, hence the induced map between $\Tor(\mbb Z/2^k,\mbb Z/2)$ and $\Tor(\mbb Z/2^l,\mbb Z/2)$ must also be zero. Another multiplication that is absent when we work with $\mbb Z/2$-coefficients is $dh_2^2=4g$ because $4g=0 \mod 2$. We will give all these absent extensions a grey color in the upcoming picture.

On the other hand, there are also some multiplications that are not apparent from the conversion to $\mbb Z/2$-coefficients. We are not going to explain these, but they will be drawn with dashed lines. The resulting picture can also be found in \cite{bau}.

\ret{41}{9}{
\ssdropbull \ssline 3 1
\ssdropbull \ssline 3 1
\ssdropbull \ssline 3 1
\ssdropbull \ssline {-1} {-1}
\ssdropbull
\ssmoveto 4 {0}
\ssdropbull \gray \ssline 3 1 \black
\ssdropbull \ssline 3 1
\ssdropbull \ssline {-1} {-1}
\ssdropbull
\ssmoveto 4 {0} \ssline[dashed] 1 1
\ssdropbull \ssline 1 1
\ssdropbull \ssarrow 1 1 \ssdrop{}
\ssmoveto 7 1\ssline[dashed] 1 1
\ssmove 7 1 \ssdropbull
\ssline {-1}{-1} \ssdropbull \ssdroplabel{d}
\ssline 3 1 \ssdropbull
\ssmove 3 1 \ssdropbull \ssdroplabel{g}
\ssline 3 1 \ssdropbull
\ssline 3 1 \ssdropbull
\ssline 3 1 \ssdropbull
\ssmove {-4} {-2} \ssdropbull \ssline 1 1 \ssdropbull \ssarrow 1 1\ssdrop{}\ssmove {-1} {-1}
\ssmove 2 {0} \ssdropbull
\ssline 1 1 \ssmove {-6}{-2} \ssmove 5 1
\ssmove{-8}{-2} \ssline 1 1 \ssdropbull \ssarrow 1 1\ssdrop{}\ssmove {-2} {-2}
\gray \ssline {-3} {-1} \black
\ssmove {-2} {-2} \ssdropbull
\ssline 1 1 \ssdropbull
\ssline[dashed] 1 1 \ssmove {-2} {-2}
\ssline 3 1 \ssdropbull
\ssline 3 1 \ssdropbull
\gray \ssline 1 1 \black \ssdropbull \ssarrow 1 1 \ssdrop{}\ssmove {-2} {-2}
\gray \ssline 3 1 \black \ssdropbull
\ssline[dashed] 1 1 \ssmove {-1} {-1}
\gray \ssline 3 1 \black \ssdropbull
\ssline[dashed] 1 1 \ssmove {-1} {-1}
\ssline 3 1 \ssdropbull
\ssline {-1} {-1} \ssdropbull
\ssmove {-3} {-1} \ssdropbull \ssline 1 1 \ssdropbull \ssarrow 1 1 \ssdrop{}
\ssmove 7 1 \ssdropbull
\ssline {-1}{-1} \ssdropbull
\ssline 3 1 \ssdropbull
\ssmove 3 1 \ssdropbull
\ssline 3 1 \ssdropbull
\ssline 3 1 \ssdropbull
\ssline 3 1 \ssdropbull
\ssmove {-4} {-2} \ssdropbull \ssline 1 1 \ssdropbull \ssline 1 1\ssdrop{}\ssmove {-1} {-1}
\ssmove 2 {0} \ssdropbull
\ssline 1 1 \ssmove {-6}{-2} \ssmove 5 1
\ssmove{-8}{-2} \ssarrow 1 1\ssdrop{}\ssmove {-1} {-1}
\gray \ssline {-3} {-1} \black
\ssmove {-2} {-2} \ssdropbull
\ssline 1 1 \ssdropbull
\ssline[dashed] 1 1 \ssmove {-2} {-2}
\ssline 3 1 \ssdropbull
\ssline 3 1 \ssdropbull
\gray \ssline 1 1 \black \ssdropbull \ssline 1 1 \ssdrop{}\ssmove {-2} {-2}
\gray \ssline 3 1 \black \ssdropbull
\ssline[dashed] 1 1 \ssmove {-1} {-1}
\gray \ssline 3 1 \black \ssdropbull
\ssline[dashed] 1 1 \ssmove {-1} {-1}
\ssline 3 1 \ssdropbull
\ssline {-1} {-1} \ssdropbull
\ssmove {-3} {-1} \ssdropbull \ssline 1 1 \ssdropbull \ssline 1 1 \ssdrop{}
\ssmoveto {24} {0}
\ssdropbull \ssdroplabel{\Delta} \ssline 3 1
\ssdropbull \ssline 3 1
\ssdropbull \ssline 3 1
\ssdropbull \ssline {-1} {-1}
\ssdropbull
\ssmove {-4} {-2}
\ssdropbull \gray \ssline 3 1 \black
\ssdropbull \ssline 3 1
\ssdropbull \ssline {-1} {-1}
\ssdropbull
\ssmove {-5} {-1} \ssline[dashed] 1 1
\ssdropbull \ssline 1 1
\ssdropbull \ssarrow 1 1 \ssdrop{}
\ssmove {0} {-2} \ssline[dashed] 1 1
\ssmove 7 1 \ssdropbull
\ssline {-1}{-1} \ssdropbull
\ssline 3 1 \ssdropbull
\ssmove 3 1 \ssdropbull
\ssline 3 1 \ssdropbull
\ssline 3 1 \ssdropbull
\ssline 3 1 \ssdropbull
\ssmove {-4} {-2} \ssdropbull \ssline 1 1 \ssdropbull \ssline 1 1\ssdrop{}\ssmove {-1} {-1}
\ssmove 2 {0} \ssdropbull
\ssline 1 1 \ssmove {-6}{-2} \ssmove 5 1
\ssmove{-8}{-2} \ssline 1 1\ssdrop{}\ssmove {-1} {-1}
\gray \ssline {-3} {-1} \black
\ssmove {-2} {-2} \ssdropbull
\ssline 1 1 \ssdropbull
\ssline[dashed] 1 1 \ssmove {-2} {-2}
\ssline 3 1 \ssdropbull
\ssline 3 1 \ssdropbull
\gray \ssline 1 1 \black \ssdropbull \ssline 1 1 \ssdrop{}\ssmove {-2} {-2}
\gray \ssline 3 1 \black \ssdropbull
\ssline[dashed] 1 1 \ssmove {-1} {-1}
\gray \ssline 3 1 \black \ssdropbull
\ssline[dashed] 1 1 \ssmove {-1} {-1}
\ssline 3 1 \ssdropbull
\ssline {-1} {-1} \ssdropbull
\ssmove {-3} {-1} \ssdropbull \ssline 1 1 \ssdropbull \ssline 1 1 \ssdrop{}
\ssmoveto {0}{0} \ssarrow 1 1 \ssdrop{} \ssmove {-1}{-1}
\ssmove 2 {0} \ssdropbull \ssdroplabel{a_1} \ssarrow 1 1 \ssdrop{} \ssmove {-1}{-1}
\ssmove 2 {0}
\ssmove 2 {0} \ssdropbull \ssline 1 1 \ssdropbull \ssarrow 1 1 \ssdrop{} \ssmove {-2}{-2}
\ssmove 2 {0} \ssdropbull \ssarrow 1 1 \ssdrop{} \ssmove {-1}{-1}
\ssmove 2 {0} \ssdropbull \ssarrow 1 1 \ssdrop{} \ssmove {-1}{-1}
\ssmove 2 {0} \ssdropbull \ssarrow 1 1 \ssdrop{} \ssmove {-1}{-1}
\ssmove 2 {0} \ssdropbull \ssarrow 1 1 \ssdrop{} \ssmove {-1}{-1}
\ssmove 2 {0} \ssdropbull \ssarrow 1 1 \ssdrop{} \ssmove {-1}{-1}
\ssmove 2 {0} \ssdropbull \ssarrow 1 1 \ssdrop{} \ssmove {-1}{-1}
\ssmove 2 {0} \ssdropbull \ssarrow 1 1 \ssdrop{} \ssmove {-1}{-1}
\ssmove 2 {0} \ssdropbull \ssarrow 1 1 \ssdrop{} \ssmove {-1}{-1}
\ssmove 2 {0} \ssarrow 1 1 \ssdrop{} \ssmove {-1}{-1}
              \ssdropbull \ssarrow 1 1 \ssdrop{} \ssmove {-1}{-1}
\ssmove 2 {0} \ssdropbull \ssarrow 1 1 \ssdrop{} \ssmove {-1}{-1}
              \ssdropbull \ssarrow 1 1 \ssdrop{} \ssmove {-1}{-1}
\ssmove 2 {0} \ssdropbull \ssarrow 1 1 \ssdrop{} \ssmove {-1}{-1}
\ssmove 2 {0} \ssdropbull \ssline 1 1 \ssdropbull \ssarrow 1 1 \ssdrop{} \ssmove {-2}{-2}
              \ssdropbull \ssarrow 1 1 \ssdrop{} \ssmove {-1}{-1}
\ssmove 2 {0} \ssdropbull \ssarrow 1 1 \ssdrop{} \ssmove {-1}{-1}
              \ssdropbull \ssarrow 1 1 \ssdrop{} \ssmove {-1}{-1}
\ssmove 2 {0} \ssdropbull \ssarrow 1 1 \ssdrop{} \ssmove {-1}{-1}
              \ssdropbull \ssarrow 1 1 \ssdrop{} \ssmove {-1}{-1}
\ssmove 2 {0} \ssdropbull \ssarrow 1 1 \ssdrop{} \ssmove {-1}{-1}
              \ssdropbull \ssarrow 1 1 \ssdrop{} \ssmove {-1}{-1}
\ssmove 2 {0} \ssdropbull \ssarrow 1 1 \ssdrop{} \ssmove {-1}{-1}
              \ssdropbull \ssarrow 1 1 \ssdrop{} \ssmove {-1}{-1}
\ssmove 2 {0} \ssdropbull \ssarrow 1 1 \ssdrop{} \ssmove {-1}{-1}
              \ssdropbull \ssarrow 1 1 \ssdrop{} \ssmove {-1}{-1}
\ssmove 2 {0} \ssdropbull \ssarrow 1 1 \ssdrop{} \ssmove {-1}{-1}
              \ssdropbull \ssarrow 1 1 \ssdrop{} \ssmove {-1}{-1}
\ssmove 2 {0} \ssdropbull \ssarrow 1 1 \ssdrop{} \ssmove {-1}{-1}
              \ssdropbull \ssarrow 1 1 \ssdrop{} \ssmove {-1}{-1}
}{$H^*(\weier,\mbb Z/2)$ localized at 2.\label{fig:wmod2}}

Every $\bullet$ denotes a copy of $\mbb Z/2$. In this picture $g$ and $\Delta$ are also generators and the relation between them is $ga_1^4=h_1^4\Delta$. In the next picture we introduce some notation; namely the elements $x$ and $y$.

\ret{41}{9}{
\ssdropbull
\ssline 3 1 \ssdropbull
\ssline 3 1 \ssdropbull
\ssline 3 1 \ssdropbull
\ssline {-1}{-1} \ssdropbull
\ssline {-1}{-1} \ssdropbull \ssdroplabel{x}
\ssline 3 1 \ssdropbull
\ssline {-1}{-1} \ssdropbull \ssdroplabel[R]{xa_1}
\ssmove 6 2 \ssdropbull
\ssline {-1}{-1} \ssdropbull \ssdroplabel{d}
\ssline 3 1 \ssdropbull
\ssline {-1}{-1} \ssdropbull
\ssline {-1}{-1} \ssdropbull \ssdroplabel{y}
\ssline 3 1 \ssdropbull
\ssline 3 1 \ssdropbull
\ssmove {-1} 1 \ssdropbull \ssdroplabel{g} \xdy
\ssmoveto {24}{0} \ssdropbull \ssdroplabel{\Delta}\xdy
\ssmoveto {40} 8 \ssdropbull \ssline 3 1 \ssdropbull \ssmove {-3}{-1} \ssvoidarrow 1 1
\ssmoveto {0}{0} \ssarrow 1 1 \ssdrop{} \ssmove {-1}{-1}
\ssmove 2 {0} \ssdropbull \ssdroplabel{a_1} \ssarrow 1 1 \ssdrop{} \ssmove {-1}{-1}
\ssmove 2 {0} \ssdropbull\ssarrow 1 1 \ssdrop{} \ssmove {-1}{-1}
\ssmove 2 {0} \ssdropbull\ssarrow 1 1 \ssdrop{} \ssmove {-1}{-1}
\ssmove 2 {0} \ssdropbull\ssarrow 1 1 \ssdrop{} \ssmove {-1}{-1}
\ssmove 2 {0} \ssdropbull\ssarrow 1 1 \ssdrop{} \ssmove {-1}{-1}
\ssmove 2 {0} \ssdropbull\ssarrow 1 1 \ssdrop{} \ssmove {-1}{-1}
\ssmove 2 {0} \ssdropbull\ssarrow 1 1 \ssdrop{} \ssmove {-1}{-1}
\ssmove 2 {0} \ssdropbull\ssarrow 1 1 \ssdrop{} \ssmove {-1}{-1}
\ssmove 2 {0} \ssdropbull\ssarrow 1 1 \ssdrop{} \ssmove {-1}{-1}
\ssmove 2 {0} \ssdropbull\ssarrow 1 1 \ssdrop{} \ssmove {-1}{-1}
\ssmove 2 {0} \ssdropbull\ssarrow 1 1 \ssdrop{} \ssmove {-1}{-1}
\ssmove 2 {0} \ssarrow 1 1 \ssdrop{} \ssmove {-1}{-1} \ssdropbull\ssarrow 1 1 \ssdrop{} \ssmove {-1}{-1}
\ssmove 2 {0} \ssdropbull \ssarrow 1 1 \ssdrop{} \ssmove {-1}{-1} \ssdropbull\ssarrow 1 1 \ssdrop{} \ssmove {-1}{-1}
\ssmove 2 {0} \ssdropbull \ssarrow 1 1 \ssdrop{} \ssmove {-1}{-1} \ssdropbull\ssarrow 1 1 \ssdrop{} \ssmove {-1}{-1}
\ssmove 2 {0} \ssdropbull \ssarrow 1 1 \ssdrop{} \ssmove {-1}{-1} \ssdropbull\ssarrow 1 1 \ssdrop{} \ssmove {-1}{-1}
\ssmove 2 {0} \ssdropbull \ssarrow 1 1 \ssdrop{} \ssmove {-1}{-1} \ssdropbull\ssarrow 1 1 \ssdrop{} \ssmove {-1}{-1}
\ssmove 2 {0} \ssdropbull \ssarrow 1 1 \ssdrop{} \ssmove {-1}{-1} \ssdropbull\ssarrow 1 1 \ssdrop{} \ssmove {-1}{-1}
\ssmove 2 {0} \ssdropbull \ssarrow 1 1 \ssdrop{} \ssmove {-1}{-1} \ssdropbull\ssarrow 1 1 \ssdrop{} \ssmove {-1}{-1}
\ssmove 2 {0} \ssdropbull \ssarrow 1 1 \ssdrop{} \ssmove {-1}{-1} \ssdropbull\ssarrow 1 1 \ssdrop{} \ssmove {-1}{-1}
\ssmove 2 {0} \ssdropbull \ssarrow 1 1 \ssdrop{} \ssmove {-1}{-1} \ssdropbull\ssarrow 1 1 \ssdrop{} \ssmove {-1}{-1}
\ssmoveto {20}{4} \ssarrow 1 1 \ssdrop{} \ssmove {-1}{-1}
\ssmove 2 {0} \ssdropbull\ssarrow 1 1 \ssdrop{} \ssmove {-1}{-1}
\ssmove 2 {0} \ssdropbull\ssarrow 1 1 \ssdrop{} \ssmove {-1}{-1}
\ssmove 2 {0} \ssdropbull\ssarrow 1 1 \ssdrop{} \ssmove {-1}{-1}
}{$H^*(\weier,\mbb Z/2)$ localized at 2.}

In order to compute $H^*(\mell,\mbb Z/2)$, we proceed with localizing this cohomology. Because $c_4\equiv a_1^4 \mod 2$ we might as well invert $a_1$ instead of $c_4$. We know that $xa_1^2=da_1^2=ya_1=0$ and $g=h_1^4a_1^{-4}\Delta$, so inverting $a_1$ yields the following 2-periodic picture, in which every $\ominus$ denotes a copy of $\mbb Z/2[a_1^{-12}\Delta]$.

\ret{-8...32}{3}{
\ssmoveto {-8}{0} \ssdrop{\resizebox{2mm}{!}{$\ominus$}}\ssline 1 1 \ssdrop{\resizebox{2mm}{!}{$\ominus$}} \ssvoidarrow 1 1\ssmove {-1}{-1}
\ssmove 2 {0} \ssdrop{\resizebox{2mm}{!}{$\ominus$}}\ssline 1 1 \ssdrop{\resizebox{2mm}{!}{$\ominus$}} \ssvoidarrow 1 1\ssmove {-1}{-1}
\ssmove 2 {0} \ssdrop{\resizebox{2mm}{!}{$\ominus$}}\ssline 1 1 \ssdrop{\resizebox{2mm}{!}{$\ominus$}} \ssvoidarrow 1 1\ssmove {-1}{-1}
\ssmove 2 {0} \ssdrop{\resizebox{2mm}{!}{$\ominus$}}\ssdroplabel{\scriptstyle a_1^{-1}}\ssline 1 1 \ssdrop{\resizebox{2mm}{!}{$\ominus$}} \ssvoidarrow 1 1\ssmove {-1}{-1}
\ssmove 2 {0} \ssdrop{\resizebox{2mm}{!}{$\ominus$}}\ssline 1 1 \ssdrop{\resizebox{2mm}{!}{$\ominus$}} \ssvoidarrow 1 1\ssmove {-1}{-1}
\ssmove 2 {0} \ssdrop{\resizebox{2mm}{!}{$\ominus$}}\ssdroplabel{\scriptstyle a_1}\ssline 1 1 \ssdrop{\resizebox{2mm}{!}{$\ominus$}} \ssvoidarrow 1 1\ssmove {-1}{-1}
\ssmove 2 {0} \ssdrop{\resizebox{2mm}{!}{$\ominus$}}\ssline 1 1 \ssdrop{\resizebox{2mm}{!}{$\ominus$}} \ssvoidarrow 1 1\ssmove {-1}{-1}
\ssmove 2 {0} \ssdrop{\resizebox{2mm}{!}{$\ominus$}}\ssline 1 1 \ssdrop{\resizebox{2mm}{!}{$\ominus$}} \ssvoidarrow 1 1\ssmove {-1}{-1}
\ssmove 2 {0} \ssdrop{\resizebox{2mm}{!}{$\ominus$}}\ssline 1 1 \ssdrop{\resizebox{2mm}{!}{$\ominus$}} \ssvoidarrow 1 1\ssmove {-1}{-1}
\ssmove 2 {0} \ssdrop{\resizebox{2mm}{!}{$\ominus$}}\ssline 1 1 \ssdrop{\resizebox{2mm}{!}{$\ominus$}} \ssvoidarrow 1 1\ssmove {-1}{-1}
\ssmove 2 {0} \ssdrop{\resizebox{2mm}{!}{$\ominus$}}\ssline 1 1 \ssdrop{\resizebox{2mm}{!}{$\ominus$}} \ssvoidarrow 1 1\ssmove {-1}{-1}
\ssmove 2 {0} \ssdrop{\resizebox{2mm}{!}{$\ominus$}}\ssline 1 1 \ssdrop{\resizebox{2mm}{!}{$\ominus$}} \ssvoidarrow 1 1\ssmove {-1}{-1}
\ssmove 2 {0} \ssdrop{\resizebox{2mm}{!}{$\ominus$}}\ssline 1 1 \ssdrop{\resizebox{2mm}{!}{$\ominus$}} \ssvoidarrow 1 1\ssmove {-1}{-1}
\ssmove 2 {0} \ssdrop{\resizebox{2mm}{!}{$\ominus$}}\ssline 1 1 \ssdrop{\resizebox{2mm}{!}{$\ominus$}} \ssvoidarrow 1 1\ssmove {-1}{-1}
\ssmove 2 {0} \ssdrop{\resizebox{2mm}{!}{$\ominus$}}\ssline 1 1 \ssdrop{\resizebox{2mm}{!}{$\ominus$}} \ssvoidarrow 1 1\ssmove {-1}{-1}
\ssmove 2 {0} \ssdrop{\resizebox{2mm}{!}{$\ominus$}}\ssline 1 1 \ssdrop{\resizebox{2mm}{!}{$\ominus$}} \ssvoidarrow 1 1\ssmove {-1}{-1}
\ssmove 2 {0} \ssdrop{\resizebox{2mm}{!}{$\ominus$}}\ssline 1 1 \ssdrop{\resizebox{2mm}{!}{$\ominus$}} \ssvoidarrow 1 1\ssmove {-1}{-1}
\ssmove 2 {0} \ssdrop{\resizebox{2mm}{!}{$\ominus$}}\ssline 1 1 \ssdrop{\resizebox{2mm}{!}{$\ominus$}} \ssvoidarrow 1 1\ssmove {-1}{-1}
\ssmove 2 {0} \ssdrop{\resizebox{2mm}{!}{$\ominus$}}\ssline 1 1 \ssdrop{\resizebox{2mm}{!}{$\ominus$}} \ssvoidarrow 1 1\ssmove {-1}{-1}
\ssmove 2 {0} \ssdrop{\resizebox{2mm}{!}{$\ominus$}}\ssline 1 1 \ssdrop{\resizebox{2mm}{!}{$\ominus$}} \ssvoidarrow 1 1\ssmove {-1}{-1}
\ssmove 2 {0} \ssdrop{\resizebox{2mm}{!}{$\ominus$}}\ssline 1 1 \ssdrop{\resizebox{2mm}{!}{$\ominus$}} \ssvoidarrow 1 1\ssmove {-1}{-1}
\ssmove 2 {0} \ssdrop{\resizebox{2mm}{!}{$\ominus$}}\ssline 1 1 \ssdrop{\resizebox{2mm}{!}{$\ominus$}} \ssvoidarrow 1 1\ssmove {-1}{-1}
\ssmove 2 {0} \ssdrop{\resizebox{2mm}{!}{$\ominus$}}\ssline 1 1 \ssdrop{\resizebox{2mm}{!}{$\ominus$}} \ssvoidarrow 1 1\ssmove {-1}{-1}
}{$H^*(\weier,\mbb Z/2)[a_1^{-1}]$ localized at 2.}

The element $\Delta$ is a non-zero-divisor generator in $H^*(\weier,\mbb Z/2)$ and the algebra $H^*(\weier,\mbb Z/2)[\Delta^{-1}]$ also has an easy description. It is periodic with period $\Delta$ to both the left and the right, and the map $i_\Delta$ is an inclusion.

\ret{-8...32}{12}{
\ssdrop{\resizebox{2mm}{!}{$\overt$}} \xdy
\ssmove {-1}{1} \ssdropbull \ssname{g} \xdy
\ssmoveto {24}{0}\ssdrop{\resizebox{2mm}{!}{$\overt$}} \xdy
\ssmoveto {-24}{0}\ssdrop{\resizebox{2mm}{!}{$\overt$}} \xdy
\ssmove {-1}{1} \ssdropbull \ssname{g1}\xdy
\ssmove {-1}{1} \ssdropbull \ssname{g2}\xdy
\ssmoveto {-48}{0}\ssdrop{\resizebox{2mm}{!}{$\overt$}} \xdy
\ssmove {-1}{1} \ssdropbull \ssname{g3}\xdy
\ssmove {-1}{1} \ssdropbull \ssname{g4}\xdy
\ssmoveto {-8}{0} \ssdrop{\resizebox{2mm}{!}{$\overt$}}\ssline 1 1 \ssdrop{\resizebox{2mm}{!}{$\overt$}} \ssarrow 1 1 \ssdrop{} \ssmove {-2}{-2}
\ssmove 2 {0} \ssdrop{\resizebox{2mm}{!}{$\overt$}}\ssline 1 1 \ssdrop{\resizebox{2mm}{!}{$\overt$}} \ssarrow 1 1 \ssdrop{} \ssmove {-2}{-2}
\ssmove 2 {0} \ssdrop{\resizebox{2mm}{!}{$\overt$}}
\ssline 1 1 \ssdrop{\resizebox{2mm}{!}{$\overt$}} \ssarrow 1 1 \ssdrop{} \ssmove {-2}{-2}
\ssmove 2 {0} \ssdrop{\resizebox{2mm}{!}{$\overt$}}
\ssline 1 1 \ssdrop{\resizebox{2mm}{!}{$\overt$}} \ssarrow 1 1 \ssdrop{} \ssmove {-2}{-2}
\ssmove 2 {0} \ssline 1 1\ssdrop{\resizebox{2mm}{!}{$\overt$}} \ssarrow 1 1 \ssdrop{} \ssmove {-2}{-2}
\ssmove 2 {0} \ssdrop{\resizebox{2mm}{!}{$\overt$}}
\ssline 1 1 \ssdrop{\resizebox{2mm}{!}{$\overt$}} \ssarrow 1 1 \ssdrop{} \ssmove {-2}{-2}
\ssmove 2 {0} \ssdrop{\resizebox{2mm}{!}{$\overt$}}\ssline 1 1 \ssdrop{\resizebox{2mm}{!}{$\overt$}} \ssarrow 1 1 \ssdrop{} \ssmove {-2}{-2}
\ssmove 2 {0} \ssdrop{\resizebox{2mm}{!}{$\overt$}}\ssline 1 1 \ssdrop{\resizebox{2mm}{!}{$\overt$}} \ssarrow 1 1 \ssdrop{} \ssmove {-2}{-2}
\ssmove 2 {0} \ssdrop{\resizebox{2mm}{!}{$\overt$}}\ssline 1 1 \ssdrop{\resizebox{2mm}{!}{$\overt$}} \ssarrow 1 1 \ssdrop{} \ssmove {-2}{-2}
\ssmove 2 {0} \ssdrop{\resizebox{2mm}{!}{$\overt$}}\ssline 1 1 \ssdrop{\resizebox{2mm}{!}{$\overt$}} \ssarrow 1 1 \ssdrop{} \ssmove {-2}{-2}
\ssmove 2 {0} \ssdrop{\resizebox{2mm}{!}{$\overt$}}\ssline 1 1 \ssdrop{\resizebox{2mm}{!}{$\overt$}} \ssarrow 1 1 \ssdrop{} \ssmove {-2}{-2}
\ssmove 2 {0} \ssdrop{\resizebox{2mm}{!}{$\overt$}}\ssline 1 1 \ssdrop{\resizebox{2mm}{!}{$\overt$}} \ssarrow 1 1 \ssdrop{} \ssmove {-2}{-2}
\ssmove 2 {0} \ssdrop{\resizebox{2mm}{!}{$\overt$}}\ssline 1 1 \ssdrop{\resizebox{2mm}{!}{$\overt$}} \ssarrow 1 1 \ssdrop{} \ssmove {-2}{-2}
\ssmove 2 {0} \ssdrop{\resizebox{2mm}{!}{$\overt$}}\ssline 1 1 \ssdrop{\resizebox{2mm}{!}{$\overt$}} \ssarrow 1 1 \ssdrop{} \ssmove {-2}{-2}
\ssmove 2 {0} \ssdrop{\resizebox{2mm}{!}{$\overt$}}\ssline 1 1 \ssdrop{\resizebox{2mm}{!}{$\overt$}} \ssarrow 1 1 \ssdrop{} \ssmove {-2}{-2}
\ssmove 2 {0} \ssdrop{\resizebox{2mm}{!}{$\overt$}}\ssline 1 1 \ssdrop{\resizebox{2mm}{!}{$\overt$}} \ssarrow 1 1 \ssdrop{} \ssmove {-2}{-2}
\ssmove 2 {0} \ssline 1 1 \ssdrop{\resizebox{2mm}{!}{$\overt$}} \ssarrow 1 1 \ssdrop{} \ssmove {-2}{-2}
\ssmove 2 {0} \ssdrop{\resizebox{2mm}{!}{$\overt$}}\ssline 1 1 \ssdrop{\resizebox{2mm}{!}{$\overt$}} \ssarrow 1 1 \ssdrop{} \ssmove {-2}{-2}
\ssmove 2 {0} \ssdrop{\resizebox{2mm}{!}{$\overt$}}\ssline 1 1 \ssdrop{\resizebox{2mm}{!}{$\overt$}} \ssarrow 1 1 \ssdrop{} \ssmove {-2}{-2}
\ssmove 2 {0} \ssdrop{\resizebox{2mm}{!}{$\overt$}}\ssline 1 1 \ssdrop{\resizebox{2mm}{!}{$\overt$}} \ssarrow 1 1 \ssdrop{} \ssmove {-2}{-2}
\ssmove 2 {0} \ssdrop{\resizebox{2mm}{!}{$\overt$}}\ssline 1 1 \ssdrop{\resizebox{2mm}{!}{$\overt$}} \ssarrow 1 1 \ssdrop{} \ssmove {-2}{-2}
\ssgoto{g}\ssarrow 1 1 \ssdrop{} \ssmove {-1}{-1}
\ssmove 2 {0} \ssdropbull \ssarrow 1 1 \ssdrop{}\ssmove {-1}{-1}
\ssmove 2 {0} \ssdropbull \ssarrow 1 1 \ssdrop{}\ssmove {-1}{-1}
\ssmove 2 {0} \ssdropbull \ssarrow 1 1 \ssdrop{}\ssmove {-1}{-1}
\ssgoto{g1}\ssarrow 1 1 \ssdrop{} \ssmove {-1}{-1}
\ssmove 2 {0} \ssdropbull \ssarrow 1 1 \ssdrop{}\ssmove {-1}{-1}
\ssmove 2 {0} \ssdropbull \ssarrow 1 1 \ssdrop{}\ssmove {-1}{-1}
\ssmove 2 {0} \ssdropbull \ssarrow 1 1 \ssdrop{}\ssmove {-1}{-1}
\ssgoto{g2}\ssarrow 1 1 \ssdrop{} \ssmove {-1}{-1}
\ssmove 2 {0} \ssdropbull \ssarrow 1 1 \ssdrop{}\ssmove {-1}{-1}
\ssmove 2 {0} \ssdropbull \ssarrow 1 1 \ssdrop{}\ssmove {-1}{-1}
\ssmove 2 {0} \ssdropbull \ssarrow 1 1 \ssdrop{}\ssmove {-1}{-1}
\ssgoto{g3}\ssarrow 1 1 \ssdrop{} \ssmove {-1}{-1}
\ssmove 2 {0} \ssdropbull \ssarrow 1 1 \ssdrop{}\ssmove {-1}{-1}
\ssmove 2 {0} \ssdropbull \ssarrow 1 1 \ssdrop{}\ssmove {-1}{-1}
\ssmove 2 {0} \ssdropbull \ssarrow 1 1 \ssdrop{}\ssmove {-1}{-1}
\ssgoto{g4}\ssarrow 1 1 \ssdrop{} \ssmove {-1}{-1}
\ssmove 2 {0} \ssdropbull \ssarrow 1 1 \ssdrop{}\ssmove {-1}{-1}
\ssmove 2 {0} \ssdropbull \ssarrow 1 1 \ssdrop{}\ssmove {-1}{-1}
\ssmove 2 {0} \ssdropbull \ssarrow 1 1 \ssdrop{}\ssmove {-1}{-1}
}{$H^*(\weier,\mbb Z/2)[\Delta^{-1}]$ localized at 2.}

A $\bullet$ still represents a copy of $\mbb Z/2$ and a $\overt$ represents a copy of $\mbb Z/2[a_1^{12}\Delta^{-1}]$.

For $H^*(\weier,\mbb Z/2)[a_1^{-1},\Delta^{-1}]$ we get the following, where every $\oplus$ denotes a copy of $\mbb Z/2[a_1^{-12}\Delta,a_1^{12}\Delta^{-1}]$.

\ret{-8...32}{3}{
\ssmoveto {-8}{0} \ssdrop{\resizebox{2mm}{!}{$\oplus$}}\ssline 1 1 \ssdrop{\resizebox{2mm}{!}{$\oplus$}} \ssvoidarrow 1 1\ssmove {-1}{-1}
\ssmove 2 {0} \ssdrop{\resizebox{2mm}{!}{$\oplus$}}\ssline 1 1 \ssdrop{\resizebox{2mm}{!}{$\oplus$}} \ssvoidarrow 1 1\ssmove {-1}{-1}
\ssmove 2 {0} \ssdrop{\resizebox{2mm}{!}{$\oplus$}}\ssline 1 1 \ssdrop{\resizebox{2mm}{!}{$\oplus$}} \ssvoidarrow 1 1\ssmove {-1}{-1}
\ssmove 2 {0} \ssdrop{\resizebox{2mm}{!}{$\oplus$}}\ssdroplabel{\scriptstyle a_1^{-1}}\ssline 1 1 \ssdrop{\resizebox{2mm}{!}{$\oplus$}} \ssvoidarrow 1 1\ssmove {-1}{-1}
\ssmove 2 {0} \ssdrop{\resizebox{2mm}{!}{$\oplus$}}\ssline 1 1 \ssdrop{\resizebox{2mm}{!}{$\oplus$}} \ssvoidarrow 1 1\ssmove {-1}{-1}
\ssmove 2 {0} \ssdrop{\resizebox{2mm}{!}{$\oplus$}}\ssdroplabel{\scriptstyle a_1}\ssline 1 1 \ssdrop{\resizebox{2mm}{!}{$\oplus$}} \ssvoidarrow 1 1\ssmove {-1}{-1}
\ssmove 2 {0} \ssdrop{\resizebox{2mm}{!}{$\oplus$}}\ssline 1 1 \ssdrop{\resizebox{2mm}{!}{$\oplus$}} \ssvoidarrow 1 1\ssmove {-1}{-1}
\ssmove 2 {0} \ssdrop{\resizebox{2mm}{!}{$\oplus$}}\ssline 1 1 \ssdrop{\resizebox{2mm}{!}{$\oplus$}} \ssvoidarrow 1 1\ssmove {-1}{-1}
\ssmove 2 {0} \ssdrop{\resizebox{2mm}{!}{$\oplus$}}\ssline 1 1 \ssdrop{\resizebox{2mm}{!}{$\oplus$}} \ssvoidarrow 1 1\ssmove {-1}{-1}
\ssmove 2 {0} \ssdrop{\resizebox{2mm}{!}{$\oplus$}}\ssline 1 1 \ssdrop{\resizebox{2mm}{!}{$\oplus$}} \ssvoidarrow 1 1\ssmove {-1}{-1}
\ssmove 2 {0} \ssdrop{\resizebox{2mm}{!}{$\oplus$}}\ssline 1 1 \ssdrop{\resizebox{2mm}{!}{$\oplus$}} \ssvoidarrow 1 1\ssmove {-1}{-1}
\ssmove 2 {0} \ssdrop{\resizebox{2mm}{!}{$\oplus$}}\ssline 1 1 \ssdrop{\resizebox{2mm}{!}{$\oplus$}} \ssvoidarrow 1 1\ssmove {-1}{-1}
\ssmove 2 {0} \ssdrop{\resizebox{2mm}{!}{$\oplus$}}\ssline 1 1 \ssdrop{\resizebox{2mm}{!}{$\oplus$}} \ssvoidarrow 1 1\ssmove {-1}{-1}
\ssmove 2 {0} \ssdrop{\resizebox{2mm}{!}{$\oplus$}}\ssline 1 1 \ssdrop{\resizebox{2mm}{!}{$\oplus$}} \ssvoidarrow 1 1\ssmove {-1}{-1}
\ssmove 2 {0} \ssdrop{\resizebox{2mm}{!}{$\oplus$}}\ssline 1 1 \ssdrop{\resizebox{2mm}{!}{$\oplus$}} \ssvoidarrow 1 1\ssmove {-1}{-1}
\ssmove 2 {0} \ssdrop{\resizebox{2mm}{!}{$\oplus$}}\ssline 1 1 \ssdrop{\resizebox{2mm}{!}{$\oplus$}} \ssvoidarrow 1 1\ssmove {-1}{-1}
\ssmove 2 {0} \ssdrop{\resizebox{2mm}{!}{$\oplus$}}\ssline 1 1 \ssdrop{\resizebox{2mm}{!}{$\oplus$}} \ssvoidarrow 1 1\ssmove {-1}{-1}
\ssmove 2 {0} \ssdrop{\resizebox{2mm}{!}{$\oplus$}}\ssline 1 1 \ssdrop{\resizebox{2mm}{!}{$\oplus$}} \ssvoidarrow 1 1\ssmove {-1}{-1}
\ssmove 2 {0} \ssdrop{\resizebox{2mm}{!}{$\oplus$}}\ssline 1 1 \ssdrop{\resizebox{2mm}{!}{$\oplus$}} \ssvoidarrow 1 1\ssmove {-1}{-1}
\ssmove 2 {0} \ssdrop{\resizebox{2mm}{!}{$\oplus$}}\ssline 1 1 \ssdrop{\resizebox{2mm}{!}{$\oplus$}} \ssvoidarrow 1 1\ssmove {-1}{-1}
\ssmove 2 {0} \ssdrop{\resizebox{2mm}{!}{$\oplus$}}\ssline 1 1 \ssdrop{\resizebox{2mm}{!}{$\oplus$}} \ssvoidarrow 1 1\ssmove {-1}{-1}
\ssmove 2 {0} \ssdrop{\resizebox{2mm}{!}{$\oplus$}}\ssline 1 1 \ssdrop{\resizebox{2mm}{!}{$\oplus$}} \ssvoidarrow 1 1\ssmove {-1}{-1}
\ssmove 2 {0} \ssdrop{\resizebox{2mm}{!}{$\oplus$}}\ssline 1 1 \ssdrop{\resizebox{2mm}{!}{$\oplus$}} \ssvoidarrow 1 1\ssmove {-1}{-1}
}{$H^*(\weier,\mbb Z/2)[a_1^{-1}]$ localized at 2.}

To calculate $H^*(\mell,\mbb Z/2)$, we work with the version of (\ref{eq:vie2}) with $\mbb Z/2$-coefficients.

For the kernel of $D$, we recall from the proof of lemma \ref{lem:ker} that any element $(u,v)$ of the group $\ker D\subset H^*(\weier,\mbb Z/2)[a_1^{-1}]\oplus H^*(\weier,\mbb Z/2)[\Delta^{-1}]$ satisfies $u=i_{a_1}(v)$ and that we can use this to identify $\ker D$ with the subset of elements $v\in H^*(\weier,\mbb Z/2)[\Delta^{-1}]$ such that $i_{a_1}(v)\in H^*(\weier,\mbb Z/2)[a_1^{-1}]$. We note that products in $H^*(\weier,\mbb Z/2)[\Delta^{-1}]$ containing $h_2$, $x$, $d$ or $y$ are sent to zero by $i_{a_1}$, so these products are certainly in $\ker D$. The remaining part of $H^*(\weier,\mbb Z/2)[\Delta^{-1}]$ is spanned by elements of the form $h_1^jg^ka_1^l\Delta^m$ and we compute that $$i_{a_1}(h_1^jg^ka_1^l\Delta^m)=h_1^{j+4k}a_1^{l-4k}\Delta^{k+m}.$$
This belongs to $H^*(\weier,\mbb Z/2)[a_1^{-1}]$ if and only if $k+m\geq0$. We now have a complete description of the kernel of $D$.

We continue with the computation of the cokernel of $D$. This group is generated by the basis elements of $H^*(\weier,\mbb Z/2)[a_1^{-1},\Delta^{-1}]$ that are not hit by $i_{\Delta}$ and $i_{a_1}$. So let's look at an element $a=h_1^xa_1^y\Delta^z$ with $x\geq0$ and $y,z\in\mbb Z$. Then $a\not\in i_\Delta \left(H^*(\weier,\mbb Z/2)[a_1^{-1}]\right)$ precisely when $z<0$. We have that $a\not\in i_{a_1}\in \left(H^*(\weier,\mbb Z/2)[\Delta^{-1}]\right)$ iff it cannot be written in the form $h_1^{j+4k}a_1^{l-4k}\Delta^{k+m}$ with $j,k,l\geq0$. That is if and only if $\lfloor\frac{x}{4}\rfloor+\lfloor\frac{y}{4}\rfloor<0$ or equivalently $x<4\lceil\frac{-y}{4}\rceil$. We conclude that the cokernel is generated by the classes of the elements $h_1^xa_1^y\Delta^z$ with $4\lceil\frac{-y}{4}\rceil>x\geq0$ and $z<0$. 

Now we can compute $H^*(\mell, \mbb Z/2)$ with the short exact sequence (\ref{eq:sh}) induced by the Mayer-Vietoris sequence. Because we are working with $\mbb Z/2$-coefficients, this sequence splits and the ring $H^i(\mell, \mbb Z/2)$ is entirely determined by $\coker^{i-1} D$ and $\ker^i D$. There is only one possible result, namely $$H^i(\mell, \mbb Z/2)=\coker^{i-1} D\oplus\ker^i D.$$ As can be seen in Figure \ref{fig:mell2q} we are only interested in the negative degrees, so we have focussed the following picture on this part.

\ret{-52...-12}{22}
{\ssmoveto {-24}{0} \minxdy
\ssmoveto {-48}{0} \minxdy \ssmove {-1} 1 \minxdy
\ssmoveto {-72}{0} \minxdy \ssmove {-1} 1 \minxdy \ssmove {-1} 1 \minxdy \ssmove {-1} 1 \ssdropbull \ssname{g} \xdy
\ssmoveto {-96}{0} \minxdy \ssmove {-1} 1 \minxdy \ssmove {-1} 1 \minxdy \ssmove {-1} 1 \minxdy \ssmove {-1} 1 \ssdropbull \ssname{g1} \xdy
\ssmoveto {-120}{0} \minxdy \ssmove {-1} 1 \minxdy \ssmove {-1} 1 \minxdy \ssmove {-1} 1 \minxdy \ssmove {-1} 1 \minxdy \ssmove {-1} 1 \ssdropbull \ssname{g2} \xdy
\ssmoveto {-144}{0} \minxdy \ssmove {-1} 1 \minxdy \ssmove {-1} 1 \minxdy \ssmove{-1} 1 \minxdy \ssmove {-1} 1 \minxdy \ssmove {-1} 1 \minxdy \ssmove {-1} 1 \ssdropbull \xdy
\ssgoto{g} \ssarrow 1 1 \ssdrop{}\ssmove{-1}{-1}
\ssgoto{g1} \ssline 1 1 \ssdropbull \ssline 1 1 \ssdropbull \ssline 1 1 \ssdropbull \ssline 1 1 \ssdropbull \ssarrow 1 1 \ssdrop{}\ssmove{-5}{-5}
  \ssmove 2 {0} \ssinfbullstring{1}{1}{3}\ssmove{-2}{-2}
  \ssmove 2 {0} \ssdropbull \ssarrow 1 1 \ssdrop{}\ssmove{-1}{-1}
\ssgoto{g2} \ssline 1 1 \ssdropbull \ssarrow 1 1 \ssdrop{}\ssmove{-2}{-2}
  \ssmove 2 {0} \ssinfbullstring{1}{1}{2} \ssmove{-1}{-1}
  \ssmove 2 {0} \ssinfbullstring{1}{1}{2} \ssmove{-1}{-1}
  \ssmove 2 {0} \ssinfbullstring{1}{1}{2} \ssmove{-1}{-1}
\ssmoveto {-27} 1
  \ssdropbull \ssline 1 1 \ssdropbull \ssline 1 1 \ssdropbull \ssline 1 1 \ssdropbull \ssline[dashed] {-3} {-1} \ssmove {-2} {-2}
  \ssdropbull \ssline 1 1 \ssdropbull \ssline 1 1 \ssdropbull \ssline 1 1 \ssdropbull \ssmove {-3}{-3}\ssmove {-2} {0}
  \ssdropbull \ssline 1 1 \ssdropbull \ssline 1 1 \ssdropbull \ssline 1 1 \ssdropbull \ssmove {-3}{-3}\ssmove {-2} {0}
  \ssdropbull \ssline 1 1 \ssdropbull \ssline 1 1 \ssdropbull \ssline 1 1 \ssdropbull \ssmove {-3}{-3}\ssmove {-2} {0}
  \ssdropbull \ssline 1 1 \ssdropbull \ssline 1 1 \ssdropbull \ssline 1 1 \ssdropbull \ssline 1 1
    \ssdropbull \ssline 1 1 \ssdropbull \ssline 1 1 \ssdropbull \ssline 1 1 \ssdropbull \ssline[dashed]{-3}{-1} \ssmove{-6}{-6}
  \ssdropbull \ssline 1 1 \ssdropbull \ssline 1 1 \ssdropbull \ssline 1 1 \ssdropbull \ssline 1 1
    \ssdropbull \ssline 1 1 \ssdropbull \ssline 1 1 \ssdropbull \ssline 1 1 \ssdropbull \ssmove {-7}{-7}\ssmove {-2} {0}
  \ssdropbull \ssline 1 1 \ssdropbull \ssline 1 1 \ssdropbull \ssline 1 1 \ssdropbull \ssline 1 1
    \ssdropbull \ssline 1 1 \ssdropbull \ssline 1 1 \ssdropbull \ssline 1 1 \ssdropbull \ssmove {-7}{-7}\ssmove {-2} {0}
  \ssdropbull \ssline 1 1 \ssdropbull \ssline 1 1 \ssdropbull \ssline 1 1 \ssdropbull \ssline 1 1
    \ssdropbull \ssline 1 1 \ssdropbull \ssline 1 1 \ssdropbull \ssline 1 1 \ssdropbull \ssmove {-7}{-7}\ssmove {-2} {0}
  \ssdropbull \ssline 1 1 \ssdropbull \ssline 1 1 \ssdropbull \ssline 1 1 \ssdropbull \ssline 1 1
    \ssdropbull \ssline 1 1 \ssdropbull \ssline 1 1 \ssdropbull \ssline 1 1 \ssdropbull \ssline 1 1
    \ssdropbull \ssline 1 1 \ssdropbull \ssline 1 1 \ssdropbull \ssline 1 1 \ssdropbull \ssline[dashed]{-3}{-1} \ssmove{-10}{-10}
  \ssdropbull \ssline 1 1 \ssdropbull \ssline 1 1 \ssdropbull \ssline 1 1 \ssdropbull \ssline 1 1
    \ssdropbull \ssline 1 1 \ssdropbull \ssline 1 1 \ssdropbull \ssline 1 1 \ssdropbull \ssline 1 1
    \ssdropbull \ssline 1 1 \ssdropbull \ssline 1 1 \ssdropbull \ssline 1 1 \ssdropbull \ssmove {-11}{-11}\ssmove {-2} {0}
  \ssdropbull \ssline 1 1 \ssdropbull \ssline 1 1 \ssdropbull \ssline 1 1 \ssdropbull \ssline 1 1
    \ssdropbull \ssline 1 1 \ssdropbull \ssline 1 1 \ssdropbull \ssline 1 1 \ssdropbull \ssline 1 1
    \ssdropbull \ssline 1 1 \ssdropbull \ssline 1 1 \ssdropbull \ssline 1 1 \ssdropbull \ssmove {-11}{-11}\ssmove {-2} {0}
  \ssdropbull \ssline 1 1 \ssdropbull \ssline 1 1 \ssdropbull \ssline 1 1 \ssdropbull \ssline 1 1
    \ssdropbull \ssline 1 1 \ssdropbull \ssline 1 1 \ssdropbull \ssline 1 1 \ssdropbull \ssline 1 1
    \ssdropbull \ssline 1 1 \ssdropbull \ssline 1 1 \ssdropbull \ssline 1 1 \ssdropbull \ssmove {-11}{-11}\ssmove {-2} {0}
  \ssdropbull \ssline 1 1 \ssdropbull \ssline 1 1 \ssdropbull \ssline 1 1 \ssdropbull \ssline[dashed]{-3}{-1} \ssmove {0}{-2}
  \ssdropbull \ssline 1 1 \ssdropbull \ssline 1 1 \ssdropbull \ssline 1 1 \ssdropbull \ssline 1 1
    \ssdropbull \ssline 1 1 \ssdropbull \ssline 1 1 \ssdropbull \ssline 1 1 \ssdropbull \ssline 1 1
    \ssdropbull \ssline 1 1 \ssdropbull \ssline 1 1 \ssdropbull \ssline 1 1 \ssdropbull \ssline 1 1
    \ssdropbull \ssline 1 1 \ssdropbull \ssline 1 1 \ssdropbull \ssline 1 1 \ssdropbull \ssline[dashed]{-3}{-1} \ssmove{-14}{-14}
  \ssdropbull \ssline 1 1 \ssdropbull \ssline 1 1 \ssdropbull \ssline 1 1 \ssdropbull \ssmove {-3}{-3}
  \ssdropbull \ssline 1 1 \ssdropbull \ssline 1 1 \ssdropbull \ssline 1 1 \ssdropbull \ssline 1 1
    \ssdropbull \ssline 1 1 \ssdropbull \ssline 1 1 \ssdropbull \ssline 1 1 \ssdropbull \ssline 1 1
    \ssdropbull \ssline 1 1 \ssdropbull \ssline 1 1 \ssdropbull \ssline 1 1 \ssdropbull \ssline 1 1
    \ssdropbull \ssline 1 1 \ssdropbull \ssline 1 1 \ssdropbull \ssline 1 1 \ssdropbull \ssmove {-15}{-15}\ssmove {-2} {0}
  \ssdropbull \ssline 1 1 \ssdropbull \ssline 1 1 \ssdropbull \ssline 1 1 \ssdropbull \ssmove {-3}{-3}
  \ssdropbull \ssline 1 1 \ssdropbull \ssline 1 1 \ssdropbull \ssline 1 1 \ssdropbull \ssline 1 1
    \ssdropbull \ssline 1 1 \ssdropbull \ssline 1 1 \ssdropbull \ssline 1 1 \ssdropbull \ssline 1 1
    \ssdropbull \ssline 1 1 \ssdropbull \ssline 1 1 \ssdropbull \ssline 1 1 \ssdropbull \ssline 1 1
    \ssdropbull \ssline 1 1 \ssdropbull \ssline 1 1 \ssdropbull \ssline 1 1 \ssdropbull \ssmove {-15}{-15}\ssmove {-2} {0}
  \ssdropbull \ssline 1 1 \ssdropbull \ssline 1 1 \ssdropbull \ssline 1 1 \ssdropbull \ssmove {-3}{-3}
  \ssdropbull \ssline 1 1 \ssdropbull \ssline 1 1 \ssdropbull \ssline 1 1 \ssdropbull \ssline 1 1
    \ssdropbull \ssline 1 1 \ssdropbull \ssline 1 1 \ssdropbull \ssline 1 1 \ssdropbull \ssline 1 1
    \ssdropbull \ssline 1 1 \ssdropbull \ssline 1 1 \ssdropbull \ssline 1 1 \ssdropbull \ssline 1 1
    \ssdropbull \ssline 1 1 \ssdropbull \ssline 1 1 \ssdropbull \ssline 1 1 \ssdropbull \ssmove {-15}{-15}\ssmove {-2} {0}
  \ssdropbull \ssline 1 1 \ssdropbull \ssline 1 1 \ssdropbull \ssline 1 1 \ssdropbull \ssline 1 1
    \ssdropbull \ssline 1 1 \ssdropbull \ssline 1 1 \ssdropbull \ssline 1 1 \ssdropbull \ssline[dashed]{-3}{-1} \ssmove {-4}{-6}
  \ssdropbull \ssline 1 1 \ssdropbull \ssline 1 1 \ssdropbull \ssline 1 1 \ssdropbull \ssline 1 1
    \ssdropbull \ssline 1 1 \ssdropbull \ssline 1 1 \ssdropbull \ssline 1 1 \ssdropbull \ssline 1 1
    \ssdropbull \ssline 1 1 \ssdropbull \ssline 1 1 \ssdropbull \ssline 1 1 \ssdropbull \ssline 1 1
    \ssdropbull \ssline 1 1 \ssdropbull \ssline 1 1 \ssdropbull \ssline 1 1 \ssdropbull \ssline 1 1
    \ssdropbull \ssline 1 1 \ssdropbull \ssline 1 1 \ssdropbull \ssline 1 1 \ssdropbull \ssline[dashed]{-3}{-1} \ssmove{-18}{-18}
  \ssdropbull \ssline 1 1 \ssdropbull \ssline 1 1 \ssdropbull \ssline 1 1 \ssdropbull \ssline 1 1
    \ssdropbull \ssline 1 1 \ssdropbull \ssline 1 1 \ssdropbull \ssline 1 1 \ssdropbull \ssline 1 1
    \ssdropbull \ssline 1 1 \ssdropbull \ssline 1 1 \ssdropbull \ssline 1 1 \ssdropbull \ssline 1 1
    \ssdropbull \ssline 1 1 \ssdropbull \ssline 1 1 \ssdropbull \ssline 1 1 \ssdropbull \ssline 1 1
    \ssdropbull \ssline 1 1 \ssdropbull \ssline 1 1 \ssdropbull \ssline 1 1 \ssdropbull \ssmove {-19}{-19}\ssmove {-2} {0}
  \ssdropbull \ssline 1 1 \ssdropbull \ssline 1 1 \ssdropbull \ssline 1 1 \ssdropbull \ssline 1 1
    \ssdropbull \ssline 1 1 \ssdropbull \ssline 1 1 \ssdropbull \ssline 1 1 \ssdropbull \ssline 1 1
    \ssdropbull \ssline 1 1 \ssdropbull \ssline 1 1 \ssdropbull \ssline 1 1 \ssdropbull \ssline 1 1
    \ssdropbull \ssline 1 1 \ssdropbull \ssline 1 1 \ssdropbull \ssline 1 1 \ssdropbull \ssline 1 1
    \ssdropbull \ssline 1 1 \ssdropbull \ssline 1 1 \ssdropbull \ssline 1 1 \ssdropbull \ssmove {-19}{-19}\ssmove {-2} {0}
  \ssdropbull \ssline 1 1 \ssdropbull \ssline 1 1 \ssdropbull \ssline 1 1 \ssdropbull \ssline 1 1
    \ssdropbull \ssline 1 1 \ssdropbull \ssline 1 1 \ssdropbull \ssline 1 1 \ssdropbull \ssline 1 1
    \ssdropbull \ssline 1 1 \ssdropbull \ssline 1 1 \ssdropbull \ssline 1 1 \ssdropbull \ssline 1 1
    \ssdropbull \ssline 1 1 \ssdropbull \ssline 1 1 \ssdropbull \ssline 1 1 \ssdropbull \ssline 1 1
    \ssdropbull \ssline 1 1 \ssdropbull \ssline 1 1 \ssdropbull \ssline 1 1 \ssdropbull \ssmove {-19}{-19}\ssmove {-2} {0}
  \ssdropbull \ssline 1 1 \ssdropbull \ssline 1 1 \ssdropbull \ssline 1 1 \ssdropbull \ssline 1 1
    \ssdropbull \ssline 1 1 \ssdropbull \ssline 1 1 \ssdropbull \ssline 1 1 \ssdropbull \ssline 1 1
    \ssdropbull \ssline 1 1 \ssdropbull \ssline 1 1 \ssdropbull \ssline 1 1 \ssdropbull \ssline 1 1
    \ssdropbull \ssline 1 1 \ssdropbull \ssline 1 1 \ssdropbull \ssline 1 1 \ssdropbull \ssline 1 1
    \ssdropbull \ssline 1 1 \ssdropbull \ssline 1 1 \ssdropbull \ssline 1 1 \ssdropbull \ssline 1 1
    \ssdropbull \ssline 1 1 \ssdropbull \ssline 1 1 \ssdropbull \ssline 1 1 \ssdropbull \ssmove {-23}{-23}\ssmove {-2} {0}
  \ssdropbull \ssline 1 1 \ssdropbull \ssline 1 1 \ssdropbull \ssline 1 1 \ssdropbull \ssline 1 1
    \ssdropbull \ssline 1 1 \ssdropbull \ssline 1 1 \ssdropbull \ssline 1 1 \ssdropbull \ssline 1 1
    \ssdropbull \ssline 1 1 \ssdropbull \ssline 1 1 \ssdropbull \ssline 1 1 \ssdropbull \ssline 1 1
    \ssdropbull \ssline 1 1 \ssdropbull \ssline 1 1 \ssdropbull \ssline 1 1 \ssdropbull \ssline 1 1
    \ssdropbull \ssline 1 1 \ssdropbull \ssline 1 1 \ssdropbull \ssline 1 1 \ssdropbull \ssline 1 1
    \ssdropbull \ssline 1 1 \ssdropbull \ssline 1 1 \ssdropbull \ssline 1 1 \ssdropbull \ssmove {-23}{-23}\ssmove {-2} {0}
  \ssdropbull \ssline 1 1 \ssdropbull \ssline 1 1 \ssdropbull \ssline 1 1 \ssdropbull \ssline 1 1
    \ssdropbull \ssline 1 1 \ssdropbull \ssline 1 1 \ssdropbull \ssline 1 1 \ssdropbull \ssline 1 1
    \ssdropbull \ssline 1 1 \ssdropbull \ssline 1 1 \ssdropbull \ssline 1 1 \ssdropbull \ssline 1 1
    \ssdropbull \ssline 1 1 \ssdropbull \ssline 1 1 \ssdropbull \ssline 1 1 \ssdropbull \ssline 1 1
    \ssdropbull \ssline 1 1 \ssdropbull \ssline 1 1 \ssdropbull \ssline 1 1 \ssdropbull \ssline 1 1
    \ssdropbull \ssline 1 1 \ssdropbull \ssline 1 1 \ssdropbull \ssline 1 1 \ssdropbull \ssmove {-23}{-23}\ssmove {-2} {0}
  \ssdropbull \ssline 1 1 \ssdropbull \ssline 1 1 \ssdropbull \ssline 1 1 \ssdropbull \ssline 1 1
    \ssdropbull \ssline 1 1 \ssdropbull \ssline 1 1 \ssdropbull \ssline 1 1 \ssdropbull \ssline 1 1
    \ssdropbull \ssline 1 1 \ssdropbull \ssline 1 1 \ssdropbull \ssline 1 1 \ssdropbull \ssline 1 1
    \ssdropbull \ssline 1 1 \ssdropbull \ssline 1 1 \ssdropbull \ssline 1 1 \ssdropbull \ssline 1 1
    \ssdropbull \ssline 1 1 \ssdropbull \ssline 1 1 \ssdropbull \ssline 1 1 \ssdropbull \ssline 1 1
    \ssdropbull \ssline 1 1 \ssdropbull \ssline 1 1 \ssdropbull \ssline 1 1 \ssdropbull \ssmove {-23}{-23}\ssmove {-2} {0}
 }{$H^*(\mell,\mbb Z/2)$ localized at 2.}

This resolves all the issues raised in (\ref{eq:cases2}) and we find that
\begin{align*}
&H^{0,-24k+4}(\mell,\mbb Z/2)\simeq0 &\quad \textrm{ so }\qquad& H^{1,-24k+4}(\mell)\simeq (Z_{(2)})^k,\\
&H^{4j+5,-24k+4}(\mell,\mbb Z/2)\simeq(\mbb Z/2)^{k} &\quad \textrm{ so }\qquad& H^{4j+5,-24k+4}(\mell)\simeq(\mbb Z/2)^{k-1}\times\mbb Z/8,\\
&H^{4j+3,-24k-4}(\mell,\mbb Z/2)\simeq(\mbb Z/2)^{k+1} &\quad \textrm{ so }\qquad& H^{4j+3,-24k-4}(\mell) \simeq(\mbb Z/2)^{k+1},\\
&H^{4j+3,-24k-12}(\mell,\mbb Z/2)\simeq(\mbb Z/2)^{k+1} &\quad \textrm{ so }\qquad& H^{4j+3,-24k-12}(\mell)\simeq(\mbb Z/2)^{k+1},\\
&H^{4j+2,-24k-14}(\mell,\mbb Z/2)\simeq(\mbb Z/2)^{k+1} &\quad \textrm{ so }\qquad& H^{4j+2,-24k-14}(\mell)\simeq(\mbb Z/2)^{k+1}.
\end{align*}

\subsection{The elliptic spectral sequence}

We first look at the multiplicative structure of $H^*(\mell)$ for a moment. For the bidegree $H^{1,-20}(\mell)\simeq \mbb Z_{(2)}$ the short exact sequence
$$0\longrightarrow\coker^{0,-20} D\longrightarrow H^{1,-20}(\mell)\longrightarrow\ker^{1,-20} D\longrightarrow0,$$
boils down to
$$0\longrightarrow \mbb Z_{(2)}\longrightarrow \mbb Z_{(2)}\longrightarrow \mbb Z/4\longrightarrow 0.$$ So the map $\coker^{0,-20} D\longrightarrow H^{1,-20}(\mell)$ is multiplication by 4 and the generator of $H^{1,-20}(\mell)$ is $\frac14$ times the image of the generator of $\coker^{0,-20} D$. We denote the generator of $\coker^{0,-20} D$ by $c_4^{-1}c_6\Delta^{-1}$, hence we denote the generator of $H^{1,-20}(\mell)$ by $[\frac14c_4^{-1}c_6\Delta^{-1}]$. This generator is sent to the element $h_2\Delta^{-1}$ in $\ker^{1,-20} D$, which we view as a subset of $H^*(\weier)[\Delta^{-1}]$. Hence we denote the generator of $H^{1,-20}(\mell)$ also by $\langle h_2\Delta^{-1}\rangle$.

As before, we will use the notation $[\frac1cxy]$ with $c\in\mbb N$ for an element of $H^*(\mell)$ when $x$ or $y$ is not an element of $H^*(\mell)$ and when $c$ times $[\frac1cxy]$ corresponds to the product $xy$ in the cokernel of $D$. We will use the notation $\langle xy\rangle$ for an element of $H^*(\mell)$ when $x$ or $y$ is not in $H^*(\mell)$ and when this element corresponds to the product $xy$ in the kernel of $D$.

We recall from the beginning of this section that $dh_2^2=4g$ in $H^*(\weier)$, as well as the relations stated in (\ref{eq:c41}) and (\ref{eq:c42}). We remark that the generator of $H^{4,-24}(\mell)$ is $\langle 2g\Delta^{-2}\rangle$ and we compute that $h_2\langle dh_2\Delta^{-2}\rangle=2\langle 2g\Delta^{-2}\rangle$. Let $f$ denote the generator of $H^{5,-20}(\weier)$ (see figure \ref{fig:mell2}). Because $2h_1=0$, we observe that $h_1^4[c_4^{-2}c_6\Delta^{-1}]\in H^{5,-20}(\mell)$ is an element of order 2. Hence we find that $h_1^4[c_4^{-2}c_6\Delta^{-1}]=4f$. Analogously we see that $h_2\langle 2g\Delta^{-2}\rangle\in H^{5,-20}(\mell)$ must be an element of order 4, which implies an $h_2\langle2g\Delta^{-2}\rangle=2f$.

From the $E_5$-sheet of the spectral sequence on there is another interesting extension in the negative degrees. This extension follows from and is very similar to the extension $4\cdot gh_2=gh_1^3$ in the positive degrees. Let us denote the generator $h_1[c_4^{-2}c_6\Delta^{-1}]$ of $H^{2,-26}(\mell)$ by $g'$ (see figure \ref{fig:mell2}). Then the computation $$h_1^3g'\Delta^2= h_1^4[c_4^{-2}c_6\Delta^{-1}]\Delta^2 = 4f\Delta^2=4\langle h_2g\Delta^{-2}\rangle\Delta^2=\langle h_1^3g\Delta^{-2}\rangle\Delta^2=h_1^3g$$ implies that $g'\Delta^2=g$ and $2\cdot g'=\langle2g\Delta^{-1}\rangle$, which is the generator of the class $H^{4,-24}(\mell)$.

Looking at the resulting picture, which is drawn on the next page, we observe a rotational symmetry between the positive and the negative degrees. The multiplicative structure in the neighborhood of $g^k\Delta^l$ for $k+l<0$ looks precisely like the multiplicative structure in the neighborhood of $g^k\Delta^l$ for $k+l\geq0$ rotated $180$ degrees. This phenomenon can also be observed in $H^*(\mell,\mbb Z/2)$.

\clearpage

\begin{figure}
\begin{minipage}{.85\linewidth}
\includepdf[pagecommand={}]{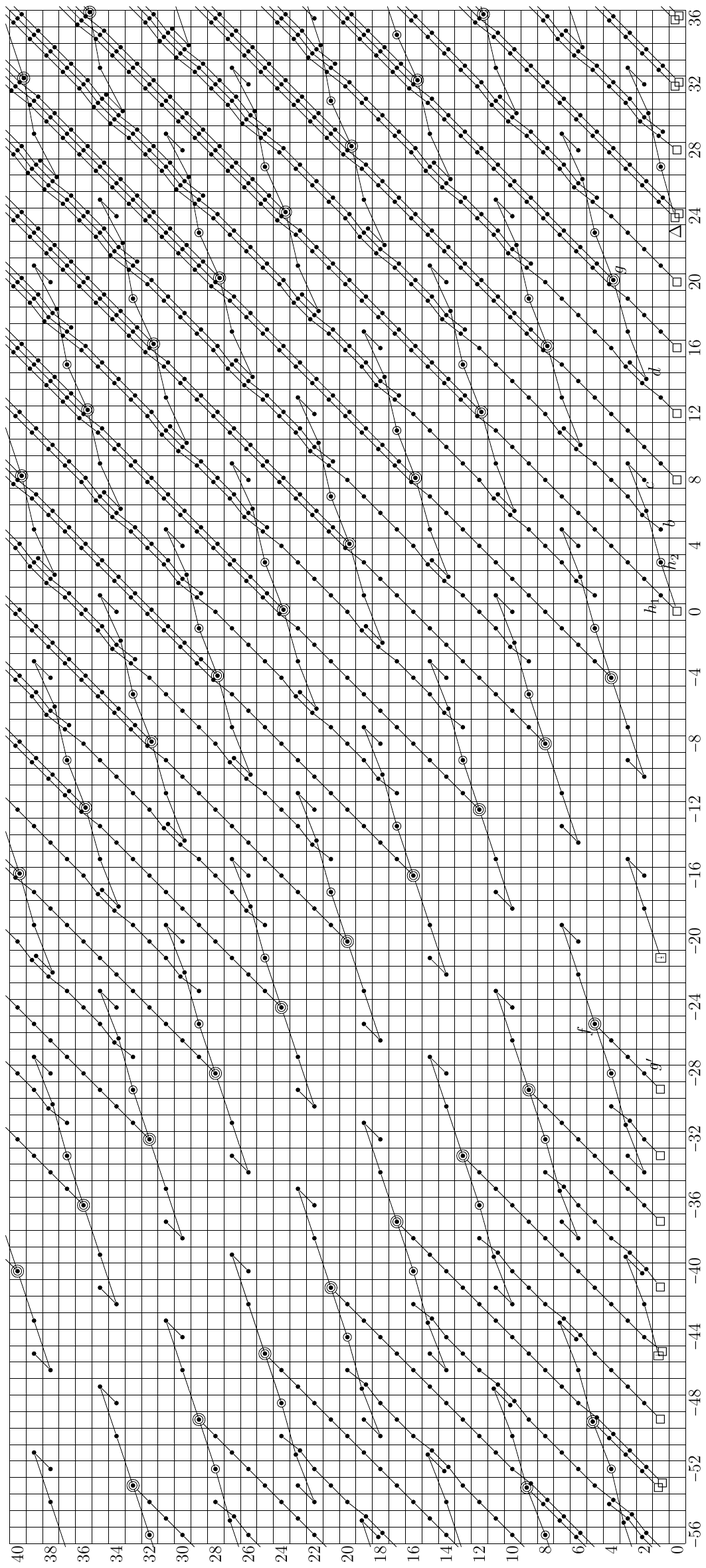}
\end{minipage}\hfill
\begin{minipage}{.15\linewidth}
\rotcaption{$H^*(\mell)$ localized at 2 \label{fig:mell2}}
\end{minipage}
\end{figure}

\clearpage

We now have the $E_2$-sheet of the elliptic spectral sequence (\ref{eq:ess}) which converges to $\pi_*\Tmf$. We would like to compute the differentials in this spectral sequence. We start with the differentials that are calculated by Bauer for the spectral sequence $$H^q\left(\weier,\omega^{\otimes p}\right)\Rightarrow\pi_{2p-q}\tmf.$$ We saw in the introduction that this relates to the spectral sequence of $\pi_*\TMF$ by $$H^q\left(\mathcal M_{ell},\omega^{\otimes p}\right)\simeq H^q\left(\weier,\omega^{\otimes p}\right)[\Delta^{-1}]\Rightarrow\pi_{2p-q}\tmf[\Delta^{-24}]\simeq\pi_{2p-q}\TMF.$$ Hence the differentials in this spectral sequence are multiples of the ones calculated by Bauer by factors $\Delta^{-24k}$. The inclusion of stacks $\mathcal{M}_{ell}\hookrightarrow\mell$ induces a map of cohomologies $H^*\left(\mell\right)\to H^*\left(\mathcal M_{ell}\right)$. This map can be factored as \begin{equation}\label{eq:diff}H^*\left(\mell\right)\stackrel{\rho}{\longrightarrow}H^*\left(\mell\right)[c_4^{-1}]\oplus H^*\left(\mell\right)[\Delta^{-1}]\longtwoheadrightarrow H^*\left(\mell\right)[\Delta^{-1}]\simeq H^*\left(\mathcal M_{ell}\right).\end{equation} 
It induces a map from the spectral sequence for $\pi_*\Tmf$ to the spectral sequence for $\pi_*\TMF$ and we pull back the differentials of this last spectral sequence along this map.

Secondly we calculate the $d_3$-differentials using the Leibniz rule $$d_3(tb)=(d_3t)b+t(d_3b)$$ with $b\in H^{1,6}\left(\mell\right)$ and $t$ an element of the form $h_1^j[c_4^{-k}c_6^l\Delta^{-m}]$. We already know that $d_3b=h_1^4$, so when $d_3(tb)=0$ and $h_1^4t\neq 0$ we find that $(d_3t)b=t(d_3b)=th_1^4\neq0$. Hence $d_3t\neq0$ and there is a $d_3$-differential supported by $t$. Because $bc_6=0$, we know that $d_3(tb)=d_3(0)=0$ for all $t$ which satisfy $l=1$. Hence all $t$ satisfying $l=1$ and $h_1^4t\neq0$ support a differential. This leaves only room for $d_3$-differentials from $h_1^2[c_4^{-1}\Delta^{-m}]$ to $h_1^5[c_4^{-3}c_6\Delta^{-m}]$ for all $m\geq1$. Would this differential exist, then it would imply a $d_3$-differential between $h_1^3[c_4^{-1}\Delta^{-m}]$ and $h_1^6[c_4^{-3}c_6\Delta^{-m}]$. Because the latter differential does not exist, neither does the former.

We have already noted that $f$ and $4f$ are in the same bidegree, as opposed to $h_2g$ and $4h_2g$, and that $g'$ and $2g'$ are in different bidegrees (from the $E_5$-sheet on), as opposed to $g$ and $2g$. This has as a consequence that a $d_n$-differential in the spectral sequence for $\pi_*\TMF$ might constitute, via the map (\ref{eq:diff}), in a $d_{n\pm2}$-differential (for $(n\pm2)\geq5$) in the spectral sequence for $\pi_*\Tmf$. In the spectral sequence for $\pi_*\TMF$ there are no $d_k$-differentials for $k>23$, so in the present case there are no $d_k$-differentials with $k>25$. Hence the spectral sequence collapses at the $E_{26}$-sheet, meaning that $E_\infty$-sheet equals the $E_{26}$-sheet.

Let $e_2$ be the function assigning to a non-zero integer $a$ the highest integer $b$ such that $2^b$ divides $a$, together with the convenient choice of $e_2(0)=3$.  Define the $\mathbb Z$-module $R^+$ as 
\begin{multline}\label{eq:r+}
R^+=\left(\mathbb Z_{(2)}\oplus 2c_6\mathbb Z_{(2)}\oplus h_1\mathbb Z/2\mathbb Z\oplus h_1^2\mathbb Z/2\mathbb Z\right)\left\{c_4^j\Delta^l\right\}\hspace{-2mm}{\tiny{\begin{array}{l}j>0 \\ 7\geq l\geq 0 \end{array}}}\oplus\bigoplus_{l=0}^7 \Delta^l\left(2^{3-e_2(l)}\mathbb Z_{(2)}\oplus2c_6\mathbb Z_{(2)}\right),
\end{multline}
with $h_1^ic_4^jc_6^k\Delta^l$ in degree $i+8j+12k+24l$, of which a few elements are drawn in filtration $i$ in Figure \ref{fig:pos2}. Define $S^+$ as the direct sum of the other groups of degree smaller than 192 drawn in Figure \ref{fig:pos2}. Notice that these are all in filtration 1 till 22. Similarly, define $R^-$ as the $\mathbb Z$-module 
\begin{multline}\label{eq:r-}
R^-=\left(\mathbb Z_{(2)}\oplus 2c_6\mathbb Z_{(2)}\oplus h_1\mathbb Z/2\mathbb Z\oplus h_1^2\mathbb Z/2\mathbb Z\right)\left\{[c_4^j\Delta^{l-1}]\right\}\hspace{-2mm}{\tiny{\begin{array}{l}j<-1 \\ -7\leq l\leq 0 \end{array}}}\oplus\\ \bigoplus_{l=-7}^0[c_4^{-1}\Delta^{l-1}]\left(\mathbb Z_{(2)}\oplus2^{e_2(l)-3}\cdot2c_6\mathbb Z_{(2)}\oplus h_1\mathbb Z/2\mathbb Z\oplus h_1^2\mathbb Z/2\mathbb Z\right),
\end{multline}
with $[h_1^ic_4^jc_6^k\Delta^l]$ in degree $i+8j+12k+24l-1$, of which only the terms $2^{e_2(n)-2}[c_4^{-1}c_6\Delta^{n-1}]$ are drawn in filtration 1 in Figure \ref{fig:neg2}. Define $S^-$ as the direct sum of the other groups of degree bigger then -212 drawn in Figure \ref{fig:neg2}. Notice that these are all in filtration 2 till 23.

\begin{theorem}\label{th:2}
Additively $\pi_*\Tmf_{(2)}$ is given by
$$\pi_*\Tmf_{(2)}=\left(\mathbb Z[\Delta^8]\otimes (R^+\oplus S^+)\right)\oplus\left(\mathbb Z[\Delta^{-8}]\otimes (R^-\oplus S^-)\right),$$
where  $R^+$, $R^-$, $S^+$ and $S^-$ are as defined above, in (\ref{eq:r+}), (\ref{eq:r-})and the subsequent sentences.

The multiplicative structure is generated by 
\renewcommand{\theenumi}{\Roman{enumi}}
\newcounter{rowcount}
\begin{enumerate}
\item multiplications by $c_4$, $c_6$, $2^{3-e_2(l)}\Delta^l$, with $1\leq l\leq8$, and $c_4^3-c_6^2-1728\Delta=0$,
\item multiplications by $h_1$ and $h_2$, which are all drawn in the Figures \ref{fig:pos2} and \ref{fig:neg2}, 
\item multiplications by $c$, $d$, $g$, $\langle h_1\Delta\rangle$, $\langle2h_2\Delta\rangle$, $\langle c\Delta\rangle$, $\langle h_2\Delta^2\rangle$ pulled back from $H^*(\weier)[c_4^{-1},\Delta^{-1}]$ along $D$, together with the extensions\bigskip

\begin{tabular}{@{\stepcounter{rowcount}\makebox[3em][r]{(\arabic{rowcount})}\hspace*{2\tabcolsep}}ll}
$c\langle d\Delta^{8i}\rangle = h_1^2\langle g\Delta^{8i}\rangle$ & for all $i\in\mathbb Z$,\\
$d\langle d\Delta^{8i}\rangle = c\langle g\Delta^{8i}\rangle$ & for all $i\in\mathbb Z$,\\
$d^2\langle d \Delta^{8i}\rangle = h_1^2g\langle g \Delta^{8i}\rangle$ & for all $i\in\mathbb Z$,\\
$c\cdot \langle c\Delta^{8i+1}\rangle =2g\langle g\Delta^{8i}\rangle$ & for all $i\in\mathbb Z$,\\
 $c\cdot \langle 2h_2\Delta^{8i+1}\rangle = h_1d\langle g\Delta^{8i}\rangle$ & for all $i\in\mathbb Z$,\\
 $d\cdot \langle 2h_2\Delta^{8i+1}\rangle = h_1g\langle g\Delta^{8i}\rangle$ & for all $i\in\mathbb Z$,\\
 $c\cdot\langle h_2\Delta^{8i+2}\rangle = \langle h_1dg\Delta^{8i+1}\rangle$ & for all $i\in\mathbb Z$,\\
 $\langle h_1\Delta\rangle^4 = g^5$,&\\
 $\langle 2h_2\Delta\rangle^2 = dg^2$,&\\
 $\langle h_1\Delta\rangle\cdot\langle h_2d\Delta^4\rangle=\langle h_1^2g^2\Delta^4\rangle$.&
\end{tabular}
\end{enumerate}
\end{theorem}

\begin{proof}
\renewcommand{\theenumi}{\Roman{enumi}}
\renewcommand{\theenumii}{\arabic{enumii}}
\begin{enumerate}
\item These multiplications all follow from the map (\ref{eq:diff}). From the $E_5$-sheet the elements of $R^+$ (in filtration 2 or less) cannot support any differentials any more, except for the pure powers of $\Delta$. Because the structure of these terms is quite clear we'll omit them from $E_5^{0,36}$ on. The same is true in the negative degrees for elements of $R^-$.
\item Most of these extensions also follow from the map (\ref{eq:diff}). The few extensions that do not, are derived from an other extension by multiplication by an element of $H^*(\mell)$. These extensions that follow from the multiplicative structure are drawn with a dashed line (see Figure \ref{fig:neg2}).

We have made one amendment to the picture of Tilman Bauer. Define $p$ as the map $p: H^*(\mell)\longrightarrow \pi_*\Tmf$. Then Bauer draws a line between the generators with coordinates $(65,3)$ and $(66,10)$, i.e. $h_2d\Delta^2$ and $h_1^2g^2\Delta$. These elements are representatives of the classes $p(h_2d\Delta^2)$ and $p(h_1^2g^2\Delta)$, so this line would indicate an $\eta$-multiplication between those two classes. There is also an element with coordinates $(65,9)$, namely $h_1g^2\Delta$ of which we already know that $\eta p(h_1g^2\Delta)=p(h_1^2g^2\Delta)$. We also know that $p(h_2d\Delta^2 + h_1g^2\Delta)=p(h_2d\Delta^2)$, so $h_2d\Delta^2$ plus any multiple of $h_1g^2\Delta$ is also a representative of $p(h_2d\Delta^2)$. Hence the drawn multiplication between $(65,3)$ and $(66,10)$ holds no extra information on the level of homotopy groups of $\Tmf$.

Furthermore we fixed one missing multiplication in the positive degrees; from $(122,2)$ to $(125,21)$.
\item 
  \begin{enumerate}
  \item From the $E_5$-sheet of the spectral sequence on, we have that $h_1^3\langle g\Delta^{8i}\rangle = 4h_2\langle g\Delta^{8i}\rangle = h_2^3\langle d\Delta^{8i}\rangle = h_1c\langle d\Delta^{8i}\rangle$. This implies that $h_1^2\langle g\Delta^{8i}\rangle=c\langle d\Delta^{8i}\rangle$ from the $E_5$ sheet on.
  \item In the Bockstein spectral sequence that converges to the $E_2$-sheet, we have the Massey products $\langle \langle d\Delta^{8i}\rangle,2,h_2^2\rangle = h_1\langle g\Delta^{8i}\rangle$ en $\langle 2, h_2^2, c\rangle = h_1d$. The Massey product shuffling lemma gives $$ h_1\langle g\Delta^{8i}\rangle\cdot c=\langle \langle d\Delta^{8i}\rangle,2,h_2^2\rangle c = \langle d\Delta^{8i}\rangle\langle 2,h_2^2,c\rangle = \langle d\Delta^{8i}\rangle dh_1.$$
  \item By the previous two statements, we have $cd^2\langle d\Delta^{8i}\rangle = cd\cdot d\langle d\Delta^{8i}\rangle = h_1^2g\cdot c\langle g\Delta^{8i}\rangle$. 
  \item We recall that, for $i<0$, the element $\langle h_2\Delta^{8i+2}\rangle$ is also denoted by $[\frac14 c_4^{-1}c_6\Delta^{8i+2}]$. Looking at Figures \ref{fig:pos2} and \ref{fig:neg2} we see that $$cg\cdot\langle c\Delta^{8i+1}\rangle = h_1c \langle h_2\Delta^{8i+2}\rangle = h_2^3 \langle h_2\Delta^{8i+2}\rangle = 2g^2\langle g\Delta^{8i}\rangle.$$
  \item Using the Massey product shuffling lemma in the elliptic spectral sequence, we find $$c\cdot\langle 2h_2\Delta^{8i+1}\rangle = \langle \langle g\Delta^{8i}\rangle, h_2,2h_2\rangle c = \langle g\Delta^{8i}\rangle\langle h_2,2h_2,c\rangle = h_1d\langle g\Delta^{8i}\rangle.$$
  \item We derive the statement from $h_1d\cdot \langle 2h_2\Delta^{8i+1}\rangle = cd\langle g\Delta^{8i}\rangle = h_1^2g\langle g\Delta^{8i}\rangle$ by dividing by $h_1$.
  \item This is a direct result from $h_2\cdot\langle h_1c\Delta^{8i+2}\rangle = 2g^2\langle g\Delta^{8i}\rangle = h_1\langle h_1dg\Delta^{8i+1}\rangle$, which is drawn in Figure \ref{fig:pos2}.
  \item From the Massey product shuffling lemma it follows that $$\langle h_1\Delta\rangle^4 = \langle h_1\Delta\rangle^3\cdot\langle h_1\Delta\rangle = \langle h_1\Delta\rangle^3 \langle g,h_2,h_1\rangle = \langle \langle h_1\Delta\rangle^3, g, h_2\rangle h_1 = \langle h_2\Delta^4\rangle h_1.$$ As seen in Figure \ref{fig:pos2}, we know that $\langle h_2\Delta^4\rangle h_1=g^5$.
  \item We compute that $$\langle 2h_2\Delta\rangle^2=\langle 2h_2\Delta\rangle \cdot \langle g, h_2^2, 2\rangle =\langle \langle 2h_2\Delta\rangle, g, h_2^2\rangle 2 = \langle h_2 ^2\Delta^2\rangle 2 = dg^2.$$
  \item This is immediate by the extension $\langle h_1\Delta\rangle\cdot\langle h_2d\Delta^4\rangle = h_2 \langle h_1d\Delta^5\rangle = \langle h_1^2g^2\Delta^4\rangle$, which we find in Figure \ref{fig:pos2}.
  \end{enumerate}
\end{enumerate}
\end{proof}

\clearpage
\begin{figure}
\hspace{-.05\linewidth}
\begin{minipage}{.2\linewidth}
\rotcaption{$\pi_*\Tmf$  localized at 2, the positive degrees.\label{fig:pos2}}
\end{minipage}\hfill
\begin{minipage}{.8\linewidth}
\includepdf[pagecommand={}]{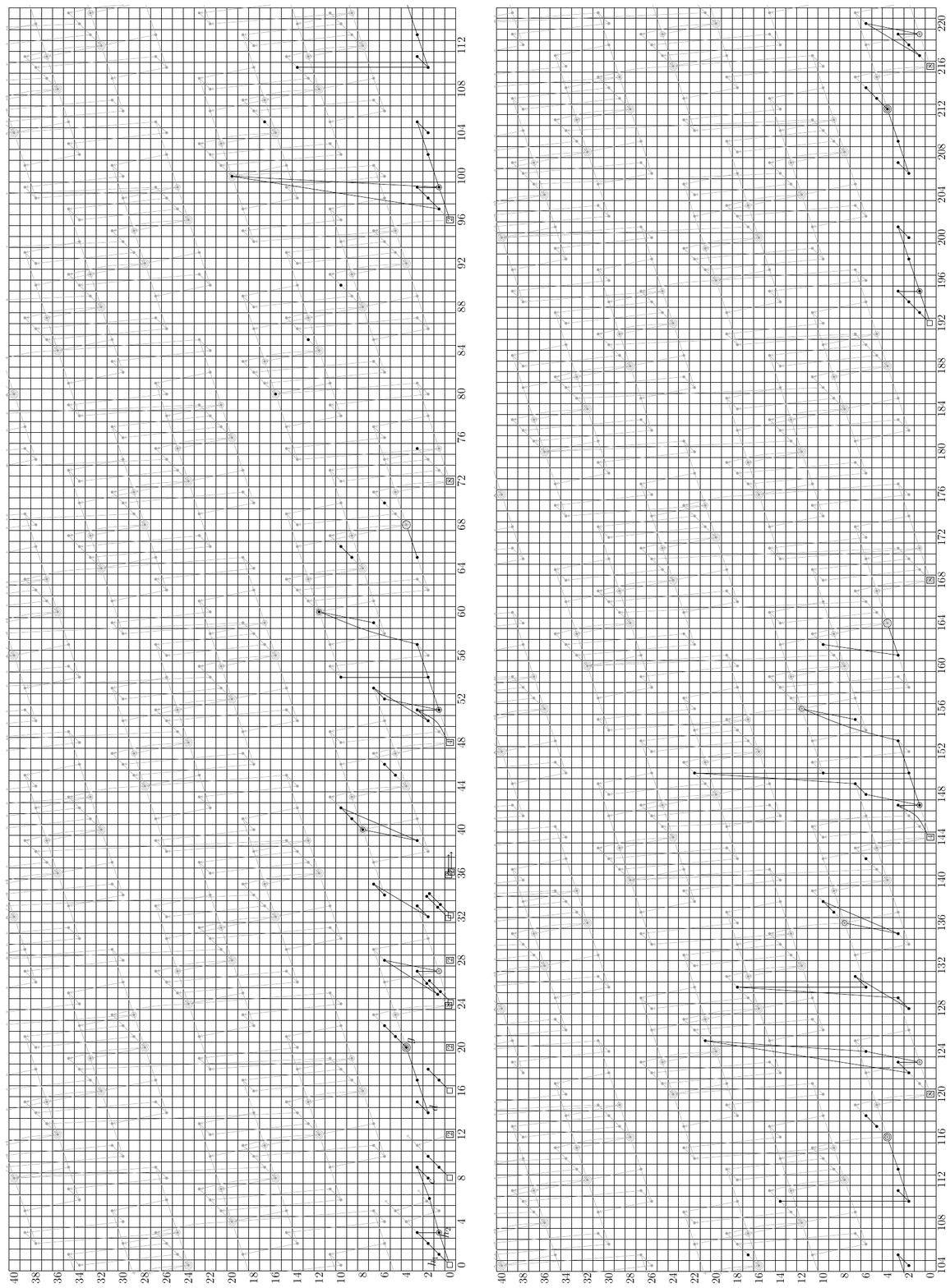}
\end{minipage}
\end{figure}
\clearpage
\begin{figure}
\hspace{-.05\linewidth}
\begin{minipage}{.2\linewidth}
\rotcaption{$\pi_*\Tmf$ localized at 2, the negative degrees.\label{fig:neg2}}
\end{minipage}\hfill
\begin{minipage}{.8\linewidth}
\includepdf[pagecommand={}]{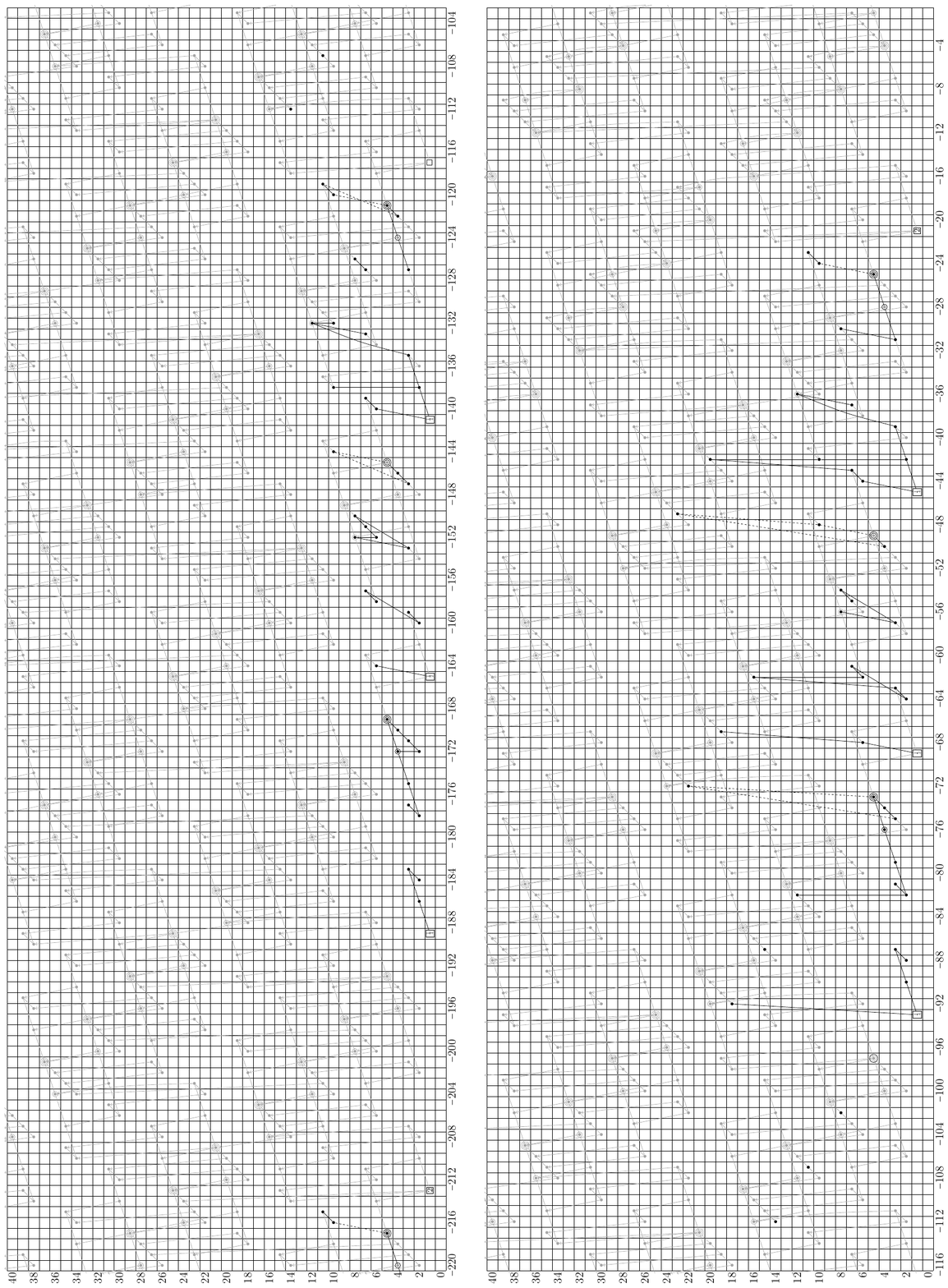}
\end{minipage}
\end{figure}
\clearpage

\bibliographystyle{acm}  
\bibliography{tmf}
\end{document}